\documentclass{amsart}
\usepackage{amsfonts,amsthm,amsmath}
\usepackage{amssymb}
\usepackage{amscd}
\usepackage[utf8]{inputenc}
\usepackage[a4paper]{geometry}
\usepackage{enumerate}
\usepackage[usenames,dvipsnames]{color}
\usepackage{tikz}
\usepackage{soul}
\usepackage{array}
\usepackage{array,tabularx}

\DeclareMathOperator{\Coeff}{Coeff}
\DeclareMathOperator*{\Res}{Res}
\newcommand\note[1]{\mbox{}\marginpar{ \scriptsize\raggedright
\hspace{1pt}\color{red} #1}}

\numberwithin{equation}{section}
\numberwithin{equation}{subsection}

\theoremstyle{plain}
\newtheorem*{theorem*}{Theorem}

\newtheorem{theorem}[equation]{Theorem}
\newtheorem{lemma}[equation]{Lemma}
\newtheorem{proposition}[equation]{Proposition}

\newtheorem{thm}[equation]{Theorem}

\theoremstyle{definition}
\newtheorem{example}[equation]{Example}
\newtheorem{remark}[equation]{Remark}

\newtheorem{definition}[equation]{Definition}

\newtheorem{bekezdes}[equation]{}

\newcommand{\fr}{\mathfrak{r}}

\def\C{\mathbb C}
\def\Q{\mathbb Q}

\def\Z{\mathbb Z}

\def\im{{\rm Im}}

\newcommand{\calv}{{\mathcal V}}

\newcommand{\calm}{{\mathcal M}}
\newcommand{\calt}{{\mathcal T}}

\newcommand{\cali}{{\mathcal I}}

\newcommand{\calO}{{\mathcal O}}

\newcommand{\calS}{{\mathcal S}}

\newcommand{\calL}{\mathcal{L}}

\newcommand{\tX}{\widetilde{X}}
\newcommand{\mfl}{\mathfrak{L}}

\newcommand{\cX}{{\mathcal X}}
\newcommand{\cO}{{\mathcal O}}
\newcommand{\bP}{{\mathbb P}}
\newcommand*{\linebundle}{\mathcal{L}}
\newcommand{\bC}{{\mathbb C}}
\newcommand{\cF}{{\mathcal F}}

\newcommand{\eca}{{\rm ECa}}
\newcommand{\pic}{{\rm Pic}}

\newcommand{\bt}{{\mathbf t}}

\newcommand{\bZ}{{\mathbb{Z}}}
\newcommand{\bQ}{{\mathbb{Q}}}

\author{J\'anos Nagy}
\address{Rényi Alfréd Institute of Mathematics,  Budapest, Hungary}
\email{janomo4@gmail.com}

\title{Hyperelliptic involutions on generic normal surface singularities}

\begin{document}



\begin{abstract}

In the classical case of irreducible smooth algebraic curves every genus $2$ curve is hyperelliptic, or in other words there is a complete linear series
$g_2^1$ on them.
On the other hand if $g > 2$, then a generic smooth curve of genus $2$ is nonhyperelliptic.

In this article we investigate the situation of normal surface singularities, so we fix a resolution graph $\mathcal{T}$  and a generic singularity with resolution $\tX$
corresponding to it in the sense of \cite{NNII}. 
We consider an integer effective cycle $Z$ on the resolution $\tX$ and investigate the existence of a complete linear series $g_2^1$ on it.
The article has the main motivation that we will use heavily the results in it to compute the class of the image varieties of Abel maps
in a following manuscript.
\end{abstract}

\maketitle

\linespread{1.2}

%\date{}

\pagestyle{myheadings} \markboth{{\normalsize  J. Nagy}} {{\normalsize Line bundles}}

%\setcounter{tocdepth}{1}
%\tableofcontents

\section{Introduction}

In the classical case of irreducible smooth algebraic curves every genus $2$ curve is hyperelliptic, or in other words there is a complete linear series
$g_2^1$ on them. On the other hand if $g > 2$, then a generic smooth curve of genus $2$ is not hyperelliptic.

In this article we investigate the situation of normal surface singularities, so we fix a resolution graph $\mathcal{T}$  and a generic singularity with resolution $\tX$
corresponding to it in the sense of \cite{NNII}. Let the reduced exceptional divisor of $\tX$ be $E$. We consider an integer effective cycle $Z \geq E$ on the resolution $\tX$ and investigate the existence of a complete linear series $g_2^1$ on it, we prove the following two main theorems:

\begin{theorem*}\textbf{A}
Suppose that we have an arbitrary rational homology sphere resolution graph $\mathcal{T}$ and a generic resolution $\tX$ corresponding to it, and an effective integer cycle $Z \geq E$, such that $H^0(\calO_Z(K+Z))_{reg} \neq \emptyset$. 
Let us have two vertices of the resolution graph $u', u''$. 
Suppose that $Z_{u'} = Z_{u''} = 1$ and assume that $h^1(\calO_Z) - h^1(\calO_{Z - Z_{u'} \cdot E_{u'}}) \geq 3$.

With these conditions, for every line bundle $\calL \in \im(c^{-E_{u'}^* - E_{u''}^*}(Z))$ one has $h^0(Z, \calL) = 1$.
\end{theorem*}

\begin{remark}
1) The condition $H^0(\calO_Z(K+Z))_{reg} \neq \emptyset$ is natural in the sense that  if $H^0(\calO_Z(K+Z))_{reg} = \emptyset$ 
and $0 \leq Z' < Z$ is the cohomological cycle of $Z$ ($Z'$ is the least cycle, such that $H^0(\calO_Z(K+Z)) = H^0(\calO_{Z'}(K+Z'))$) and if $Z'_{u'} = Z'_{u''} = 1$ holds
then for every line bundle $\calL \in \im(c^{-E_{u'}^* - E_{u''}^*}(Z))$ one has $h^0(Z, \calL) = h^0(Z', \calL| Z')$ and $h^1(Z, \calL) = h^1(Z', \calL| Z')$.

2) The condition $h^1(\calO_Z) - h^1(\calO_{Z - Z_{u'} \cdot E_{u'}}) \geq 3$ is equivalent to the condition $e_Z(u') \geq 3$ in the text since $e_Z(u') = h^1(\calO_Z) - h^1(\calO_{Z - Z_{u'} \cdot E_{u'}})$.
We use here this seemingly more complicated version, since we explain the terminology of $e_Z(u')$ later.
\end{remark}

\begin{theorem*}\textbf{B}
Let us have a rational homology sphere resolution graph $\mathcal{T}$ and a generic resolution $\tX$ corresponding to it, and an effective integer cycle $Z \geq E$, such that 
$H^0(\calO_Z(K+Z))_{reg} \neq \emptyset$.

Let us have a vertex of the resolution graph $u$.
Suppose that $Z_{u} = 1$ and assume that $h^1(\calO_Z) - h^1(\calO_{Z - Z_{u} \cdot E_{u}}) \geq 3$.

With these conditions for every line bundle $\calL \in \im(c^{-2E_{u}^*}(Z))$ one has $h^0(Z, \calL) = 1$.
\end{theorem*}

We also show that the seemingly unnatural conditions in the two theorems $h^1(\calO_Z) - h^1(\calO_{Z - Z_{u'} \cdot E_{u'}}) \geq 3$ and $h^1(\calO_Z) - h^1(\calO_{Z - Z_{u} \cdot E_{u}}) \geq 3$ are crucial.

More explicitly we show the following:

\begin{theorem*}\textbf{C}
 Let us have a rational homology sphere resolution graph $\mathcal{T}$ and a generic resolution $\tX$ corresponding to it.
Consider an effective integer cycle $Z \geq E$ and two vertices $u', u'' \in |Z|$ such that 
$H^0(\calO_Z(K+Z))_{reg} \neq \emptyset$, $Z_{u'} = Z_{u''} = 1$.

Suppose futhermore that
$$0 < h^1(\calO_Z) - h^1(\calO_{Z - Z_{u'} \cdot E_{u'}})  = h^1(\calO_Z) - h^1(\calO_{Z - Z_{u''} \cdot E_{u'}}) = h^1(\calO_Z) - h^1(\calO_{Z - Z_{u''} \cdot E_{u'} - Z_{u''} \cdot E_{u'} })  \leq 2.$$

With these conditions there exists a line bundle $\calL \in \im(c^{-E_{u'}^* - E_{u''}^*}(Z))$, such that $h^0(Z, \calL) = 2$.
\end{theorem*}

\medskip

In the following we wish to give the motivation for these results (since this is only a motivational part we do not define everything here, for the corresponding definitions see the pleriminary section):

\bigskip

In the classical case of a genus $g$ smooth curve $C$, Brill-Noether theory investigates the structure of the Brill-Noether strata $W_d^r = \{\calL \in \pic^{d}(C) \ |\  h^0(C, \calL) \geq r+1\}$ 
among the generic curves or among other special families of algebraic curves. The classical Brill-Noether's theorem \cite{ACGH} states that for a genus $g$ generic curve $C$ and numbers $d \geq 1$ the variety $W_d^r$ is nonempty if and only if $g \geq (r+1)(g-d+r)$.  
This means that the maximal $h^0$ of a line bundle in $\pic^{d}(C)$ is $r+1$, where $r$ is the maximal integer such that $g \geq (r+1)(g-d+r)$. 

Also if $C$ is a generic genus $g$ smooth curve then if $d \geq 1$ and $r \geq 0$ and $g \geq (r+1)(g-d+r)$, then $\dim(W_d^r) = g - (r+1)(g-d+r)$. 

It means that the dimension of Brill-Noether stratas is known if the curve $C$ is generic.

In \cite{NNAD} the author and A. Némethi studied the image varieties of Abel maps in the corresponding Picard groups focusing mostly on the dimension of these varieties.
In \cite{NNAD} these dimensions are computed algorithmically from analytic invariants of the singularity, like cohomology numbers of cycles, giving also explicit combinatorial formulae from the resolution graph, when the analytic type is generic.

The reason of our interest in these image varieties is that these are irreducible components of Brill-Noether stratas in the corresponding Picard groups with the value of $h^1$ equal to their
codimension (see \cite{NNI}). In the case of normal surface singularities these dimensions are already interesting invariants which vary if we move the analytic type of the singularity.

In a following manuscript we compute the class of the image varieties of Abel maps in the case of generic singularities, where we will heavily use the results of this paper.

\medskip

Let us sketch briefly what is the significance of Theorem\textbf{A} in the computation of the class of images of Abel maps:

Let us have a rational homology sphere resolution graph $\mathcal{T}$ and a generic resolution $\tX$ corresponding to it, and an effective integer cycle $Z \geq E$, such that 
$H^0(\calO_Z(K+Z))_{reg} \neq \emptyset$.

Assume that we have a Chern class $l'$ such that $(l', E_v) \geq 0$ for every vertices $v \in \calv$ (that is $-l'$ is in the Lipman cone). Assume furthermore that for every cycle 
$0 \leq Z' \leq Z$  we have $h^1(\calO_{Z'}) - \dim(\im(c^{l'}(Z'))) < h^1(\calO_{Z}) - \dim(\im(c^{l'}(Z)))$ for the Abel maps $c^{l'}(Z'), c^{l'}(Z)$ (This means in other words that the
image variety of the Abel map $c^{l'}(Z)$ is not the direct product of an affine space and the image variety of an Abel map $c^{l'}(Z')$ corresponding to a smaller cycle $Z'$).

Let us look at the divisors of the line bundle $\calO_Z(K+Z)$, for a vertex $u \in \calv$ let us write $t_u = (K + Z, Z) \geq 0$.

We will prove that the line bundle $\calO_Z(K+Z)$ has no got fixed components, has not got base points at the intersection points of exceptional divisors and has not got base points
on exceptional divisors $E_u$, where $(l', E_u) > 0$.

Furthermore let us have a generic section $s \in H^0(\calO_Z(K+Z))_{reg}$ such that the divisor $D = |S|$ consists of disjoint smooth transversal cuts to the exceptional divisors,
let us write $D  = \sum_{u \in \calv, 1 \leq i \leq t_u} D_{u, i}$. Let us fix a small open neigborhood $s \in U \subset H^0(\calO_Z(K+Z))_{reg}$ and for $t \in U$ let us write similarly $D  = \sum_{u \in \calv, 1 \leq i \leq t_u} D_{u, i, t}$.

Let us count how many different ways  we can choose indices $1 \leq j_{u, 1}, \cdots, \leq j_{u, (l', E_u)} \leq t_u$ for the vertices $u \in \calv | (l', E_u) > 0$ such that
$\sum_{u \in |l'|^*, 1 \leq i \leq (l', E_u)} D_{u, j_{u, i}, t}$ is a generic divisor in $\eca^{l'}(Z)$ if $t$ is a generic element in $U$.

Notice that this number is at most $\prod_{u \in |l'|^*} {t_u \choose (l', E_u)}$, we will show that this number is exactly the class of the image variety of the Abel map $\im(c^{l'}(Z'))$.

We will actually prove that in fact there is equality here, so for every choice of indices $1 \leq j_{u, 1}, \cdots, \leq j_{u, (l', E_u)} \leq t_u$ $\sum_{u \in |l'|^*, 1 \leq i \leq (l', E_u)} D_{u, j_{u, i}, t}$ is a generic divisor in $\eca^{l'}(Z)$ if $t$ is a generic element in $U$.

It turns out that the easiest obstacle for this statement would be a suitable counterexapmle for Theorem\textbf{A}.

Indeed assume that there are two vertices $u, v \in  |l'|^*$ such that $Z_u = Z_v = 1$ and there is a line bundle $\calL \in \pic^{-E_u^* - E_v^*}(Z)$ such that $h^0(Z, \calL) = 2$.

We will show in Lemma\ref{differentialpos} that it gives an open subset $U' \subset E_u$ and an injective dominant map $f: E_u \to E_v$ such that if $s \in H^0(\calO_Z(K+Z))_{reg}$,
then for some $x \in U$ we have $x \in |s|$ if and only if $f(x) \in |s|$.

On the other hand it would definitely mean that there are indices $1 \leq i \leq t_u$, $1 \leq j \leq t_v$ such that $D_{u, i, t} + D_{v, j, t}$ is not a generic divisor in $\eca^{-E_u^* - E_v^*}(Z)$ if $t$ is a generic element in $U$ and then our statement about the class of the image variety of the Abel map $\im(c^{l'}(Z'))$ could not hold.

\bigskip

The structure of the paper will be the following:

In section 2) we summarise the necessary background on normal surface singularities. In section 3) we recall the necessary definitions and results about effective Cartier divisors and Abel maps from \cite{NNI}.  In section 4) we recall our working definition about generic normal surface singularities and the main cohomological results from \cite{NNII}. In section 5) we recall from \cite{R} the results about relatively generic analytic structures on normal surface singularities.  In section 6) we prove some results about base points of canonical line bundles on cycles of generic analytic type.  In section 7) we give an outline of the proof of Theorem\textbf{A}, including the main ideas without going into the technical details.  In section 8), section 9) and section 10) we prove Theorem\textbf{A}, Theorem\textbf{B} and Theorem\textbf{C}.

\section{Preliminaries}\label{s:prel}

\subsection{The resolution}\label{ss:notation}
Let $(X,o)$ be the germ of a complex analytic normal surface singularity,
 and let us fix  a good resolution  $\phi:\widetilde{X}\to X$ of $(X,o)$.
We denote the exceptional curve $\phi^{-1}(0)$ by $E$, and let $\{E_v\}_{v\in\calv}$ be
its irreducible components. Set also $E_I:=\sum_{v\in I}E_v$ for any subset $I\subset \calv$.
For the cycle $l=\sum n_vE_v$ let its support be $|l|=\cup_{n_v\not=0}E_v$.
Mixing the two notations we will use $E_{|l|} = \sum_{v\in |l|}E_v$ for an arbitrary cycle $l$.
For more details see \cite{Nfive}.

\subsection{Topological invariants}\label{ss:topol}
Let $\calt$ be the dual resolution graph
associated with $\phi$;  it  is a connected graph.
Then $M:=\partial \widetilde{X}$ can be identified with the link of $(X,o)$, it is 
an oriented  plumbed 3--manifold associated with $\calt$.
%Then $\widetilde{X}$, as a smooth manifold,
%serves as the plumbing  4--manifold associated with $\calt$,
%and  is the plumbed
%3--manifold (and also the `link' of $(X,o)$).
%A resolution is minimal if there is no rational $E_v$ with $E_v^2=-1$.
We will assume that  \emph{$M$ is a rational homology sphere},
or, equivalently,  $\mathcal{T}$ is a tree and all genus
decorations of $\mathcal{T}$ are zero. We use the same
notation $\mathcal{V}$ for the set of vertices of $\calt$, and let $\delta_v$ be the valency of a vertex $v$.
%, and $\mathcal{N}$ for the set of nodes, i.e. vertices with $\delta_v\geq 3$.
%Let $\widetilde{X}$ be the plumbed 4--manifold associated with
%$\mathcal{T}$, hence $\partial \widetilde{X} = M$.

$L:=H_2(\widetilde{X},\mathbb{Z})$, endowed
with the negative definite intersection form  $I=(\,,\,)$, is a lattice. It is
freely generated by the classes of 2--spheres $\{E_v\}_{v\in\mathcal{V}}$. The elements $l=\sum n_vE_v \in L$ are called \emph{cycles} and we define their $E$-support by $|l|=\cup_{n_v\not=0}E_v$. 
 Then $L':= H^2(\widetilde{X},\mathbb{Z})$ is generated
by the (anti)dual classes $\{E^*_v\}_{v\in\mathcal{V}}$ defined
by $(E^{*}_{v},E_{w})=-\delta_{vw}$, the opposite of the Kronecker symbol.
The intersection form embeds $L$ into $L'$. Then $H_1(M,\mathbb{Z})\simeq L'/L$, abridged by $H$.
Usually one also identifies $L'$ with those rational cycles $l'\in L\otimes \Q$ for which
$(l',L)\in\Z$, or, $L'={\rm Hom}_\Z(L,\Z)$. 

For $l'_1,l'_2\in L\otimes \Q$ with $l'_i=\sum_v l'_{iv}E_v$ ($i=\{1,2\}$)
one considers a partial ordering $l'_1\geq l'_2$ defined coordinatewise by $l'_{1v}\geq l'_{2v}$
for all $v\in\calv$. In particular,
$l'$ is an effective rational cycle if $l'\geq 0$.
% We set also $\min\{\ell'_1,\ell'_2\}:= \sum_v\min\{l'_{1v},l'_{2v}\}E_v$ and
% analogously $\min\{F\}$ for a finite subset $F\subset L\otimes \Q$.

Each class $h\in H=L'/L$ has a unique representative $r_h=\sum_vr_vE_v\in L'$ in the semi-open cube
(i.e. each $r_v\in \bQ\cap [0,1)$), such that its class  $[r_h]$ is $h$.
 %and denote the class of $x\in L'$ in $H$ by $[x]$.

All the $E_v$--coordinates of any $E^*_u$ are strict positive.
We define the Lipman cone as $\calS':=\{l'\in L'\,:\, (l', E_v)\leq 0 \ \mbox{for all $v$}\}$.
It is generated over $\bZ_{\geq 0}$ by $\{E^*_v\}_v$. We will also introduce the notation $\calS:=\calS'\cap L$.

For more details regarding the above combinatorial package associated with the topology of normal surface singularities we refer to \cite{Nfive}.
%We also write $\calS:=\calS'\cap L$.

\subsection{Analytic invariants}\label{ss:analinv}
\subsubsection{} The group ${\rm Pic}(\widetilde{X})$
of  isomorphism classes of analytic line bundles on $\widetilde{X}$ appears in the (exponential) exact sequence
\begin{equation}\label{eq:PIC}
0\to {\rm Pic}^0(\widetilde{X})\to {\rm Pic}(\widetilde{X})\stackrel{c_1}
{\longrightarrow} L'\to 0, \end{equation}
where  $c_1$ denotes the first Chern class. Here
$ {\rm Pic}^0(\widetilde{X})=H^1(\widetilde{X},\calO_{\widetilde{X}})\simeq
\C^{p_g}$, where $p_g$ is the {\it geometric genus} of
$(X,0)$. $(X,0)$ is called {\it rational} if $p_g(X,0)=0$.
The works of Artin \cite{Artin62,Artin66} characterized rational singularities topologically
via the graphs; such graphs are called `rational'. By this criterion, $\calt$
is rational if and only if $\chi(l)\geq 1$ for any effective non--zero cycle $l\in L_{>0}$.
Here $\chi(l)=-(l,l-Z_K)/2$ is the Riemann-Roch function and $Z_K\in L'$ is the (anti)canonical cycle
identified by adjunction formulae
$(-Z_K+E_v,E_v)+2=0$ for all $v$.

The epimorphism
$c_1$ admits a unique group homomorphism section $l'\mapsto s(l')\in {\rm Pic}(\widetilde{X})$,
 which extends the natural
section $l\mapsto \calO_{\widetilde{X}}(l)$ valid for integral cycles $l\in L$, and
such that $c_1(s(l'))=l'$  \cite{OkumaRat}.
% We write  $\calO_{\widetilde{X}}(l')$ for $s(l')$, and
We call $s(l')$ the  {\it natural line bundles} on $\widetilde{X}$ and they will be denoted by $\calO_{\tX}(l')$. 
By  the very  definition, $\calL$ is natural if and only if some power $\calL^{\otimes n}$
of it has the form $\calO_{\tX}(l)$ for some $l\in L$.

\subsubsection{$\mathbf{{Pic}(Z)}$}
Similarly, if $Z\in L_{>0}$ is a non--zero effective integral cycle such that its support is $|Z| =E$,
and $\calO_Z^*$ denotes
the sheaf of units of $\calO_Z$, then ${\rm Pic}(Z)=H^1(Z,\calO_Z^*)$ is  the group of isomorphism classes
of invertible sheaves on $Z$. It appears in the exact sequence
  \begin{equation}\label{eq:PICZ}
0\to {\rm Pic}^0(Z)\to {\rm Pic}(Z)\stackrel{c_1}
{\longrightarrow} L'\to 0, \end{equation}
where ${\rm Pic}^0(Z)=H^1(Z,\calO_Z)$.
If $Z_2\geq Z_1$ then there are natural restriction maps,
${\rm Pic}(\widetilde{X})\to {\rm Pic}(Z_2)\to {\rm Pic}(Z_1)$.
Similar restrictions are defined at  ${\rm Pic}^0$ level too.
These restrictions are homomorphisms of the exact sequences  (\ref{eq:PIC}) and (\ref{eq:PICZ}).

Furthermore, we define a section of (\ref{eq:PICZ}) by
$s_Z(l'):=
%r(s(l'))=
{\mathcal O}_{\widetilde{X}}(l')|_{Z}$.
It also satisfies $c_1\circ s_Z={\rm id}_{L'}$. We write  ${\mathcal O}_{Z}(l')$ for $s_Z(l')$, and they are called 
 {\it natural line bundles } on $Z$.

We also use the notations ${\rm Pic}^{l'}(\widetilde{X}):=c_1^{-1}(l')
\subset {\rm Pic}(\widetilde{X})$ and
${\rm Pic}^{l'}(Z):=c_1^{-1}(l')\subset{\rm Pic}(Z)$
respectively. Multiplication by $\calO_{\widetilde{X}}(-l')$, or by
$\calO_Z(-l')$, provides natural affine--space isomorphisms
${\rm Pic}^{l'}(\widetilde{X})\to {\rm Pic}^0(\widetilde{X})$ and
${\rm Pic}^{l'}(Z)\to {\rm Pic}^0(Z)$.

\bekezdes\label{bek:restrnlb} {\bf Restricted natural line bundles.}
%We wish to bring to the attention of the reader the following warning.
The following warning is appropriate.
Note that if $\tX_1$ is a connected small convenient  neighbourhood
of the union of some of the exceptional divisors (hence $\tX_1$ also stays as the resolution
of the singularity obtained by contraction of that union of exceptional  curves), then one can repeat the definition of
natural line bundles at the level of $\tX_1$ as well.

 However, the restriction to $\tX_1$ of a natural line bundle of $\tX$ (even of type
$\calO_{\tX}(l)$ with $l$ integral cycle supported on $E$)  usually is not natural on $\tX_1$:
$\calO_{\tX}(l')|_{\tX_1}\not= \calO_{\tX_1}(R(l'))$
 (where $R:H^2(\tX,\Z)\to H^2(\tX_1,\Z)$ is the natural cohomological 
 restriction), though their Chern classes coincide.

Therefore, in inductive procedure when such restriction is needed,
 we will deal with the family of {\it restricted natural line bundles}. This means the following.
If we have two resolution spaces $\tX_1 \subset \tX$ with resolution graphs $\mathcal{T}_1 \subset \mathcal{T}$ and we have a Chern class $l' \in L'$, then we denote 
$\calO_{\tX_1}(l') = \calO_{\tX}(l') | \tX_1$.

Similarly if $Z$ is an effective integer cycle on $\tX$ with maybe $|Z| \neq E$, then we denote $\calO_{Z}(l') = \calO_{\tX}(l') | Z$. Furthermore if $\calL$ is a line bundle on $\tX_1$, then we denote $\calL(l') = \calL \otimes \calO_{\tX}(l')$.
Similarly if $Z$ is  an effective integer cycle on $\tX$ and $\calL$ is a line bundle on $Z$, then we denote $\calL(l') = \calL \otimes \calO_Z(l')$.

Though the next statement is elementary from \cite{NNII}, it is a key ingredient in several arguments:

\begin{lemma}\label{lem:resNat}\cite{NNII}
With the above notations, the line bundle  $\calO_{\tX_1}(l') \in \pic(\tX_1)$ depends only on its Chern class $l'$ and on the
(non--compact) divisor $E_{top}\cap \tX_1$ of $\tX_1$ and it doesn't depend on the analytic type of the large singularirty $\tX$.
\end{lemma}

\bekezdes \label{bek:ansemgr} {\bf The analytic semigroups.} \
By definition, the analytic semigroup associated with the resolution $\tX$ is

\begin{equation}\label{eq:ansemgr}
\calS'_{an}:= \{l'\in L' \,:\,\calO_{\tX}(-l')\ \mbox{has no  fixed components}\}.
\end{equation}
It is a subsemigroup of $\calS'$. One also sets $\calS_{an}:=\calS_{an}'\cap L$, a subsemigroup
of $\calS$. In fact, $\calS_{an}$
consists of the restrictions   ${\rm div}_E(f)$ of the divisors
${\rm div}(f\circ \phi)$ to $E$, where $f$ runs over $\calO_{X,o}$. Therefore, if $s_1, s_2\in \calS_{an}$, then
${\rm min}\{s_1,s_2\}\in \calS_{an}$ as well (take the generic linear combination of the corresponding functions).
In particular,  for any $l\in L$, there exists a {\it unique} minimal
$s\in \calS_{an}$ with $s\geq l$.

Similarly, for any $h\in H=L'/L$ set $\calS'_{an,h}:\{l'\in \calS_{an}\,:\, [l']=h\}$.
Then for any  $s'_1, s'_2\in \calS_{an,h}$ one has
${\rm min}\{s'_1,s'_2\}\in \calS_{an,h}$, and
for any $l'\in L'$   there exists a unique minimal
$s'\in \calS_{an,[l']}$ with $s'\geq l'$.

For any $l'\in\calS_{an}'$ there exists an ideal sheaf $\cali(l')$ with 0--dimensional support along $E$ such that
 $H^0(\tX,\calO_{\tX}(-l'))\cdot \calO_{\widetilde{X}}=\calO_{\widetilde{X}}(-l')\cdot \cali(l')$.

The ideal $\cali(l') $ describes the space of base points of the line bundle $\calO_{\tX}(-l')$.

If $l'\in\calS'_{an}$ and  the divisor of a generic global section of $\calO_{\tX}(-l')$ intersects
$E_v$, then $(l',E_v)<0$. In particular, if $p\in E_v$ is a  base point then necessarily $(l',E_v)<0$.

Choose a base point $p$ of $\calO_{\tX}(-l')$, and assume that it is a regular point of $E$, and that $\cali(l')_p$
 in the
local ring $\calO_{\tX,p}$ is of the form $(x^t,y)$, where $x,y$ are some local coordinates at $p$  with $\{x=0\}=E$ (locally),
and $t\geq 1$.
Then we say that $p$ is a {\it $t$--simple base point}, in such cases we write $t=t(p)$. Furthermore, $p$ is called {\it simple}
if it is $t$--simple for some $t\geq 1$.

Let us have a Chern class $l' \in S'_{an}$ and let us have a base point $p \in E_{v, reg}$ of a natural line $\calO_{\tX}(-l')$, which is simple, there is another interpretation of the positive integer $t$, such that $p$ is $t$-simple.

Let us have a generic section in $s \in H^0(\calO_{\tX}(-l'))$ and $D = |s|$, then we know that $D$ has a cut $D'$, which is transversal at the base point $p$.

Let us blow up the exceptional divisor $E_v$ along the cut $D'$ sequentially, so let's blow up first at the point $p$ and let the new exceptional divisor be $E_{v_1}$ and let us denote
the strict transform of the cut $D'$ with the same notation.
Then let us blow up $E_{v_1}$ at the intersection point $E_{v_1} \cap D'$ and let the new exceptional divisor be $E_{v_2}$ and so on.

Let us denote the given resolution at the $i$-th step by $\tX_i$ with the blow up map $b_i : \tX_i \to \tX$ and let us have the natural line bundle $\calL_i = \calO_{\tX_i}(-b_i^*(l') - \sum_{1 \leq j \leq i} j \cdot E_{v_j}) = \calO_{\tX_i}(D_{st})$, where $D_{st}$ is the strict transform of the divisor $D$.

Let $t$ be the minimal number, such that $\calL_t$ hasn't got a base point along the excpetional divisor $E_{v_t}$.
Equivalently $t$ is the maximal integer, such that $H^0(\tX_t, \calL_t) = H^0(\calO_{\tX_t}( - b_t^*(l')))$ and $h^1(\tX_t, \calL_t)  = h^1(\calO_{\tX}(-l')) + t$. 

In this case $p$ is a $t$-simple base point ot the natural line bundle $\calO_{\tX}(-l')$.

\section{Effective Cartier divisors and Abel maps}

  In this section we review some needed material from \cite{NNI}.

We fix a good resolution $\phi:\tX\to X$ of a normal surface singularity,
whose link is a rational homology sphere. 

\subsubsection{} \label{ss:4.1}
Let us fix an effective integral cycle  $Z\in L$, $Z\geq E$. (The restriction $Z\geq E$ is imposed by the
easement of the presentation, everything can be adopted  for $Z>0$).

Let $\eca(Z)$  be the space of effective Cartier (zero dimensional) divisors supported on  $Z$.
Taking the class of a Cartier divisor provides  a map
$c:\eca(Z)\to \pic(Z)$.
Let  $\eca^{l'}(Z)$ be the set of effective Cartier divisors with
Chern class $l'\in L'$, that is,
$\eca^{l'}(Z):=c^{-1}(\pic^{l'}(Z))$.

We consider the restriction of $c$, $c^{l'}:\eca^{l'}(Z)
\to \pic^{l'}(Z)$ too, sometimes still denoted by $c$. 

For any $Z_2\geq Z_1>0$ one has the natural  commutative diagram
\begin{equation}\label{eq:diagr}
\begin{picture}(200,45)(0,0)
\put(50,37){\makebox(0,0)[l]{$
\eca^{l'}(Z_2)\,\longrightarrow \, \pic^{l'}(Z_2)$}}
\put(50,8){\makebox(0,0)[l]{$
\eca^{l'}(Z_1)\,\longrightarrow \, \pic^{l'}(Z_1)$}}
\put(70,22){\makebox(0,0){$\downarrow$}}
\put(135,22){\makebox(0,0){$\downarrow$}}
\end{picture}
\end{equation}

As usual, we say that $\calL\in \pic^{l'}(Z)$ has no fixed components if
\begin{equation}\label{eq:H_0}
H^0(Z,\calL)_{reg}:=H^0(Z,\calL)\setminus \bigcup_v H^0(Z-E_v, \calL(-E_v))
\end{equation}
is non--empty. 
Note that $H^0(Z,\calL)$ is a module over the algebra
$H^0(\calO_Z)$, hence one has a natural action of $H^0(\calO_Z^*)$ on
$H^0(Z, \calL)_{reg}$. This second action is algebraic and free.  Furthermore,
 $\calL\in \pic^{l'}(Z)$ is in the image of $c$ if and only if
$H^0(Z,\calL)_{reg}\not=\emptyset$. In this case, $c^{-1}(\calL)=H^0(Z,\calL)_{reg}/H^0(\calO_Z^*)$.

One verifies that $\eca^{l'}(Z)\not=\emptyset$ if and only if $-l'\in \calS'\setminus \{0\}$. Therefore, it is convenient to modify the definition of $\eca$ in the case $l'=0$: we (re)define $\eca^0(Z)=\{\emptyset\}$,
as the one--element set consisting of the `empty divisor'. We also take $c^0(\emptyset):=\calO_Z$, then we have
\begin{equation}\label{eq:empty}
\eca^{l'}(Z)\not =\emptyset \ \ \Leftrightarrow \ \ l'\in -\calS'.
\end{equation}
If $l'\in -\calS'$  then
  $\eca^{l'}(Z)$ is a smooth variety whose dimension equals with the intersection number $(l',Z)$. Moreover,
if $\calL\in \im (c^{l'}(Z))$ (the image of the map $c^{l'}$)
then  the fiber $c^{-1}(\calL)$
 is a smooth, irreducible quasiprojective variety of  dimension
 \begin{equation}\label{eq:dimfiber}
\dim(c^{-1}(\calL))= h^0(Z,\calL)-h^0(\calO_Z)=%\chi(Z,\calL)-\chi(Z)+h^1(Z,\calL)-h^1(\calO_Z)=
 (l',Z)+h^1(Z,\calL)-h^1(\calO_Z).
 \end{equation}

Let us recall the following statement from \cite{NNI}:

\begin{lemma}\label{injectiveabel}\cite{NNI}
Let us have a singularity $(X, 0)$ and its good resolution $\tX$ with resolution graph $\mathcal{T}$ and an effective integer cycle on it $Z$. 
Suppose that there is a vertex $v \in \calv$ such that $Z_v = 1$, then the Abel map $c^{-E_v^*}(Z) : \eca^{-E_v^*}(Z) \to \pic^{-E_v^*}(Z)$ is injective.
\end{lemma}

We also will use the statement of the next lemma frequently:

\begin{proposition}\cite{NNI}\label{redcycleF}
Let  $\tX$ be a resolution of a singularity $(X, 0)$ with resolution graph $\mathcal{T}$ and consider an $l' \in -S'$ and $Z_1 \leq Z_2$ effective cycles  such that $h^1(\calO_{Z_1}) = h^1(\calO_{Z_2})$.
If $\calL \in \im(c^{l'}(Z_2))$ is a line bundle then we have $h^1(Z_2, \calL) = h^1(Z_1, \calL | Z_1)$.
\end{proposition}

\bekezdes \label{bek:I}
Consider again  a Chern class (or cycle) $l'\in-\calS'$ as above.
The $E^*$--support $|l'|^* = J(l')\subset \calv$ of $l'$ is defined via the identity  $l'=\sum_{v\in J(l')}a_vE^*_v$ with all
$\{a_v\}_{v\in J}$ nonzero. Its role is the following.

Besides the Abel map $c^{l'}(Z)$ one can consider its `multiples' $\{c^{nl'}(Z)\}_{n\geq 1}$ as well. It turns out
(cf. \cite[\S 6]{NNI}) that $n\mapsto \dim \im (c^{nl'}(Z))$
is a non-decreasing sequence, and  $\im (c^{nl'}(Z))$ is an affine subspace for $n\gg 1$, whose dimension $e_Z(l')$ is independent of $n$, and essentially it depends only
on $J(l')$.
We denote the linearization of this affine subspace by $V_Z(J) \subset H^1(\calO_Z)$ or if the cycle $Z \gg 0$, then $ V_{\tX}(J) \subset H^1(\calO_{\tX})$.

Moreover, by \cite[Theorem 6.1.9]{NNI},
\begin{equation*}\label{eq:ezl}
e_Z(l')=h^1(\calO_Z)-h^1(\calO_{Z|_{\calv\setminus J(l')}}),
\end{equation*}
where $Z|_{\calv\setminus J(l')}$ is the restriction of the cycle $Z$ to its $\{E_v\}_{v\in \calv\setminus J(l')}$
coordinates.

If $Z\gg 0$ (i.e. all its $E_v$--coordinates are very large), then (\ref{eq:ezl}) reads as
\begin{equation*}\label{eq:ezlb}
e_Z(l')=h^1(\calO_{\tX})-h^1(\calO_{\tX(\calv\setminus J(l'))}),
\end{equation*}
where $\tX(\calv\setminus J(l'))$ is a convenient small tubular neighbourhood of $\cup_{v\in \calv\setminus J(l')}E_v$.

Let $\Omega _{\tX}(J)$ be the subspace of $H^0(\tX\setminus E, \Omega^2_{\tX})/ H^0(\tX,\Omega_{\tX}^2)$ generated by differential forms which have no poles along $E_J\setminus \cup_{v\not\in J}E_v$.
Then, cf. \cite[\S8]{NNI},
\begin{equation*}\label{eq:ezlc}
h^1(\calO_{\tX(\calv\setminus J)})=\dim \Omega_{\tX}(J).
\end{equation*}

Similarly let $\Omega _{Z}(J)$ be the subspace of $H^0(\calO_{\tX}(K + Z))/ H^0(\calO_{\tX}(K))$ generated by differential forms which have no poles along $E_J\setminus \cup_{v\not\in J}E_v$.
Then, cf. \cite[\S8]{NNI},
\begin{equation*}\label{eq:ezlc}
h^1(\calO_{Z_{(\calv\setminus J)}})=\dim \Omega_{Z}(J).
\end{equation*}

We have also the following duality from \cite{NNI} supporting the equalities above:

\begin{theorem}\cite{NNI}\label{th:DUALVO}
Via Laufer duality one has  $V_{\tX}(J)^*=\Omega_{\tX}(J)$ and $V_{Z}(J)^*=\Omega_Z(J)$.
\end{theorem}

\section{Analytic invariants of generic analytic type}\label{s:AnGen}

For a precise working definition of a generic analytic type see \cite{NNII}, \cite{R}, in a slightly simplified language we can regard the generic analytic structure in the following way as well.

Fix a rational homology sphere resolution graph $\mathcal{T}$, for each $E_v$ ($v\in\calv$) the disc bundle with Euler number $E_v^2$ is taut:
it has no analytic moduli. The generic $\tX$ is obtained by gluing `generically' these bundles according to the edges of $\Gamma$ (as an analytic plumbing).

\subsection{Review of some results of \cite{NNII}}\label{ss:ReviewNNII}

 The list of analytic invariants, associated with a generic analytic type
  (with respect to a fixed resolution graph),
 which in \cite{NNII} are described topologically include  the following ones:
 $h^1(\calO_Z)$, $h^1(\calO_Z(l'))$ (with certain restrictions on the Chern class $l'$),
  --- this last one applied  for $Z\gg 0$ provides  $h^1(\calO_{\tX})$
 and  $h^1(\calO_{\tX}(l')$) too ---,
%the multivariable Hilbert series $\sum_{l\in L}\mathfrak{h}(l)\bt^l$ associated with the divisorial filtration,
the analytic semigroup, %the cohomological cycle
and the maximal ideal cycle of $\tX$.
See above or \cite{CDGPs,CDGEq,Lipman,Nfive,MR}
for the definitions and relationships between them.
The topological characterizations use the Riemann--Roch expression $\chi:L'\to \bQ$.

In the next theorem  the bundles $\calO_{\tX}(-l')$ are the `restricted natural line bundles'
associated with some pair $\tX\subset \tX_{top}$. In particular, it is valid even if $\tX_{top}=\tX$ and the bundles are natural line bundles.
The theorem (and basically several statements
 regarding generic analytic structure and restricted natural line bundles) says that these bundles behave cohomologically
as the generic line bundles in ${\rm Pic}^{-l'}(\tX)$ (for more comments on this phenomenon  see \cite{NNII}).

\begin{theorem}\cite[Theorem A]{NNII}\label{th:OLD}
 Fix a rational homology sphere resolution graph $\mathcal{T}$ (tree of $\bP^1$'s) and let's have a generic analytic type $\tX$ corresponding to it. Then
the following identities hold:\\
(a) For any effective cycle $Z\in L_{>0}$, such that the support $|Z|$ is connected, we have
\begin{equation*}
h^1(\calO_Z) = 1-\min_{0< l \leq Z,l\in L}\{\chi(l)\}.
\end{equation*}
(b) If $l'=\sum_{v\in \calv}l'_vE_v \in L'$ satisfies
$l'_v >0$ for any $E_v$ in the support of $Z$ then
\begin{equation*}
h^1(\calO_Z(-l'))=\chi(l')-\min _{0\leq l\leq Z, l\in L}\, \{\chi(l'+l)\}.
\end{equation*}
%(For a characterization valid for more general Chern classes $l'$ see section \ref{s:appli}.)\\
(c) If $p_g(X,o)=h^1(\tX,\calO_{\tX})$ is the geometric genus of $(X,o)$ then
\begin{equation*}
p_g(X,o)= 1-\min_{l\in L_{>0}}\{\chi(l)\} =-\min_{l\in L}\{\chi(l)\}+\begin{cases}
1 & \mbox{if $(X,o)$ is not rational}, \\
0 & \mbox{else}.
\end{cases}
\end{equation*}
(d) More generally, for any $l'\in L'$
\begin{equation*}
h^1(\calO_{\tX}(-l'))=\chi(l')-\min _{l\in L_{\geq 0}}\, \{\chi(l'+l)\}+
%\epsilon(l'),
%\end{equation*}
%where
%\begin{equation*}
%\epsilon(l')=
\begin{cases}
1 & \mbox{if \ $l'\in L_{\leq 0}$ and $(X,o)$ is not rational}, \\
0 & \mbox{else}.
\end{cases}
\end{equation*}
(e) For $l\in L$ set $\mathfrak{h}(l)=\dim ( H^0(\tX, \calO_{\tX})/ H^0(\tX, \calO_{\tX}(-l)))$.
% the $l$--coefficient  of the %$H(\bt)=\sum _{l\in L} \mathfrak{h}(l)\bt^{l}$
% multivariable Hilbert series associated with the divisorial filtration induced by the exceptional divisors of %$\phi$.
Then $\mathfrak{h}(0)=0$ and for $l_0>0$ one has
%Write $l'$ as $r_h+l_0$ for some $l_0\in L$ and
% $r_h\in L' $ the unique representative of $h=[l']$ in the semi-open cube of $L'$.
%Then  $\mathfrak{h}(r_h)=0$ for $l_0=0$. Furthermore, for $l_0>0$ and $h\not=0$
%\begin{equation*}
%\mathfrak{h}(l_0)=
%\min_{l\in L_{\geq 0}} \{\chi(l_0+l)\}-\min_{l\in L_{\geq 0}} \{\chi(r_h+l)\}.
%\end{equation*}
%For $h=0$ and $l'=l_0>0$
\begin{equation*}
\mathfrak{h}(l_0)=\min_{l\in L_{\geq 0}} \{\chi(l_0+l)\}-\min_{l\in L_{\geq 0}} \{\chi(l)\}+
\begin{cases}
1 & \mbox{if $(X,o)$ is not rational}, \\
0 & \mbox{else}.
\end{cases}
\end{equation*}
%(f)
%Write the multivariable equivariant Poincar\'e series $P(\bt)=-H(\bt)\cdot \prod_{v\in\calv}
%(1-t_v^{-1})$ as $\sum_{l'\in\calS'}\mathfrak{p}(l')\bt^{l'}$.
%It is supported in the Lipman (antinef) cone, in particular in $L'_{\geq 0}$.
% Then
%$\mathfrak{p}(0)=1$ and for $l'>0$ one has
%$$\mathfrak{p}(l')= \sum_{I\subset \calv} (-1)^{|I|+1}\, \min_{ l\in L_{\geq 0}}
%\chi(l'+l+E_I).$$
(f) \ %Consider  the   analytic semigroup
%$\calS'_{an}:= \{l'\in L' \,:\,\calO_{\tX}(l')\ \mbox{has no  fixed components }\}\subset \calS'$. Then
$\calS'_{an}= \{l'\,:\, \chi(l')<
\chi(l' +l) \ \mbox{for any $l\in L_{>0}$}\}\cup\{0\}$.\\
(g)
 Assume that $\Gamma$ is a non--rational graph  and set
$\calm=\{ Z\in L_{>0}\,:\, \chi(Z)=\min _{l\in L}\chi(l)\}$.
Then %the unique minimal element of $\calm$ is the cohomological cycle, while
the unique maximal element of $\calm$ is the maximal ideal cycle of\, $\tX$.

(Note that in the above formulae one also has $\min _{l\in L_{\geq 0}}\{\chi(l)\}=\min _{l\in L}\{\chi(l)\}$.)
\end{theorem}

\section{The relative setup.}

In this section we wish to summarise the results from \cite{R} about relatively generic analytic structures what we need in this article, as always we again deal only with rational homology sphere
resolution graphs.

We consider an effective integer cycle $Z \geq E$ on a resolution $\tX$ with resolution graph $\mathcal{T}$, and a smaller cycle $Z_1 \leq Z$, where we denote $|Z_1| = \calv_1$ and the possibly nonconnected subgraph corresponding to it by $\mathcal{T}_1$.

We have the restriction map $r: \pic(Z)\to \pic(Z_1)$ and one has also the (cohomological) restriction operator
  $R_1 : L'(\mathcal{T}) \to L_1':=L'(\mathcal{T}_1)$
(defined as $R_1(E^*_v(\mathcal{T}))=E^*_v(\mathcal{T}_1)$ if $v\in \calv_1$, and
$R_1(E^*_v(\mathcal{T}))=0$ otherwise).

For any $\calL\in \pic(Z)$ and any $l'\in L'(\mathcal{T})$ it satisfies
\begin{equation*}
c_1(r(\calL))=R_1(c_1(\calL)).
\end{equation*}

In particular, we have the following commutative diagram as well:

\begin{equation*}  %\label{eq:diagr}
\begin{picture}(200,40)(30,0)
\put(50,37){\makebox(0,0)[l]{$
\ \ \eca^{l'}(Z)\ \ \ \ \ \stackrel{c^{l'}(Z)}{\longrightarrow} \ \ \ \pic^{l'}(Z)$}}
\put(50,8){\makebox(0,0)[l]{$
\eca^{R_1(l')}(Z_1)\ \ \stackrel{c^{R_1(l')}(Z_1)}{\longrightarrow} \  \pic^{R_1(l')}(Z_1)$}}
%\put(76,22){\makebox(0,0){{\tiny }$ \downarrow$}}
\put(162,22){\makebox(0,0){$\downarrow \, $\tiny{$r$}}}
\put(78,22){\makebox(0,0){$\downarrow \, $\tiny{$\fr$}}}
\end{picture}
\end{equation*}

We can consider instead of the `total' Abel map $c^{l'}(Z)$ only its restriction above a fixed fiber of $r$.

That is, we fix some  $\mfl\in \pic^{R_1(l')}(Z_1)$, and we study the restriction of $c^{l'}(Z)$ to $(r\circ c^{l'}(Z))^{-1}(\mfl)\to r^{-1}(\mfl)$.

 The subvariety $(r\circ c^{l'}(Z))^{-1}(\mfl)
=(c^{R_1(l')}(Z_1) \circ \fr)^{-1}(\mfl) \subset \eca^{l'}(Z)$ is denoted by $\eca^{l', \mfl}(Z)$.

Let us denote the relative Abel map by $c^{l', \mfl}(Z): \eca^{l', \mfl}(Z) \to r^{-1}(\mfl)$ in the following.

\begin{theorem}\cite{R}\label{relativspace}
Fix an arbitrary singularity $\tX$ a Chern class $l'\in -\calS'$, an integer effective cycle $Z\geq E$ and a cycle $Z_1 \leq Z$ and let's have a line bundle  $\mfl\in \pic^{R(l')}(Z_1)$.
Assume that  $\eca^{l', \mfl}(Z)$ is nonempty, then it is smooth of dimension $h^1(Z_1,\mfl)  - h^1(\calO_{Z_1})+ (l', Z)$ and irreducible.
\end{theorem}

Let's recall from \cite{R} the analouge of the theroems about dominance of Abel maps in the relative setup:

\begin{definition}\cite{R}
Fix an arbitrary singularity $\tX$, a Chern class $l'\in -\calS'$, an integer effective cycle $Z\geq E$, a cycle $Z_1 \leq Z$ and a line bundle $\mfl\in \pic^{R_1(l')}(Z_1)$ as above.
We say that the pair $(l',\mfl ) $ is {\it relative
dominant} on the cycle $Z$, if the closure of $ r^{-1}(\mfl)\cap \im(c^{l'}(Z))$ is $r^{-1}(\mfl)$.
\end{definition}

\begin{theorem}\label{th:dominantrel}\cite{R}
 One has the following facts:

(1) If $(l',\mfl)$ is relative dominant on the cycle $Z$, then $ \eca^{l', \mfl}(Z)$ is
nonempty and $h^1(Z,\calL)= h^1(Z_1,\mfl)$ for any
generic line bundle $\calL\in r^{-1}(\mfl)$.

(2) $(l',\mfl)$ is relative dominant on the cycle $Z$,  if and only if for all
 $0<l\leq Z$, $l\in L$ one has
$$\chi(-l')- h^1(Z_1, \mfl) < \chi(-l'+l)-
 h^1((Z-l)_1, \mfl(-l)).$$, where we denote $(Z-l)_1 = \min(Z-l, Z_1)$.
\end{theorem}

\begin{theorem}\label{th:hegy2rel}\cite{R}
Fix an arbitrary singularity $\tX$, a Chern class $l'\in -\calS'$, an integer effective cycle $Z\geq E$, a cycle $Z_1 \leq Z$ and a line bundle $\mfl\in \pic^{R_1(l')}(Z_1)$ as in Theorem \ref{th:dominantrel}. 

1) For any $\calL\in r^{-1}(\mfl)$ one has
\begin{equation*}\label{eq:genericLrel}
\begin{array}{ll}h^1(Z,\calL)\geq \chi(-l')-
\min_{0\leq l\leq Z,\ l\in L} \{\,\chi(-l'+l) -
h^1((Z-l)_1, \mfl(-l))\, \}.\end{array}\end{equation*}
Furthermore, if $\calL$ is generic in $r^{-1}(\mfl)$
then we have equality.

2)  If $\calL\in r^{-1}(\mfl)$ is a generic line bundle on the image of the relative Abel map $\im(c^{l', \mfl}(Z))$ then one has $$h^1(Z, \calL) = h^1(Z_1, \mfl) + \dim(r^{-1}(\mfl)) - \dim(\im(c^{l', \mfl}(Z))).$$
In particular if the relative Abel map is not dominant then we have:
\begin{equation*}
\dim(r^{-1}(\mfl)) - \dim(\im(c^{l', \mfl}(Z))) > \chi(-l') - h^1(Z_1, \mfl)  - \min_{0\leq l\leq Z,\ l\in L} (\chi(-l'+l) - h^1((Z-l)_1, \mfl(-l))).
\end{equation*}

\end{theorem}

\subsection{Relatively generic analytic structures}

In the following we recall the results from \cite{R} about relatively generic analytic structures:

Let us fix a a topological type, in other words a resolution graph $\mathcal{T}$ with vertex set $\calv$,
we consider a partition $\calv = \calv_1 \cup  \calv_2$.

They define two (not necessarily connected) subgraphs $\mathcal{T}_1$ and $\mathcal{T}_2$.

We call the intersection of an exceptional divisor from $ \calv_1 $ with an exceptional divisor from  $ \calv_2 $ a
{\it contact point}.

For any cycle $Z\in L=L(\mathcal{T})$ we write $Z=Z_1+Z_2$,
where $Z_i\in L(\mathcal{T}_i)$ is
supported in $\mathcal{T}_i$ ($i=1,2$).
Furthermore, parallel to the restrictions
$r_i : \pic(Z)\to \pic(Z_i)$ one also has the (cohomological) restriction operators
  $R_i : L'(\mathcal{T}) \to L_i':=L'(\mathcal{T}_i)$
(defined as $R_i(E^*_v(\mathcal{T}))=E^*_v(\mathcal{T}_i)$ if $v\in \calv_i$, and
$R_i(E^*_v(\mathcal{T}))=0$ otherwise).

\noindent

For any $l'\in L'(\mathcal{T})$ and any $\calL\in \pic^{l'}(Z)$ it satisfies $c_1(r_i(\calL))=R_i(c_1(\calL))$.
In the following for the sake of simplicity we will denote $r = r_1$ and $R = R_1$.

\noindent

Furthermore fix an analytic type $\tX_1$ for the subgraph $\mathcal{T}_1$ (if $\mathcal{T}_1$ is disconnected, then an analytic type for each of its connected components).

For each vertex $v_2 \in \calv_2$ which has got a neighbour $n(v_2)$ in $\calv_1$ we fix a cut $D_{v_2}$ on $\tX_1$, along we glue the exceptional divisor $E_{v_2}$.
This means that $D_{v_2}$ is a divisor, which intersects the exceptional divisor $E_{v_1}$ transversally in one point and we will glue the tubular neighborhood of the exceptional divisor $E_{v_2}$ in a way, such that $E_{v_2} \cap \tX_1 = D_{v_2}$.

If for some vertex $v_2 \in \calv_2$, which has got a neighbour in $\calv_1$ we don't say explicitely what is the fixed cut $D_{v_2}$, then it should be understood in the way that we glue the exceptional divisor $E_{v_2}$ along a generic cut, in more precise terms we fix a generic divisor $ D_{v_2} \in \eca^{-E_{n(v_2)}^*}(\tX_1)$ in this case. 

\noindent

Let us glue the tubular neihgbourhoods of the exceptional divisors $E_{v_2}, v_2 \in \calv_2$  with the above conditions generically to the fixed resolution $\tX_1$.
We get a singularity $\tX$ with resolution graph $\mathcal{T}$ and we say that $\tX$ is a relatively generic singularity corresponding to the analytical structure $\tX_1$ and the cuts $D_{v_2}$. 

For a more precise and rigorous explanation of relative genericity look at \cite{R}, we have the following theorems with this setup from \cite{R}:

\begin{theorem}\cite{R}\label{relgen1}
Let us have the setup as above and let us have furthermore an effective cycle $Z$ on $\tX$ with $Z = Z_1 + Z_2$, where $|Z_1| \subset \calv_1$ and $|Z_2| \subset \calv_2$.

Let us have a natural line bundle $\calL= \calO_{\tX}(l')$ on $\tX$, such that $ l' = - \sum_{v \in \calv} a_v E_v \in L'_{\mathcal{T}}$, with $a_v > 0, v \in \calv_2 \cap |Z|$,  and denote $c_1 (\calL | Z) = l'_ Z \in L'_{|Z|}$, furthermore let us denote $\mfl = \calL | Z_1$, then we have the following:

We have $H^0(Z,\calL)_{reg} \not=\emptyset$ if and only if $(l',\mfl)$ is relative dominant on the cycle $Z$ or equivalently:

\begin{equation*}
\chi(-l')- h^1(Z_1, \mfl) < \chi(-l'+l)-  h^1((Z-l)_1, \mfl(-l)),
\end{equation*}
for all $0 < l \leq Z$.
\end{theorem}

\begin{theorem}\cite{R}\label{relgen2}
Let us have the same setup as in the previous theorem, then we have:

\begin{equation*}
h^1(Z, \calL) =  h^1(Z, \calL_{gen}),                                   
\end{equation*}
where $\calL_{gen}$ is a generic line bundle in $ r^{-1}(\mfl) \subset \pic^{l'_m}(Z)$, or equivalently:

\begin{equation*}
h^1(Z, \calL)= \chi(-l') - \min_{0 \leq l \leq Z}(\chi(-l'+l)-  h^1((Z-l)_1, \mfl(-l))).
\end{equation*}
\end{theorem}

\begin{theorem}\cite{R}\label{relgen3}
Let us have a natural line bundle $\calL= \calO_{\tX}(l')$ on $\tX$, such that $ l' = - \sum_{v \in \calv} a_v E_v$, and assume that  $a_v \neq 0$ if $ v \in \calv_2 \cap |Z|$. Let us denote $c_1 (\calL | Z) = l'_ Z \in  L'_{|Z|}$ and $\mfl = \calL | Z_1$.

Assume that $H^0(Z,\calL)_{reg} \not=\emptyset$, and pick an arbitrary $D \in (c^{l'_Z}(Z))^{-1}\calL \subset \eca^{l'_Z, \mfl}(Z)$, then $ c^{l'_m}(Z) : \eca^{l'_m, \mfl}(Z) \to r^ {-1}(\mfl  )$ is a submersion in $D$, and $h^1(Z,\calL) = h^1(Z_1, \mfl)$.

In particular the map $ c^{l'_Z}(Z) :  \eca^{l'_Z, \mfl}(Z) \to r^ {-1}(\mfl  )$ is dominant, which means that $(l'_m,\mfl)$ is relative dominant on the cycle $Z$, or equivalently:

\begin{equation*}
\chi(-l')- h^1(Z_1, \mfl) < \chi(-l'+l)-  h^1((Z-l)_1, \mfl(-l)),
\end{equation*}
for all $0 < l \leq Z$.
\end{theorem}

\begin{remark}

In the theorems above in any formula one can replace $l'$ with $l'_Z$, since for every $0 \leq l \leq Z$ one has $\chi(-l')- \chi(-l'+l) = \chi(-l'_Z)- \chi(-l'_Z+l) = -(l', l) - \chi(l)$.
\end{remark}

\section{Base points of the line bundle $\calO_Z(K + Z)$ on generic singularities}

In the following we prove a few lemmas about the base points of the line bundle $\calO_Z(K + Z)$, where $Z$ is an integer effective cycle on a generic singularity and 
$H^0(Z, K+Z)_{reg} \neq \emptyset$.

\begin{lemma}\label{baseI}
Assume that $\mathcal{T}$ is a rational homology sphere resolution graph and $\tX$ is a generic singularity corresponding to it, and let us have a cycle $Z$ on it, such that $|Z| = \calv$, and $H^0(\calO_{Z}( Z+K))_{reg} \neq \emptyset$.
The line bundle $\calL_{Z} = \calO_Z(K + Z)$ hasn't got a basepoint at the intersection point of exceptional divisors.
\end{lemma}
\begin{proof}
Let us denote the line bundle $\calL_{Z} = \calO_Z(K + Z)$ and assume that it has got a basepoint at the intersection point of exceptional divisors
$E_u$ and $E_v$.

Let us blow up the intersection point of $E_u$ and $E_v$ and let the new exceptional divisor be $E_{new}$, the new resolution graph $\mathcal{T}_{new}$ and the new resolution
$\tX_{new}$ which is certainly generic corresponding to the resolution graph $\mathcal{T}_{new}$.

Let us denote the cycle $Z_{new} = \pi^*(Z) - E_{new}$, with these notations we should prove that $H^0(Z_{new}, \pi^*(\calL_{Z}))_{reg} \neq \emptyset$.

We know that $ \pi^*(\calL_{Z}) = \calO_{Z_{new}}(Z_{new} + K_{new})$, so we have to prove that $H^0(\calO_{Z_{new}}(Z_{new} + K_{new}))_{reg} \neq \emptyset$.

Equivalently, we have to show that for every effective integer cycle $Z' < Z_{new}$ one has $h^1(\calO_{Z'}) < h^1(\calO_{Z_{new}})$ (indeed by Seere duality we have
$h^1(\calO_{Z'}) = h^0(\calO_{Z'}(Z' + K_{new})$ if $Z' \leq Z_{new}$).

Notice that $h^1(\calO_{Z_{new}}) = h^1(\calO_Z) = 1- \chi(Z)$.

Indeed by \cite{NNII} we have $h^1(\calO_Z) = 1 - \min_{E \leq l \leq Z} \chi(l)$, assume that $h^1(\calO_Z) = 1 - \chi(l)$ for an effective cycle $ 0 \leq l < Z$, then we get $h^1(\calO_l) \geq 1- \chi(l)$ but on the other hand $h^1(\calO_l) \leq h^1(\calO_Z)$, which yields $h^1(\calO_l) = h^1(\calO_Z)$.

However this is impossible by the condition $H^0(\calO_{Z}( Z+K))_{reg} \neq \emptyset$, so we get $h^1(\calO_{Z_{new}}) = h^1(\calO_Z) = 1- \chi(Z)$.
Notice that by the arguement above we also get that for each  cycle $0 \leq l < Z$ one has $\chi(l) > \chi(Z)$.

Assume to the contrary that $Z' < Z_{new}$ and we have $h^1(\calO_{Z'}) = h^1(\calO_{Z_{new}})$.

Let the connected components of $|Z'|$ be $|Z'_1|, \cdots , |Z'_n|$, where $Z' = \sum_{1 \leq i \leq n}Z'_i$ and if $Z' \geq E_{new}$, then $Z'_1 \geq E_{new}$. 
  
We know that for every $1 \leq i \leq n$ one has $h^1(\calO_{Z'_i}) = 1 - \min_{E_{|Z'_i|} \leq l \leq Z'_i} \chi(l)$ and assume that $E_{|Z'_i|} \leq l_i \leq Z'_i$, such that 
$h^1(\calO_{Z'_i}) = 1 - \chi(l_i) = h^1(\calO_{l_i})$.

Assume first that $l_1 \ngeq E_{new}$ and let's have the the same cycles $Z_1 = l_1, \cdots, Z_n = l_n$, but on the original resolution $\tX$ and let's denote $Z^* = \sum_{1 \leq i \leq n} Z_i$.

If we have $u, v \notin Z_i, 1 \leq i \leq n$, then we have $h^1(\calO_{Z'}) = h^1(\calO_{Z^*}) < h^1(\calO_Z)$, so in this case we are done.

On the other hand we can assume by symmetry, that $u \in |Z'_1|$, in this case we have $h^1(\calO_{Z_1}) \geq 1 - \chi(Z_1) > 1 - \chi(l_1) = h^1(\calO_{l_1})$ and $h^1(\calO_{Z_i}) 
= h^1(\calO_{l_i})$ for each $2 \leq i \leq n$, so we get again $h^1(\calO_{Z'}) < h^1(\calO_{Z^*}) \leq h^1(\calO_Z)$, which means that the statement follows also in this case.

Finally assume in the following that $l_1 \geq E_{new}$ and $l_i \ngeq E_{new}, E_u, E_v, 2 \leq i \leq n$, we know again that $h^1(\calO_{Z_i}) = h^1(\calO_{l_i})$ for each $2 \leq i \leq n$.

Let us write $l_1 = l + t \cdot E_{new}$, where we know that $l \leq Z$ and $t \leq Z_u + Z_v - 1$.
Let us have the same cycles $Z_2 = l_2, \cdots, Z_n = l_n$, but on the original resolution $\tX$, and $Z_1 = l$ on the resolution $\tX$, and let us denote again $Z^* = \sum_{1 \leq i \leq n} Z_i$.

We know that $h^1(\calO_{Z_i}) \geq h^1(\calO_{l_i}) = 1 - \chi(l_i)$, if $1 \leq i \leq n$.
If $Z^* < Z $, we get that $h^1(\calO_{Z'}) \leq h^1(\calO_{Z^*}) < h^1(\calO_Z)$.

On the other hand if $Z^* = Z$, then $n = 1$ and $t < Z_u + Z_v - 1$, in this case we have $h^1(\calO_{Z'}) = 1 - \chi(l_1) < 1 - \chi(Z) = h^1(\calO_{Z_{new}})$, which proves the satetment completely.

\end{proof}

\begin{lemma}\label{basepointregular}
Assume that $\mathcal{T}$ is an arbitrary rational homology sphere resolution graph and $\tX$ a generic singularity corresponding to it, and let us have a cycle $Z$ on it, such that $|Z| = \calv$, and $H^0(\calO_Z(Z+K))_{reg} \neq \emptyset$.
Assume that $v \in |Z|$ and $Z_v = 1$, then the line bundle $\calL_{Z} = \calO_Z(K + Z)$ hasn't got a basepoint on the regular part of the exceptional divisor $E_v$.
\end{lemma}

\begin{proof}

Assume to the contrary that $p$ is a regular point on the exceptional divisor $E_v$ and it is a base point of the line bundle $\calL_Z$ with multiplicity $m$.

In other words if we pick a generic section $s \in H^0( \calL_Z)_{reg}$ and it's divisor $D$, then the contribution of the point $p$ to the intersection number $(D, E_v)$ is $m$.

Since $Z_v = 1$ we can assume that locally at $p$ the divisor $D$ on the cycle $Z$ is the $m$-th multiple of a smooth divisor on $\tX$, transversal to the exceptional divisor $E_v$.

Notice that since $H^0(\calO_Z(Z+K))_{reg} \neq \emptyset$ by the proof of the Lemma\ref{baseI} we have $h^1(\calO_Z) = h^0(\calL_Z) = 1 - \chi(Z)$, which means that $h^1(\calL_Z) = 1$.

Let us blow up the singularity at $p$, let the new exceptional divisor be $E_{p}$, the new resolution $\tX_p$ and the blow up map $\pi_p$ and let us denote $Z_{p} = \pi_p^*(Z)- E_{p}$. 

Furthermore let the strict transform of $D$ be $D'$, then we have $\pi_p^*(D) = D' + mE_{p}$ and since $p$ is a base point of the line bundle $\calL_Z$ we have $1 = h^1(\calO_Z(D)) < h^1(\calO_{Z_{new}}(D'))$.

Notice that $Z_p$ is the same cycle as $Z$, just on the new resolution $\tX_p$ and $D'$ equals set theoretically $D \setminus m(p)$.

Let's have the subspace $ \Omega_{D} \subset H^1(\calO_Z)^* =  H^0( \calL_Z)$, where $\Omega_{D}$ consisiting of the differential forms vanishing on $\im(T_D(c^{(Z - Z_K)}(Z)))$.

Similarly let $\Omega_{D'} \subset H^1(\calO_Z)^* =  H^0( \calL_Z)$ be the set of differential forms vanishing on  the subspace $\im(T_{D'}(c^{(Z_p - (Z_K)_p + m \cdot E_v^*)}(Z)))$, we 
have $ \Omega_{D} \subset \Omega_{D'}$.

We know that $\dim(\Omega_{D}) = h^1(\calO_Z( D))=1$ and we know that $\Omega_{D}$ is generated by $s \in H^0(Z, \calL_Z)_{reg}$, which means in particular that $ \Omega_{D'} \cap H^0(Z, \calL_Z)_{reg}$ is nonempty and open in $\Omega_{D'}$.

\medskip

We know, that $ h^1(\calO_{Z_{new}}(D')) > h^1(\calO_Z(D))$, which means that $\dim(\Omega_{D'}) > \dim(\Omega_{D})$.

It means there exists a section $s' \in  \Omega_{D'} \cap H^0(Z, \calL_Z)_{reg}$, such that $s' \notin \Omega_{D}$.

We get that the differential form given by $s'$ hasn't got a pole on $D'$, however it has got a pole on $D$, which means that $s'$ must have a pole on the divisor $m(p)$.

It means precisely that there is a tangent vector $\overline{v} \in T_{m(p)}(\eca^{-m E_v^*}(Z))$, such that $T_{m(p)}(s' \circ c^{-m E_v^*}(Z))(\overline{v}) \neq 0$.

We know from this that the differential form corresponding to the section $s'$ has got a pole on the exceptional divisor $E_v$ (indeed it has got already a pole on the divisor $m(p)$
supported at a regular point of $E_v$).

On the other hand we have $Z_v = 1$ and $s' \in H^0(Z, \calL_Z)$, so it must have a pole of order $1$.

We will prove in the following that $T_{m(p)}(s' \circ c^{-m E_v^*}(Z))(\overline{v}) = 0$, which will lead to a contradiction.

The line bundle $\calL_Z$ has got a base point at $p$ with multiplicity $m$, which means that locally $s'$ has got the form for some local coordinates $(x, y)$:
$$\left(\frac{x^k}{y} + \sum_{0 \leq i, j}a_{i, j} x^i y^j \right) dx \wedge dy,$$ 
where $k \geq m$ and we have $E_v = (y = 0)$.

On the other hand since the divisor of $x^m$ on $Z$ is $m(p)$, which means that we can express the tangent vector $\overline{v}$ as the differential of some deformation of the divisor $m(p)$ given by $D_t = x^m + t  \cdot \sum_{0 \leq i \leq m-1} a_i x^i$, so we have $\overline{v} = \frac{d}{dt}D_t$.

By Laufer duality we have the following equality (see \cite{NNI}):

\begin{equation*}
T_{m(p)}(s' \circ c^{-m E_v^*}(Z))(\overline{v})  = \frac{d}{dt}\left(  \int_{|x|=\epsilon, \atop |y|=\epsilon}  \log \left(1+ t \cdot \frac{ \sum_{0 \leq i \leq m-1} a_i x^i }{x^m} \right) \left(\frac{x^k}{y} + \sum_{0 \leq i, j}a_{i, j}x^i y^j \right) dx \wedge dy  \right).
\end{equation*}

However from this formula we get that $T_{m(p)}(s' \circ c^{-m E_v^*}(Z))(\overline{v}) = 0$, since the coefficient of $ t \cdot x^{-1}y^{-1}$ is $0$, this contradiction proves the lemma completely.
\end{proof}

\section{Outline of the proof of Theorem\textbf{A}}

In this section we introduce the main ideas of proving Theorem\textbf{A} without going into technical details.

Assume to the contrary that there exists a line bundle $\calL \in \im(c^{-E_{u'}^* - E_{u''}^*}(Z))$ with $h^0(Z, \calL) > 1$, we can express this also in other words that there is a dominant and  injective function on an open subset $U \subset E_{u'}$, $f: U \to E_{u''}$, such that $\calL = \calO_Z(x + f(x))$.

Also this is equivalent to the condition that the algebraic curves $ \overline{\im(c^{-E_{u'}^*}(Z) )} \subset \pic^{-E_{u'}^*}(Z)$ and $ \overline{\im(c^{-E_{u''}^*}(Z) )} \subset \pic^{-E_{u''}^*}(Z)$ are symmetric to each other.
Notice that these two curves are in distinct Picard groups which are affine spaces, but they can be shifted to the same vector space $H^1(\calO_Z)$, so the property that they are symmetric
makes sence.

From this we can prove easily that  $V_Z(u')= V_Z(u'')$ and in particular $e_Z(u')= e_Z(u'') = e_Z(u', u'')$.

From the condition $e_Z(u')= e_Z(u'') = e_Z(u', u'')$ we also get that the algebraic curves $ \overline{\im(c^{-E_{u'}^*}(Z) )} \subset \pic^{-E_{u'}^*}(Z)$ and $ \overline{\im(c^{-E_{u''}^*}(Z) )} \subset \pic^{-E_{u''}^*}(Z)$ have only one symmetry point, so  $\calL \in \im(c^{-E_{u'}^* - E_{u''}^*}(Z))$ is the only line bundle such that $h^0(Z, \calL) > 1$.

As a next ste we show by analysing the complete linear series of the line bundle $\calO_Z(K + Z)$ that the vertices $u'$ and $u''$ are end vertices of the resolution graph, let us denote their
neighbour vertices by $w'$ and $w''$.

Next we analyse the divisors of sections in $H^0(\calO_Z(K + Z))$ restricted to the exceptional divisors $E_{u'}, E_{u''}$.
Using this and the results about relatively generic line bundles we prove in Lemma\ref{differentialpos} and after it that $Z_{w'} = Z_{w''}$ and $e_Z(u') = Z_{w'} - 1$.

Next in Lemma\ref{cyclesrecursive} we prove that there are cycles $Z^t, 1 \leq t \leq Z_{w'} -1$ with the properties $(Z^t)_{w'}= (Z^t)_{w''} =  Z_{w'} -(t-1)$ and
$H^0(\calO_{Z^t}(K + Z^t))_{reg} \neq \emptyset$ and $Z^t$ is the cohomology cycle of $Z^{t-1} - E_{w'}$ and $Z^{t-1} - E_{w''}$ simultaneously.

We also show that $e_{Z^t}(u') =e_{Z^t}(u'')  = e_{Z^t}(u', u'') = Z_{w'} - t$ and $h^0(\calO_{Z^{t}}( K + Z^{t})) =  h^0(\calO_{Z}(K + Z)) - t + 1$.

Let us look at the cycle $Z' = Z^{Z_{w'} - 3}$, then we have $Z'_{w'} = Z'_{w''} = 4$, $Z'_{u'} =Z'_{u''}= 1$ , $e_{Z'}(u')= e_{Z'}( u'') = e_{Z'}(u', u'')= 3$.

We know that the two curves $ \overline{\im(c^{-E_{u'}^*}(Z') )}$ and $\overline{ \im(c^{-E_{u''}^*}(Z') )}$ are symmetric to each other since they are projections
of the curves $ \overline{\im(c^{-E_{u'}^*}(Z) )}$ and $ \overline{\im(c^{-E_{u''}^*}(Z) )}$.

This means that there is a line bundle $\calL'_s \in \pic^{-E_{u'}^* - E_{u''}^*}(Z')$, such that if $\calL_1$ is a generic line bundle in $ \im(c^{-E_{u'}^*}(Z') )$, then $\calL_2 = \calL'_s - \calL_1$ is a generic line bundle in $ \im(c^{-E_{u''}^*}(Z') )$. 

This is equivalent to $H^0(Z', \calL'_s)_{reg} \neq \emptyset$ and $h^0(Z', \calL'_s) = 2$.

Similarly let us have the cycle $Z'' = Z^{Z_{w'} - 2}$, such that $Z''_{w'} = Z''_{w''} = 3$, $Z''_{u'} = Z''_{u''} = 1$, $e_{Z''}(u')= e_{Z''}( u'') = e_{Z''}(u', u'')= 2$ and
$h^0(\calO_{Z''}(K + Z'')) = h^0(\calO_{Z'}(K + Z')) - 1$.

We know similarly as before that the two curves $ \overline{\im(c^{-E_{u'}^*}(Z'') )}$ and $ \overline{\im(c^{-E_{u''}^*}(Z'') )}$ are also symmetric to each other.

This means that there is a line bundle $\calL_s \in \pic^{-E_{u'}^* - E_{u''}^*}(Z'')$, such that if $\calL_1$ is a generic line bundle in $ \im(c^{-E_{u'}^*}(Z'') )$, then $\calL_2 = \calL_s - \calL_1$ is a generic line bundle in $ \im(c^{-E_{u''}^*}(Z'') )$. 

This is equivalent to $H^0(Z'', \calL_s)_{reg} \neq \emptyset$ and $h^0(Z'', \calL_s) = 2$.

We know that there is a function $g: U' \to E_{u''}$, where $U' \subset E_{u'}$ is an open subset, such that $\calL_s = \calO_{Z''}(p + g(p))$ for $p \in U'$.

We also get easily that the restriction of the line bundle $\calL'_s$ under the map $r' : \pic^{-E_{u'}^*  - E_{u''}^*}(Z') \to \pic^{-E_{u'}^*  - E_{u''}^*}(Z'')$ must be $\calL_s$.

It means that $\calL'_s \in r'^{-1}(\calL_s)$ and notice that $\dim(r'^{-1}(\calL_s)) = 1$ since $h^1(\calO_{Z'}) - h^1(\calO_{Z''}) = 1$. 

We have $\calO_{Z'}(p + g(p)) = \calL'_s$ for  $p \in U'$, which means that the line bundle and Chern class $(- E_{u'} - E_{u''}, \calL_s)$ are not relatively generic on the cycle $Z'$.

On the other hand in Proposition\ref{relgenericcontr} we will prove using the theorems about relatively generic analyitc structures and the genericity of the analytic structures
of $\tX$ and $Z$ that the line bundle and Chern class $(- E_{u'} - E_{u''}, \calL_s)$ are not relatively generic on the cycle $Z'$ (or a statement which is equivalent to that).

This contradiction will prove Theorem\textbf{A} completely.

\section{Hyperelliptic invoultions on generic singularities, case of $2$ vertices}

In the following we prove that there is no linear series with degree $2$ and rank $1$ on a cycle on a generic singularity, more precisely we prove first the following theorem:

\begin{theorem}\label{twovertices}
Suppose that we have an arbitrary rational homology sphere resolution graph $\mathcal{T}$ and a generic resolution $\tX$ corresponding to it, and an effective integer cycle $Z \geq E$, such that $H^0(\calO_Z(K+Z))_{reg} \neq \emptyset$. Let us have two vertices $u', u'' \in \calv$, such that $Z_{u'} = Z_{u''} = 1$ and assume that $e_Z(u') \geq 3$.

With these conditions, for every line bundle $\calL \in \im(c^{-E_{u'}^* - E_{u''}^*}(Z))$ one has $h^0(Z, \calL) = 1$.
\end{theorem}
\begin{proof}

Notice, that since $H^0(\calO_Z(K+Z))_{reg} \neq \emptyset$,  by the previous lemmas we know that $\chi(Z') > \chi(Z)$ for every cycle $0 \leq Z' < Z$.
We also know that $h^0(\calO_Z(K+Z)) = h^1(\calO_Z) = 1 - \chi(Z)$ and $h^0(\calO_Z) = 1$.

Assume in the following to the contrary that there exists a line bundle $\calL \in \im(c^{-E_{u'}^* - E_{u''}^*}(Z))$ with $h^0(Z, \calL) > 1$.

Since we know the conditions $Z_{u'} = Z_{u''} = 1$ about the coeficcients, this means in other words that there is a dominant and  injective function on an open subset $U \subset E_{u'}$, $f: U \to E_{u''}$, such that $\calL = \calO_Z(x + f(x))$ (the injectivity follows from Lemma\ref{injectiveabel}).

Notice that since  $Z_{u'} = Z_{u''} = 1$ we know that $\dim(\eca^{-E_{u'}^*}(Z) ) = \dim(\eca^{-E_{u'}^*}(Z) ) = 1$ and since $h^1(\calO_Z) > h^1(\calO_{Z- E_{u'}}), h^1(\calO_{Z- E_{u''}})$ we get that $\dim(V_Z(u')), \dim(V_Z(u'')) > 0$.

Since $V_Z(u')$ and $V_Z(u'')$ are pararell to the affine clousures of  $\im(c^{-E_{u'}^*}(Z) ), \im(c^{-E_{u''}^*}(Z) )$ we get that 
$$ \dim(\im(c^{-E_{u'}^*}(Z) )) = \dim(\im(c^{-E_{u''}^*}(Z) )) = 1$$.

Let us denote the Zariski clousures of these $1$-dimensional subsets by $ \overline{\im(c^{-E_{u'}^*}(Z) )} \subset \pic^{-E_{u'}^*}(Z)$ and $ \overline{\im(c^{-E_{u''}^*}(Z) )} \subset \pic^{-E_{u''}^*}(Z)$.

We can also express our assumption in the equivalent way that the algebraic curves $ \overline{\im(c^{-E_{u'}^*}(Z) )} \subset \pic^{-E_{u'}^*}(Z)$ and $ \overline{\im(c^{-E_{u''}^*}(Z) )} \subset \pic^{-E_{u''}^*}(Z)$ are symmetric to each other.
Notice that these two curves are in distinct Picard groups which are affine spaces, but they can be shifted to the same vector space $H^1(\calO_Z)$, so the property that they are symmetric
makes sence.

Indeed one direction is trivial and the other direction follows from the facts that the algebraic curves $ \overline{\im(c^{-E_{u'}^*}(Z) )} $ and $ \overline{\im(c^{-E_{u''}^*}(Z) )} $ are irreducible, which means by unicity theorem that if local parts of them are symmetric to each other then they are symmetric.

Notice that the affine clousures $A( \overline{\im(c^{-E_{u'}^*}(Z) )} )$ and $A( \overline{\im(c^{-E_{u''}^*}(Z) )} )$ are pararell to the subspaces $V_{u'}(Z) \subset H^1(\calO_Z)$
and $V_{u''}(Z) \subset H^1(\calO_Z)$ respectively, which means that $V_Z(u')= V_Z(u'')$ and in particular $e_Z(u')= e_Z(u'') = e_Z(u', u'')$.

Notice also that the curves $\overline{\im(c^{-E_{u'}^*}(Z) )}$ and $ \overline{\im(c^{-E_{u''}^*}(Z) )}$ are not affine lines, since $e_Z(u')= e_Z(u'') \geq 3$, so they can have only one symmetry point.

Indeed if they had more than one symmetry point, then we could find a nonzero vector such that  $\overline{\im(c^{-E_{u'}^*}(Z) )}$ is invariant for the translation with that vector, however this is impossible since $\overline{\im(c^{-E_{u'}^*}(Z) )}$ is not an affine line.

This means in particular that $\calL \in \im(c^{-E_{u'}^* - E_{u''}^*}(Z))$ is the unique line bundle, such that $h^0(Z, \calL) > 1$.

Let us have a generic point $x \in U$, we know by  Lemma\ref{baseI} and Lemma\ref{basepointregular} that the line bundle $\calO_Z(K+Z)$ hasn't got a base point on the exceptional divisor $E_{u'}$, so there is a section $s \in H^0(\calO_Z(K+Z))_{reg}$, such that $x \in |s|$, we claim that we should have $f(x) \in |s|$ too.

Indeed it follows from the facts that $h^0(\calO_Z(K+Z - x)) = h^0(\calO_Z(K+Z )) - 1$ and $h^0(\calO_Z(K+Z - x)) = h^0(\calO_Z(K+Z - x - f(x))) $.

Consider the linear series $|H^0(\calO_Z(Z + K))|$ on the cycle $Z$, the line bundle $\calO_Z(Z + K)$ hasn't got a base point on the exceptional divisors $E_{u'}, E_{u''}$, which means that this complete linear series gives a map
 $$\eta: E_{u'} \cup E_{u''} \to \bP(H^0(\calO_Z(Z + K))^*).$$

Notice that if $x \in U$ generic, then we have $\eta(x) = \eta(f(x))$ and similarly if $y \in f(U)$ generic, then $\eta(y) = \eta(f^{-1}(y))$ so it follows that the image of the divisors $E_{u''}$ and $ E_{u'}$ is the same compact $1$-dimensional curve in $\bP(H^0(\calO_Z(Z + K))^*)$.

Let's prove next the following lemma:

\begin{lemma}\label{endvertexs}
The vertices $u'$ and $u''$ are end vertices of the resolution graph $\mathcal{T}$.
\end{lemma}

\begin{proof}

Suppose to the contrary that one of the vertices $u'$ and $u''$ is not an end vertex, we can suppose by symmetry that $u'$ is not an end vertex and assume that $u'$ has got a neighbour vertex $u_n$, which is not in the same connected component of $\mathcal{T} \setminus u'$ as the vertex $u''$. 

Let us look at the intesection point $I = E_{u'} \cap E_{u'_n}$ and it's image $\eta(I) \in  \bP(H^0(Z, Z + K)^*)$, we know that there is a point $I' \in  E_{u''}$, such that $\eta(I) = \eta(I')$.

Let us blow up the resolution at the point $I$ and let's denote the new resolution by $\tX_{new}$ which is generic corresponding to the blown up resolution graph $\mathcal{T}_{new}$,
let us denote the new exceptional divisor by $E_{new}$.

Let us denote the cycle $Z_{new} = \pi^*(Z) - E_{new}$ and consider the line bundle $\calO_{Z_{new} - E_{new}}(Z_{new} - E_{new} + K_{new})$.
We know from $\eta(I) = \eta(I')$ that this line bundle has got a base point at $I'$.

Notice that we have $$h^0(\calO_{Z_{new} - E_{new}}(Z_{new} - E_{new} + K_{new})) = h^0(\calO_{Z_{new}}(Z_{new} + K_{new}))- 1 = h^0(\calO_{Z}(Z+K))-1.$$

Suppose that $0 \leq A \leq Z_{new}- E_{new}$ is the fixed component cycle of the line bundle $\calO_{Z_{new} - E_{new}}(Z_{new} - E_{new} + K_{new})$.
It means that $A$ is the unique cycle such that $h^0(\calO_{Z_{new} - E_{new}}(Z_{new} - E_{new} + K_{new})) = h^0(\calO_{Z_{new} - E_{new}- A}(Z_{new} - E_{new} + K_{new} - A))$ and $H^0(\calO_{Z_{new} - E_{new}- A}(Z_{new} - E_{new} + K_{new} - A))_{reg} \neq \emptyset$.

Notice that $A \ngeq  E_{u''}$, because otherwise we would have $\eta(I) = \eta(p)$ for every $p \in E_{u''}$, which is impossible since $\eta(E_{u''})$ is a $1$-dimensional curve.

There are two cases in the following.

If $I'$ is a regular point of $ E_{u''}$, then we know that the line bundle $ \calO_{Z_{new} - E_{new}- A}(Z_{new} - E_{new} + K_{new} - A)$ has got a base point at $I'$, however this contradicts our previous lemmas.

In the other case assume that $I'$ is the intersection point $ E_{u''} \cap E_{u''_n}$, where $u''_n$ is a neighbour of the vertex $u''$.

We know that $I'$ is not a base point of the line bundle $\calO_{Z_{new} - E_{new}- A}(Z_{new} - E_{new} + K_{new} - A)$, which means that $A \geq E_{u''_n}$.

We claim that $A \ngeq E_{u'}$, indeed notice that $h^0(\calO_{Z- E_{u'}}(Z+ K - E_{u'})) < h^0(\calO_Z(K+Z)) - 1$, because $h^1(\calO_Z) - h^1(\calO_{Z-E_{u'}}) = e_Z(u') \geq 3$.

Let us have the connected component of the subgraph $\calv_{new} \setminus u'$ containing the vertex $u'_n$ and let us denote the part of the cycle $A$ contained in this component by
$A_1$ and write $A_2 = A - A_1$, notice that we have $A_2 \geq E_{u''_n}$.

We have the equality
\begin{equation*}
\left( \chi(Z_{new}- A_1 - E_{new}) - \chi(Z) \right) + \left( \chi(Z_{new}- A_2) - \chi(Z)  \right) = \chi(Z_{new}- A - E_{new}) - \chi(Z).
\end{equation*}

Notice that $H^0(\calO_Z(Z+K)))_{reg} \neq \emptyset$ which means $h^0(\calO_Z(Z+K)) = 1 - \chi(Z)$ and $H^0(\calO_{Z_{new}- A - E_{new}}(Z_{new}- A - E_{new}+K)))_{reg} \neq \emptyset$ which means $h^0(Z_{new}- A - E_{new}) = 1 - \chi(Z_{new}- A - E_{new})$ (since the resolution is generic).

It means that we get the following:

\begin{multline*}
(\chi(Z_{new} - A_1 - E_{new}) - \chi(Z)) + (\chi(Z- A_2) - \chi(Z)) \\ =  h^0(\calO_Z(Z+K)) -  h^0(\calO_{Z_{new}- E_{new}-A}(Z_{new}  - E_{new} - A + K_{new})) = 1.
\end{multline*}

\begin{equation*}
( h^0(\calO_Z( K+Z)) -  h^0(\calO_{Z_{new}- E_{new}-A_1}( Z_{new}  - E_{new} - A_1 + K_{new}))) + ( h^0(\calO_Z(K+Z)) -  h^0(\calO_{Z- A_2}( K+Z - A_2))) \leq 1.
\end{equation*}

Now we know that $h^0(\calO_Z(K+Z)) -  h^0(\calO_{Z_{new}- E_{new}-A_1}(Z_{new}  - E_{new} - A_1 + K_{new})) \geq 1$, so we get that $h^0(\calO_Z( K+Z)) -  h^0(\calO_{Z- A_2}( K+Z - A_2))= 0$.

However this is impossible, since $H^0(\calO_Z(K+Z))_{reg} \neq \emptyset$, this contradiction proves the statement of our lemma.

\end{proof}

It means now that we have proved, that both $u'$ and $u''$ are end vertices, let us denote their neighbour vertices respectively by $w'$ and $w''$.

Let us recall that since the affine clousure of the algebraic curves $ \overline{\im(c^{-E_{u'}^*}(Z) )}$ and $ \overline{\im(c^{-E_{u''}^*}(Z) )}$ are pararell, we have  $e_Z(u') = e_Z(u'') = e_Z(u', u'')$.

This means that $h^1(\calO_{Z- E_{u'}}) = h^1(\calO_{Z- E_{u''}}) = h^1(\calO_{Z- E_{u''}- E_{u'}})$.

Let us prove next the following lemma:

\begin{lemma}\label{differentialpos}
1) Let us have a generic differential form $\omega \in H^0(\calO_Z(Z+K))_{reg}$, and let us denote the restriction of the divisor of $\omega$ to the exceptional divisor $E_{u'}$
by $\sum_{1 \leq i \leq t_{u'}} p_i$ and to the exceptional divisor $E_{u''}$ by $\sum_{1 \leq i \leq t_{u''}} q_i$.

Then we have $t_{u'} = t_{u''}$ and with some reindexing we have $\calO_Z(p_i + q_i) = \calL$ for all $1 \leq i \leq t_{u'}$, in particular $Z_{w'} = Z_{w''}$.

2) If $r_1, \cdots, r_{t_{u'}}$ are generic points on the exceptional divisor $E_{u'}$, then we have $H^0(\calO_Z(Z+K - \sum_{1 \leq i \leq t_{u'}} r_i))_{reg} \neq \emptyset$.

Furthermore the dimension of the image of the map $H^0(\calO_Z(Z+K - \sum_{1 \leq i \leq t_{u'}} r_i)) \to H^0(\calO_{E_{u''}}(Z+K - \sum_{1 \leq i \leq t_{u'}} r_i))$ is $1$.

3) The maps $c^{-t_{u'} \cdot E_{u'}^*}(Z) : \eca^{-t_{u'} \cdot E_{u'}^*}(Z) \to \pic^{-t_{u'} \cdot E_{u'}^*}(Z)$ and $c^{-t_{u''} \cdot E_{u''}^*}(Z) : \eca^{-t_{u''} \cdot E_{u''}^*}(Z) \to \pic^{-t_{u''} \cdot E_{u''}^*}(Z)$ are birational to their image, this  means in particular that $e_Z(u') = e_Z(u'') \geq t_{u'}$. 

\end{lemma}

\begin{proof}

Fort part 1) notice that since the line bundle $\calO_Z(Z+K)$ has got no base points on the exceptional divisors $E_{u'}, E_{u''}$, for a generic section $\omega \in H^0(\calO_Z(Z+K))_{reg}$,
the divisor of $\omega$ consists of $t_{u'}$ disjoint points $p_1, \cdots, p_{t_{u'}}$ on the exceptional divisor $E_{u'}$ and $t_{u''}$ disjoint points $q_1, \cdots, q_{t_{u''}}$ on the exceptional divisor $E_{u''}$.

Let us have an open subset $U' \subset U \subset E_{u'}$, such that the map $f: U' \to E_{u''}$ is a biholomorphism and if $\omega $ is enough generic we can assume that
$p_1, \cdots, p_{t_{u'}} \in U'$ and $q_1, \cdots, q_{t_{u''}} \in f(U')$.

Notice that if $x \in U'$ is an arbitrary point and $s \in H^0(\calO_Z(Z+K))_{reg}$ is a section such that $x \in |s|$, then since $H^0(\calO_Z(Z+K - x)) = H^0(\calO_Z(Z+K -x - f(x)))$
we have $f(x) \in |s|$.

Similarly if $y \in f(U')$ is an arbitrary point and $s \in H^0(\calO_Z(Z+K))_{reg}$ is a section such that $y \in |s|$, then $f^{-1}(y) \in |s|$.

This means that the set $f(p_1), \cdots, f(p_{t_u'})$ equals the set $q_1, \cdots, q_{t_{u''}}$, which indeed means that $t_{u'} = t_{u''}$ and with some reindexing we have
$q_i = f(p_i)$ so $\calO_Z(p_i + q_i) =\calO_Z(p_i + f(p_i))  = \calL$.

Notice that $t_{u'} = (Z+K, E_{u'}) = Z_{w'}-2$ and $t_{u''} = (Z+K, E_{u''}) = Z_{w''}-2$, which indeed proves $Z_{w'} = Z_{w''}$.

Notice also that $\calO_Z(K+Z) = \calO_Z(\sum_{1 \leq i \leq t_{u'}}(p_i + q_i)) = t_{u'} \cdot \calL$.

\medskip

For part 2) assume that $r_1, \cdots, r_{t_{u'}} \in U'$ are generic points on the exceptional divisor $E_{u'}$ and notice that $$\calO_Z(K+Z) =  t_{u'} \cdot \calL = \calO_Z \left(\sum_{1 \leq i \leq t_{u'}}(r_i + f(r_i)) \right).$$

This indeed yields $H^0(\calO_Z (Z+K - \sum_{1 \leq i \leq t_{u'}} r_i ))_{reg} \neq \emptyset$.

On the other hand by part 1) the points $f(r_i)$ are base points of the line bundle $\calO_Z(Z+K - \sum_{1 \leq i \leq t_{u'}} r_i)$ and $t_{u'} = t_{u''}$, so if $s \in H^0(\calO_Z(Z+K - \sum_{1 \leq i \leq t_{u'}} r_i))_{reg}$, then the divisor of $s$ on the exceptional divsior $E_{u''}$ should be $\sum_{1 \leq i \leq t_{u'}}f(r_i)$.

This indeed means, that the dimension of the image of the map $H^0(\calO_Z(Z+K - \sum_{1 \leq i \leq t_{u'}} r_i)) \to H^0(\calO_{E_{u''}}(Z+K - \sum_{1 \leq i \leq t_{u'}} r_i))$ is $1$.

\medskip

For part 3) assume to the contrary that the map $c^{-t_{u'} \cdot E_{u'}^*}(Z) : \eca^{-t_{u'} \cdot E_{u'}^*}(Z) \to \pic^{-t_{u'} \cdot E_{u'}^*}(Z)$ is not birational to its image.

It means that there are generic points $p_1, p_2, \cdots, p_{t_{u'}} \in U'$ and  a different tupple of generic points $p'_1, p_2, \cdots, p'_{t_{u'}} \in U'$, such that $\calO_Z(\sum_{1 \leq i \leq t_{u'}} p_i) = \calO_Z(\sum_{1 \leq i \leq t_{u'}} p'_i)$.
Notice that $\calO_Z(Z + K) = \calO_Z(\sum_{1 \leq i \leq t_{u'}} p'_i + \sum_{1 \leq i \leq t_{u'}} f(p_i ))$, which contradicts part 1) and 2).

Similarly we get that the map $c^{-t_{u''} \cdot E_{u''}^*}(Z) : \eca^{-t_{u''} \cdot E_{u''}^*}(Z) \to \pic^{-t_{u''} \cdot E_{u''}^*}(Z)$ is birational to its image.

\end{proof}

Notice that $h^1(\calO_Z) = 1 - \chi(Z)$ and $h^1(\calO_{Z- E_{u'}}) \geq 1- \chi(Z- E_{u'})$,
which means that $e_Z(u') \leq  \chi(Z- E_{u'}) - \chi(Z) = Z_{w'} - 1$.

On the other hand by part 3) of Lemma\ref{differentialpos} we have $e_Z(u')  \geq t_{u'} = Z_{w'} - 2$, so there are two cases, $e_Z(u') = Z_{w'} - 2$ or $e_Z(u')  = Z_{w'} - 1$.

Assume first that $e_Z(u')  = Z_{w'} - 2$, this means that the map $c^{-t_{u'} \cdot E_{u'}^*}(Z) : \eca^{-t_{u'} \cdot E_{u'}^*}(Z) \to \pic^{-t_{u'} \cdot E_{u'}^*}(Z)$ is dominant.

From part 2) of Lemma\ref{differentialpos} we know that if $p_1, \cdots, p_{t_{u'}}$ are generic points on the exceptional divisor $E_{u'}$, then we have $H^0(\calO_Z( Z+K - \sum_{1 \leq i \leq t_{u'}} p_i))_{reg} \neq \emptyset$.

On the other hand let us have the trivial line bundle $\calL_{u'} =\calO_{Z- E_{u'}}$ and the restriction map $r_{u'} : \pic^{-t_{u'} \cdot E_{u'}^*}(Z) \to \pic^{0}(Z-E_{u'})$.

We know that the line bundle $ \calO_Z(\sum_{1 \leq i \leq t_{u'}} p_i)$ is a relatively generic line bundle in $r_{u'}^{-1}(\calL_{u'})$, since the map $c^{-t_{u'} \cdot E_{u'}^*}(Z) : \eca^{-t_{u'} \cdot E_{u'}^*}(Z) \to \pic^{-t_{u'} \cdot E_{u'}^*}(Z)$ is dominant.

We know from the previous lemma that $H^0(\calO_Z( Z+K - \sum_{1 \leq i \leq t_{u'}} p_i ) )_{reg} \neq \emptyset$ , so  $h^0(\calO_Z( Z+K - \sum_{1 \leq i \leq t_{u'}} p_i ) ) > h^0(\calO_{Z- E_{u'}}(Z- E_{u'} + K - \sum_{1 \leq i \leq t_{u'}} p_i))$, which means by Seere duality that $h^1(\calO_Z( \sum_{1 \leq i \leq t_{u'}} p_i)) > h^1(\calO_{Z- E_{u'}})$.

On the other hand from Theorem\ref{th:dominantrel} we know that $h^1(\calO_Z( \sum_{1 \leq i \leq t_{u'}} p_i)) = h^1(\calO_{Z- E_{u'}}(\sum_{1 \leq i \leq t_{u'}} p_i)) =
h^1(\calO_{Z- E_{u'}})$, which is a contradiction.

It means that we can assume in the following, that $e_Z(u') = Z_{w'} - 1 = t_{u'} + 1$.

We prove the next lemma in the following:

\begin{lemma}\label{cyclesrecursive}
We define recursively the following cycles $Z^t, 1 \leq t \leq Z_{w'} -1$.

Let us have $Z^1 = Z$, and suppose that $Z^1, \cdots, Z^{t-1}$ are already defined with the properties $(Z^i)_{w'}= (Z^i)_{w''} =  Z_{w'} -(i-1)$  and
$H^0(\calO_{Z^i}(K + Z^i))_{reg} \neq \emptyset$ if $1 \leq i \leq t-1$.

Let us denote the cohomology cycle  of the cycle $Z_{t-1} -  E_{w'}$ by $Z^{t}$, so the minimal cycle, such that $h^1(\calO_{Z^{t-1} -  E_{w'}}) = h^1(\calO_{Z^{t}})$, 
or in other words $H^0(\calO_{Z^{t-1} - E_{w'}}( K + Z^{t-1}- E_{w'})) = H^0(\calO_{Z^{t}}( K + Z^{t}))$ and $H^0(\calO_{Z^{t}}(K + Z^{t}))_{reg} \neq \emptyset$.

We claim that $Z^{t}_{w'} = Z^{t}_{w''} = Z^{t-1}_{w'} - 1$, $ h^0(\calO_{Z^{t}}( K + Z^{t})) =  h^0(\calO_{Z^{t-1}}(K + Z^{t-1})) - 1$ and furthermore we have $e_{Z^t}(u') =e_{Z^t}(u'') =  e_{Z^t}(u', u'') = e_{Z^{t-1}}(u') - 1$.

\end{lemma}
\begin{proof}

We prove the statements of the lemma by induction on the parameter $t$, the base case is trivial.

We prove first the equality $ h^0(\calO_{Z^{t}}( K + Z^{t})) =  h^0(\calO_{Z^{t-1}}(K + Z^{t-1})) - 1$.

Notice first that trivially by the induction hypothesis we have $e_{Z^i}( u') \geq e_{Z^{i-1}}( u')- 1$ for $1 \leq i \leq t-1$, so this means that $e_{Z^{t-1}}( u')= e_{Z^{t-1}}(u'') > 1$.

On the other hand notice that the two algebraic curves $ \overline{\im(c^{-E_{u'}^*}(Z^{t-1}) )}$ and $\overline{\im(c^{-E_{u''}^*}(Z^{t-1}))}$ are projections
of the algebraic curves $ \overline{\im(c^{-E_{u'}^*}(Z) )}$ and $\overline{\im(c^{-E_{u''}^*}(Z))}$, which means that they are symmetric to each other and their affine clousures are pararell.

Let us have the canonical map $\eta_{t-1} : Z^{t-1} \to \bP( H^0(\calO_{Z^{t-1}}( K + Z^{t-1}))^*  )$.

Since $e_{Z^{t-1}}(u') = e_{Z^{t-1}}(u'')> 1$ we get exactly the same way as in the case of $Z_1 = Z$ that the image of the curves $E_{u'}$ and $E_{u''}$ are $1$-dimensional and they are the same $1$-dimensional compact projective curves.

Let us denote the intersection point of $E_{u'}$ and $E_{w'}$ by $I$ and let's blow up the resolution at the point $I$.
Let us denote the new resolution by $\tX_{new}$, the new exceptional divisor by $E_{new}$ and let's have the cycle $Z_{new} = \pi^*(Z^{t-1}) - E_{new}$.

We get that there is a point $q \in E_{u''}$, such that every section in $H^0(\calO_{Z_{new} - E_{new}}(Z_{new} - E_{new} + K_{new}))$ vanish at $q$.

Obviously we have $h^0(\calO_{Z_{new} - E_{new}}(Z_{new} - E_{new} + K_{new})) = h^0(\calO_{Z_{new}}(Z_{new} + K_{new})) - 1$.

Let $0 \leq A \leq Z_{new} - E_{new}$ be the fixed component cycle of the line bundle $\calO_{Z_{new} - E_{new}}(Z_{new} - E_{new} + K_{new})$.

It means that $A$ is the unique cycle such that $$H^0(\calO_{Z_{new} - E_{new}}(Z_{new} - E_{new} + K_{new})) = H^0(\calO_{Z_{new} - E_{new} - A}(Z_{new} - E_{new} - A+ K_{new})),$$
 and $H^0(\calO_{Z_{new} - E_{new} - A}(Z_{new} - E_{new} - A+ K_{new}))_{reg} \neq \emptyset$.

Since the image of the curves $E_{u'}, E_{u''}$ are one dimensional we know that $A \ngeq E_{u'}, E_{u''}$, we claim in the following that $A \geq E_{w'}$.

Indeed assume that $A \ngeq E_{w'}$ and let the restriction of the cycle to the vertex set $\calv$ be $A'$.

In this case we would get $H^0(\calO_{Z_{t-1}-A'}(Z_{t-1} - A' + K)) \neq \emptyset$ and $H^0(\calO_{Z_{t-1}-A'}(Z_{t-1} - A' + K)) = H^0(\calO_{Z_{new} - E_{new} - A}(Z_{new} - E_{new} - A+ K_{new}))$, so we would get that the line bundle $\calO_{Z_{t-1}-A'}(Z_{t-1} - A' + K)$ has got a base point at the intersection point $E_{u'} \cap E_{w'} = I$, which is impossible by Lemma\ref{baseI}.

\emph{It means that we got that $A \geq E_{w'}$ and we also claim that $A \geq E_{w''}$ in the following:}

We know that the line bundle $\calO_{Z_{new} - E_{new} - A}(Z_{new} - E_{new}- A + K_{new})$ hasn't got a base point at $q$, so we get that $A \geq E_{u''}$ if $q$ is a regular point on the exceptional divisor $E_{u''}$.
On the other hand this is impossible, since it would mean that the image $\eta_{t-1}(E_{u''})$ is just one point.
It means that $q = E_{u''} \cap E_{w''}$ and we indeed get $A \geq E_{w''}$.

Since $A \geq E_{w'}$, we have $$H^0( \calO_{Z_{new} - E_{new}}( Z_{new} - E_{new} + K_{new}))= H^0( \calO_{Z_{new} - E_{w'}}( Z_{new} - E_{w'} + K_{new})).$$

It yields indeed $h^0(\calO_{Z^{t-1} - E_{w'}}( Z^{t-1}- E_{w'} + K)) = h^0(\calO_{Z^{t-1}}(Z^{t-1} + K)) - 1$, so $ h^0(\calO_{Z^{t}}( K + Z^{t})) =  h^0(\calO_{Z^{t-1}}( Z^{t-1} + K)) - 1$.

The claim $e_{Z^t}(u') = e_{Z^t}(u'')= e_{Z^{t-1}}(u') - 1$ follows from $H^0(\calO_{Z- E_{u'}}(Z- E_{u'} + K)) \subset H^0(\calO_{Z^{t}}(K + Z^{t})) $, which is easy to show by induction.

Indeed for $t= 1$ the statement  is true, since $H^0(\calO_{Z- E_{u'}}(Z- E_{u'} + K)) \subset H^0(\calO_{Z}(K + Z))$.

Assume that $H^0(\calO_{Z- E_{u'}}(Z- E_{u'} + K)) \subset H^0(\calO_{Z^{t-1}}(K + Z^{t-1})) $, notice that in particular this means that $H^0(\calO_{Z- E_{u'}}(Z- E_{u'} + K)) = H^0(\calO_{Z_{t-1}- E_{u'}}(Z_{t-1}- E_{u'} + K))$.

Notice on the other hand that $H^0(\calO_{Z^{t}}(K + Z^{t}))  \subset H^0(\calO_{Z^{t-1}}(K + Z^{t-1}))$ is the subset of sections, which vanish at the intersection point $I = E_{u'} \cap E_{w'}$, which means indeed $H^0(\calO_{Z- E_{u'}}(Z- E_{u'} + K)) \subset H^0(\calO_{Z^{t}}(K + Z^{t})) $.

\underline{In the following we want to prove the equalities $Z^{t}_{w'} = Z^{t}_{w''} = Z^{t-1}_{w'} - 1$:}

\medskip

If $ t \leq Z_{w'} -2 = Z_{w''} - 2$, then notice that $H^0(\calO_{Z^{t}}(K + Z^{t}))_{reg} \neq \emptyset$ and $e_{Z^{t}}(v') = e_{Z^{t}}(v'') > 1$.
Also note that the algebraic curves $ \overline{\im(c^{-E_{u'}^*}(Z^{t}) )}$ and $\overline{\im(c^{-E_{u''}^*}(Z^{t}))}$ are symmetric to each other, so we get that $(Z^t)_{w'} = (Z^t)_{w''}$ in the same way as in the case $t= 1$.

Suppose on the other hand that $t = Z_{w'} -1$, in this case obviously we get $Z^{t}_{w'} \leq 2$.

On the other hand $e_{Z^{t}}(v') \geq 1$, which yields indeed that $Z^{t}_{w'} = 2$.

Indeed if $Z^{t}_{w'} < 2$ would happen, then $\chi(Z^{t} - E_{v'}) \leq \chi(Z^{t})$ and $H^0(\calO_{Z^{t}}(K + Z^{t}))_{reg} \neq \emptyset$ would give $e_{Z^{t}}(v') = 0$, which is impossible.

It means in particular that $Z^{t}_{w'} = Z_{w'} - t + 1$ if $t \leq Z_{w'} -1$, since $(Z_t)_{w'} < (Z_{t-1})_{w'}$ for every $1 < t \leq Z_{w'}-1$ and $(Z_{Z_{w'}-1})_{w'} = 2$.
From this we also get that if $ t \leq Z_{w'} -2 = Z_{w''} - 2$, then $Z^{t}_{w''} = Z^{t}_{w'} = Z_{w'} - t + 1$.

Suppose again that $t = Z_{w'} -1$, we still have to prove that $Z^{t}_{w''} \leq 2$, this will show again that $Z^{t}_{w''} = 2$, because $e_{Z^{t}, v''} \geq 1$ means again that $Z^{t}_{w''} \geq 2$.

We know that $Z^{t-1}_{w''} = 3$, so with the notations  of the proof of the first part the line bundle $\calO_{Z_{new} - E_{new} - A}( Z_{new} - E_{new}- A + K_{new})$ hasn't got a base point at $q$ and $A \ngeq E_{u''}$ so it means that $A \geq E_{w''}$ and $q = E_{u''} \cap E_{w''}$.

It means that $A \geq E_{w''}$, which yields $(Z^t)_{w''} = (Z^{t-1} - A)_{w''} \leq 2$, this proves the statement of the lemma completely.

\end{proof}

Now notice that by our condition $e_{Z}(u'') \geq 3$ we have  $Z_{w'} = Z_{w''} \geq 4$.

Let us look at the cycle $Z' = Z^{Z_{w'} - 3}$, we have by Lemma\ref{cyclesrecursive} that $Z'_{w'} = Z'_{w''} = 4$, $Z'_{u'} =Z'_{u''}= 1$ , $e_{Z'}(u')= e_{Z'}( u'') = e_{Z'}(u', u'')= 3$ and we know that the two curves $ \overline{\im(c^{-E_{u'}^*}(Z') )}$ and $\overline{ \im(c^{-E_{u''}^*}(Z') )}$ are symmetric to each other.

Similarly let us have the cycle $Z'' = Z^{Z_{w'} - 2}$, such that $Z''_{w'} = Z''_{w''} = 3$, $Z''_{u'} = Z''_{u''} = 1$, $e_{Z''}(u')= e_{Z''}( u'') = e_{Z''}(u', u'')= 2$ and
$h^0(\calO_{Z''}(K + Z'')) = h^0(\calO_{Z'}(K + Z')) - 1$ and furthermore $H^0(\calO_{Z''}(K + Z''))_{reg} \neq 0$.

We know that the two curves $ \overline{\im(c^{-E_{u'}^*}(Z'') )}$ and $ \overline{\im(c^{-E_{u''}^*}(Z'') )}$ are projections of the curves $ \overline{\im(c^{-E_{u'}^*}(Z') )}$ and $ \overline{\im(c^{-E_{u''}^*}(Z') )}$.

This means that the two curves are also symmetric to each other, so there is a line bundle $\calL_s \in \pic^{-E_{u'}^* - E_{u''}^*}(Z'')$, such that if $\calL_1$ is a generic line bundle in $ \im(c^{-E_{u'}^*}(Z'') )$, then $\calL_2 = \calL_s - \calL_1$ is a generic line bundle in $ \im(c^{-E_{u''}^*}(Z'') )$. 

This is equivalent to $H^0(Z'', \calL_s)_{reg} \neq \emptyset$ and $h^0(Z'', \calL_s) = 2$, in the following we want to identify the line bundle $\calL_s$.

We claim first that $H^0(\calO_{Z''}( Z'' + K) \otimes \calL_s^{-1})_{reg} \neq \emptyset$.

Indeed for this we have to show that if there is a cycle $0 \leq A < Z''$ and $H^0(\calO_A(K + A))_{reg} \neq \emptyset$, then $h^1(A, \calL_s) < h^1(Z'', \calL_s)$.

We have by Riemann-Roch that $h^1(Z'', \calL_s) = h^1(\calO_{Z''}) - 1$, and $h^1(A, \calL_s) \leq h^1(\calO_A) - 1$ if $A \geq E_{u'} + E_{u''}$ and $h^1(A, \calL_s) = h^1(\calO_A)$ if $A \leq Z'' - E_{u'}$ or $A \leq Z'' - E_{u''}$ (in fact the two conditions are the same by Lemma\ref{cyclesrecursive}.

In the first case we clearly have $h^1(\calO_A) - 1 < h^1(\calO_{Z''}) - 1$ and in the second case we have $h^1(\calO_A) \leq h^1(\calO_{Z'' - E_{u'} - E_{u''}}) = h^1(\calO_{Z''}) - 2$, which proves our claim completely.

Now let us have the cycle $Z''_r = Z'' - E_{u'} - E_{u''}$ and consider the line bundle $\calO_{Z''}( Z'' + K) \otimes \calL_s^{-1}$ on $Z'$.

Let the restriction of this line bundle to the cycle $Z''_r$ be $\calL_r = \calO_{Z''_r}( Z'' + K)$, let us denote the Chern class $l'' = Z'' - Z_K + E_{u'}^* + E_{u''}^*$.

We have the subspace $\eca^{ l'', \calL_r}(Z'')  \subset \eca^{l''}(Z'')$ consisting of divisors $D \in \eca^{l''}(Z'')$, such that $\calL_r = \calO_{Z''_r}(D)$.

We know from \cite{R} that $\eca^{ l'', \calL_r}(Z'')$ is an irreducible smooth algebraic subvariety of $\eca^{l''}(Z'')$ and we have the relative Abel map $c^{ l'', \calL_r}(Z'') : \eca^{ l'', \calL_r}(Z'') \to r^{-1}(\calL_r)$.

We prove the following lemma about the relative Abel map $c^{ l'', \calL_r}(Z'') : \eca^{ l'', \calL_r}(Z'') \to r^{-1}(\calL_r)$.

\begin{lemma}\label{relabelconst}
The relative Abel map $c^{ l'', \calL_r}(Z'')$ is constant and the image is the line bundle $\calO_{Z''}( Z'' + K) \otimes \calL_s^{-1}$.
\end{lemma}

\begin{proof}
We know that $H^0(\calO_{Z''}( Z'' + K) \otimes \calL_s^{-1})_{reg} \neq \emptyset$, which means that the line bundle $\calO_{Z''}( Z'' + K) \otimes \calL_s^{-1}$ is in the image of the map $c^{ l'', \calL_r}(Z'')$ and furthermore obviously we have $$(c^{ l'', \calL_r}(Z''))^{-1}(\calO_{Z''}( Z'' + K) \otimes \calL_s^{-1}) = (c^{l''}(Z''))^{-1}( \calO_{Z''}( Z'' + K - \calL_s) ).$$

On the other hand we know that $h^1(\calO_{Z''}( Z'' + K) \otimes \calL_s^{-1}) = h^0(Z'', \calL_s) = 2$, and we claim that $ h^1(\calO_{Z''_r}(Z'' + K)) = 0$.

Indeed we know that $h^1(\calO_{Z'' - E_{u'}}) = h^1(\calO_{Z''_r})$ and $H^0(\calO_{Z'' - E_{u'}}(Z'' + K))_{reg} \neq \emptyset$, so we have $h^1(\calO_{Z''_r}(Z'' + K)) = h^1(\calO_{Z'' - E_{u'}}(Z'' + K))$.

It means that we only have to prove that $h^1(\calO_{Z'' - E_{u'}}(Z'' + K)) = 0$.

Notice that we have $h^1(\calO_{Z'' - E_{u'}}(Z'' + K)) = h^0(\calO_{Z'' - E_{u'}}(- E_{u'}))$ by Seere duality.

On the other hand we have the following exact sequence:

\begin{equation*}
0 \to H^0(\calO_{Z'' - E_{u'}}(- E_{u'})) \to H^0(\calO_{Z''}) \to H^0(\calO_{E_{u'}}).
\end{equation*}

Since the map $H^0(\calO_{Z''}) \to H^0(\calO_{E_{u'}})$ is surjective and $h^0(\calO_{E_{u'}}) = 0$ we get that $h^0(Z'' - E_{u}, - E_{u}) = h^0(\calO_{Z''}) - 1$ 

It means that $h^1(\calO_{Z''}) = 1 - \chi(Z'')$  since $H^0(\calO_{Z''}(Z'' + K))_{reg} \neq \emptyset$.

We got that $h^1(\calO_{Z''}( Z'' + K) \otimes \calL_s^{-1}) =2$ and $ h^1(\calO_{Z''_r}(Z'' + K)) = 0$.

Assume to the contrary that $c^{ l'', \calL_r}(Z'')$ is nonconstant, this means that $ (c^{l''}(Z''))^{-1}( \calO_{Z''}( Z'' + K) \otimes \calL_s^{-1})$ is a proper smooth algebraic subvariety of $\eca^{ l'', \calL_r}(Z'')$.

Notice that by \cite{R} we have $$\dim(\eca^{ l'', \calL_r}(Z'')) = (l'', Z'')- h^1(\calO_{Z''_r}) + h^1(Z''_r, \calL_r) = (l'', Z'')- h^1(\calO_{Z''_r}).$$

On the other hand we have $$\dim((c^{l''}(Z''))^{-1}( \calO_{Z''}( Z'' + K) \otimes \calL_s^{-1})   ) = h^0(\calO_{Z''}( Z'' + K) \otimes \calL_s^{-1})  - h^0(\calO_{Z''}) = h^0(\calO_{Z''}( Z'' + K) \otimes \calL_s^{-1})  - 1.$$

It means that we have $\dim((c^{l''}(Z''))^{-1}( \calO_{Z''}( Z'' + K) \otimes \calL_s^{-1})   ) = h^1( Z'', \calL_s) - 1 = h^1(\calO_{Z''}) - 2$.

Notice that $ (l'', Z'')- h^1(\calO_{Z''_r})  = (l'', Z'') - h^1(\calO_{Z''}) + 2 = (l'', Z'') + \chi(Z'') + 1$ and $h^1(\calO_{Z''}) - 2 = - \chi(Z'') - 1$ and an easy calculation shows that they
are the same.

This proves that $\dim((c^{l''}(Z''))^{-1}( \calO_{Z''}( Z'' + K) \otimes \calL_s^{-1})   )  = \dim(\eca^{ l'', \calL_r}(Z'') )$, so indeed the map $c^{ l'', \calL_r}(Z'')$ is constant.
\end{proof}

Let us have a generic section $s \in H^0(\calO_{Z'' - E_{u'}}(Z'' + K))_{reg}$ and let's denote $|s| = D$.

Since $D \in \eca^{ l'', \calL_r}(Z'')$ we know from Lemma\ref{relabelconst} that $\calO_{Z''}( Z'' + K) \otimes \calL_s^{-1} = \calO_{Z''}( D)$, so $\calL_s = \calO_{Z''}(Z'' + K - D)$.

We know that the two curves $\overline{\im(c^{-E_{u'}^*}(Z') )}$ and $ \overline{\im(c^{-E_{u''}^*}(Z') )}$ are symmetric to each other, which means that there is a line bundle $\calL'_s \in \pic^{-E_{u'}^* - E_{u''}^*}(Z')$, such that if $\calL_1$ is a generic line bundle in $ \im(c^{-E_{u'}^*}(Z') )$, then $\calL_2 = \calL'_s \otimes \calL_1^{-1}$ is a generic line bundle in $ \im(c^{-E_{u''}^*}(Z') )$. 

This is equivalent to $H^0(Z', \calL'_s)_{reg} \neq \emptyset$ and $h^0(Z', \calL'_s) = 2$.

In another euqivalent way we know that there is a function $g: U' \to E_{u''}$, where $U' \subset E_{u'}$ is an open subset, such that $\calL_s = \calO_{Z''}(p + g(p))$ for $p \in U'$.

We know that the restriction of the line bundle $\calL'_s$ under the map $r' : \pic^{-E_{u'}^*  - E_{u''}^*}(Z') \to \pic^{-E_{u'}^*  - E_{u''}^*}(Z'')$ must be $\calL_s$.

It means that $\calL'_s \in r'^{-1}(\calL_s)$ and notice that $\dim(r'^{-1}(\calL_s)) = 1$ since $h^1(\calO_{Z'}) - h^1(\calO_{Z''}) = 1$. We have $\calO_{Z'}(p + g(p)) = \calL'_s$ for  $p \in U'$, which means that the line bundle and Chern class $(- E_{u'} - E_{u''}, \calL_s)$ are not relatively generic on the cycle $Z'$.

Let us have the cycle $Z' - E_{w'}$, we know that $H^1(\calO_{ Z' - E_{w'}}) = H^1(\calO_{Z''})$, and so there is a unique line bundle in $ \pic^{-E_{u'}^*  - E_{u''}^*}(Z' - E_{w'})$, 
which projects to $\calL_s$, let us denote it by $\calL_{h}$.

The next proposition will give our final contradiction:

\begin{proposition}\label{relgenericcontr}
The line bundle and Chern class $( -E_{u'}^* - E_{u''}^*, \calL_h)$ are relatively generic on the cycle $Z'$.
\end{proposition}

\begin{proof}

By Theorem\ref{th:dominantrel} we have to show that:

\begin{equation*}
\chi(E_{u'}^* + E_{u''}^*) - h^1(\calO_{Z' - E_{w'}}(Z'' + K - D)) < \chi(E_{u'}^* + E_{u''}^* + A ) -  h^1(\calO_{\min( Z'- A, Z' - E_{w'})}( Z'' + K - A - D)),
\end{equation*}
in case we have $0 < A \leq Z'$ and $H^0(\calO_{\min( Z'- A, Z' - E_{w'})}( Z'' + K - A - D))_{reg} \neq \emptyset$.

We know that $h^1(\calO_{Z' - E_{w'}}(Z'' + K - D))  = h^1(\calO_{Z''}(Z'' + K - D)) = h^1(\calO_{Z''}) - 1 = -\chi(Z'')$, so we need to deal with the cohomology numbers $h^1(\calO_{\min( Z'- A, Z' - E_{w'})}(Z'' + K - A - D)) $.

\medskip

\underline{Asumme first that $A \geq E_{u'}$ or $A \geq E_{u''}$:}

\medskip 

In this case we know that $h^1(\calO_{\min( Z'- A, Z' - E_{w'})}(Z'' + K - A - D)) = h^1(\calO_{\min( Z'- A, Z''_r)}( Z'' + K - A - D))$,
since in this case $H^0(\calO_{Z'- A}(Z'- A + K)) \subset H^0(\calO_{Z''_r}(Z''_r+ K))$ and $H^0(\calO_{\min( Z'- A, Z' - E_{w'})}( Z'' + K - A - D))_{reg} \neq \emptyset$.

Since $\calO_{Z''_r}(Z'' + K - A - D) = \calO_{Z''_r}( - A)$  we get that $h^1(\calO_{\min( Z'- A, Z' - E_{w'})}(Z'' + K - A - D)) = h^1(\calO_{\min( Z'- A, Z' - E_{w'}- E_{u'}- E_{u''})}(  - A) )$, so we have to prove that:

\begin{equation*}
\chi(E_{u'}^* + E_{u''}^*) - h^1( \calO_{Z' - E_{w'}}(Z'' + K - D)) <  \chi(E_{u'}^* + E_{u''}^* + A ) - h^1(\calO_{\min( Z'- A, Z' - E_{w'}- E_{u'}- E_{u''})}(  - A )).
\end{equation*}

Let us have a generic line bundle $\calL_{gen} \in \im(c^{-E_{u'}^*  - E_{u''}^*}(Z' - E_{w'}))$ on the image of the Abel map, it means in particular that $H^0(Z' - E_{w'}, \calL_{gen})_{reg} \neq \emptyset$ and also we have $h^0(Z' - E_{w'}, \calL_{gen}) = 1$.

Consider the line bundles in the inverse image $r'^{-1}(\calL_{gen})$, where $r': \pic^{-E_{u'}^* - E_{u''}^*}(Z') \to \pic^{-E_{u'}^* - E_{u''}^*}(Z' - E_{w'})$ is the restriction map. 

We know that the space $r'^{-1}(\calL_{gen})$ is $1$ dimensional and the image of the Abel map $c^{-E_{u'}^* - E_{u''}^*}(Z')$ intersects it in one point.

It means that the image of the relative Abel map $c^{-E_{u'}^* - E_{u''}^*, \calL_{gen}}(Z')$ is $1$-codimensional in the space $r'^{-1}(\calL_{gen})$.

This means in particular that the relative Abel map is not dominant, which means by Theorem\ref{th:dominantrel} that we have:

\begin{equation*}
\chi(E_{u'}^* + E_{u''}^*) - h^1( Z' - E_{w'}, \calL_{gen}) \geq \min_{0 < B \leq Z'} \left( \chi(E_{u'}^* + E_{u''}^* + B ) -  h^1(\calO_{\min( Z'- B, Z' - E_{w'})}(-B) \otimes  \calL_{gen}) \right).
\end{equation*}

On the other hand by part 2) of Theorem\ref{th:hegy2rel} we get that indeed we have equality:

\begin{equation*}
\chi(E_{u'}^* + E_{u''}^*) - h^1( Z' - E_{w'}, \calL_{gen}) = \min_{0 < B \leq Z'} \left( \chi(E_{u'}^* + E_{u''}^* + B ) -  h^1(\calO_{\min( Z'- B, Z' - E_{w'})}(-B) \otimes  \calL_{gen}) \right).
\end{equation*}

Substituting $B = A$ we get:

\begin{equation*}
\chi(E_{u'}^* + E_{u''}^*) - h^1( Z' - E_{w'}, \calL_{gen}) \leq \chi(E_{u'}^* + E_{u''}^* + A ) -   h^1(\calO_{\min( Z'- A, Z' - E_{w'})}(-A) \otimes  \calL_{gen}) .
\end{equation*}

\begin{equation*}
\chi(E_{u'}^* + E_{u''}^*) - h^1( Z' - E_{w'}, \calL_{gen}) \leq \chi(E_{u'}^* + E_{u''}^* + A ) -  h^1(\min( Z'- A,Z' - E_{w'}- E_{u'}- E_{u''}),  - A) .
\end{equation*}

On the other hand we know that $h^1( Z' - E_{w'}, \calL_{gen}) = h^1(\calO_{Z' - E_{w'}}( Z'' + K - D)) - 1$, which proves our claim in this case completely.
\medskip

\underline{Assume in the following that $A \ngeq E_{u'}$ and $A \ngeq E_{u'}$ and $A \ngeq E_{w'}$:}
\medskip

In this case we have $h^1(\calO_{\min( Z'- A, Z' - E_{w'})}( Z'' + K - A - D))  = h^1(\calO_{Z'- A- E_{w'}}(Z'' + K - A - D))$.
Notice that by Seere duality we have $h^1(\calO_{Z'- A- E_{w'}}(Z'' + K - A - D)) = h^0(\calO_{Z'- A- E_{w'}}( Z' -Z''- E_{w'} + D))$.

Notice that $Z' -Z''- E_{w'} + D$ is an effective divisor (although not $0$-dimensional), whose support doesn't intersect the exceptional divisor $E_{u'}$, which means that the map
$ H^0(\calO_{Z'- A- E_{w'}}(Z' -Z''- E_{w'} + D)) \to H^0(\calO_{E_{u'}})$ is surjective.

Let us have the following exact sequence:

\begin{equation*}
0 \to    H^0(\calO_{Z'- A- E_{w'} - E_{u'}}( Z' -Z''- E_{w'} - E_{u'} + D))      \to  H^0(\calO_{Z'- A- E_{w'}}(Z' -Z''- E_{w'} + D)) \to H^0(\calO_{E_{u'}}).
\end{equation*}

It means that we have $ h^0(\calO_{Z'- A- E_{w'}}( Z' -Z''- E_{w'} + D)) =h^0(\calO_{Z'- A- E_{w'} - E_{u'}}( Z' -Z''- E_{w'} - E_{u'} + D)) + 1$.

On the other hand again by Seere duality we have $h^0(\calO_{Z'- A- E_{w'} - E_{u'}}( Z' -Z''- E_{w'} - E_{u'} + D))  = h^1(\calO_{Z'- A- E_{w'} - E_{u'}}(Z'' + K - A - D))$.

It means that we get:

$$h^1(\calO_{\min( Z'- A, Z' - E_{w'})}( Z'' + K - A - D)) = h^1(\calO_{Z'- A- E_{w'} - E_{u'} -  E_{u''}}(Z'' + K - A - D)) + 1.$$ 

Using the above identities we should prove:

\begin{equation*}
\chi(E_{u'}^* + E_{u''}^*)  - h^1(\calO_{Z' - E_{w'}}(Z'' + K - D))  <  \chi(E_{u'}^* + E_{u''}^* + A ) - h^1(\calO_{Z'- A- E_{w'} - E_{u'} -  E_{u''}}(Z'' + K - A - D))  - 1.
\end{equation*}

\begin{equation*}
 - h^1(\calO_{Z' - E_{w'}}(Z'' + K - D))  <  \chi(A) - h^1(\calO_{Z'- A- E_{w'} - E_{u'} -  E_{u''}}( - A)) - 1 .
\end{equation*}

Notice that $ h^1(\calO_{Z' - E_{w'}}(Z'' + K - D)) = h^1(Z' - E_{w'}, \calL_h) =  h^1(Z'', \calL_s) = h^1( \calO_{Z'- E_{w'} - E_{u'} -  E_{u''}})  +1$, so we have to prove that:

\begin{equation*}
  -h^1( \calO_{Z'- E_{w'} - E_{u'} -  E_{u''}})  - 1  <  \chi(A) - h^1(\calO_{Z'- A- E_{w'} - E_{u'} -  E_{u''}}( - A)) - 1 .
\end{equation*}

\begin{equation*}
 h^1( \calO_{Z'- E_{w'} - E_{u'} -  E_{u''}} )   >  h^1(\calO_{Z'- A- E_{w'} - E_{u'} -  E_{u''}}(- A )) - \chi(A),
\end{equation*}

which is trivial since $H^0( \calO_{Z'- E_{w'} - E_{u'} -  E_{u''}})_{reg} \neq \emptyset$.

\medskip

\underline{Assume in the following that $A \ngeq E_{u'}$, $A \ngeq E_{u'}$ and $1 \leq A_{w'} \leq 2$:}

\medskip

This is the most subtle case of the proposition, in this case we have $h^1(\calO_{\min( Z'- A, Z' - E_{w'})}(Z'' + K - A - D))  = h^1(\calO_{Z'- A}(Z'' + K - A - D))$.
By Seere duality we get that $h^1(\calO_{Z'- A}(Z'' + K - A - D)) = h^0(\calO_{Z'- A}(Z' -Z''+ D))$.

Let us prove first the following lemma, where we will heavily use the genericity of the analytic structure of $Z$ and $\tX$:

\begin{lemma}\label{propred}
We have the equality $h^1(\calO_{Z'- A}( Z'' + K - A - D)) = h^1(\calO_{Z'- A - E_{w'}}( Z'' + K - A - D))$.
\end{lemma}

\begin{proof}

Let us blow up the exceptional divisor $E_{w'}$ sequentially $t = (Z'- A)_{w'}-1$ times in $N$ different generic points, where $N$ is a large number.

Let us denote the set of the last end vertices we get by $S_n$ and let us blow up these last vertices $M$ times, where $M$ is a large numbe, let us denote the new vertices we get by $S$.

Let the new singularity be $\tX_{new}$, and let us denote the new exceptional divisors by $E_{i, j}$, where $ 1 \leq i \leq N$ and $1 \leq j \leq t$ and $E_s, s \in S$.

Let us denote the cycles $$Z'_{new} := \pi^*(Z') - \sum_{1 \leq i \leq N, 1 \leq j \leq t} j \cdot E_{i, j} - \sum_{s \in S}(t+1) E_s,$$ 
$$Z''_{new} := \pi^*(Z'') - \sum_{1 \leq i \leq N, 1 \leq j \leq t} j \cdot E_{i, j} - \sum_{s \in S}(t+1) E_s.$$

\medskip

We know that $h^1(\calO_{Z'- A}( Z'' + K - A - D)) = h^1(\calO_{Z'_{new}-  \pi^*(A)}(\pi^*(Z'' + K - A - D)))$.

Let us denote the vertex set $\calv_{new} \setminus S = \calv_m$ and subsingularity of $\tX_{new}$ supported by the vertex set $\calv_m$ by $\tX_m$ and the restriction of the cycle $Z'_{new}-  \pi^*(A)$ by $(Z'_{new}-  \pi^*(A))_m$.

Similarly let us denote the vertex set $\calv_{new} \setminus (S \cup S_n) = \calv_u$ and the corresponding singularity by $\tX_u$ and the restriction of the cycle $Z'_{new}-  \pi^*(A)$ by $(Z'_{new}-  \pi^*(A))_u$.

Let us fix the analytic type of $\tX_m$ and let's change the analytic type of the resolution $\tX$  by moving the contact of the tubular neighborhoods of the exceptional divisors $E_s | s \in S$ with their neighbours.

Notice that the differential forms in $H^1(\calO_{Z'-A})^*$ haven't got a pole on the exceptional divisors $E_s | s \in S$, and have got poles on the exceptional divisors $E_w, w \in S_N$ of order at most $1$, since we have blown up the ecxceptional divisor $t = (Z'- A)_{w'}$ times and at each blow up the order of pole of a differential form decreases by at least $1$ (See \cite{NNII}).

Since $H^0(\calO_{Z'_{new}-  \pi^*(A)}( \pi^*(Z'' + K - A - D)))_{reg} \neq \emptyset$ this means in particular that:

$$ h^1(\calO_{Z'_{new}-  \pi^*(A)}( \pi^*(Z'' + K - A - D))) = h^1(\calO_{(Z'_{new}-  \pi^*(A))_m}(\pi^*(Z'' + K - A - D))). $$

Notice that since $t \geq 1$ and the cohomological cycle of the cycle $Z'' - E_{u'}$ has got coeficcient $1$ on the exceptional divisor $E_{w'}$ we have $$H^0(\calO_{Z'' - E_{u'}}(Z'' - E_{u'} + K)) \subset H^0(\calO_{(Z'_{new})_u}((Z'_{new})_u + K)).$$

\emph{It means that a line bundle on the cycle $\pi^*( Z'' - E_{u'})$ is determined by its restriction to the cycle $\min ((Z'_{new})_u, \pi^*( Z'' - E_{u'}))$.}

On the other hand notice that while changing the glueing of the tubular neighborhoods of the exceptional divisors $E_s | s \in S$ with their neighbours, the line bundle $\calO_{(Z'_{new})_u}(\pi^*(Z'' + K))$ doesn't change, so we can use the same chosen divisor $D$ for each modified analytic sctructure.

Notice that while moving generically the contact points of the exceptional divisors $E_s | s \in S$, the line bundle $\calO_{(Z'_{new}-  \pi^*(A))_m}( \pi^*(Z'' + K - A - D) )$ becomes a generic line bundle in $r_u^{-1}(\calO_{(Z'_{new}-  \pi^*(A))_u}( \pi^*(Z'' + K - A - D) ))$, where $r_u$ is the restriction map $\pic((Z'_{new}-  \pi^*(A))_m) \to \pic((Z'_{new}-  \pi^*(A))_u)$.

Indeed we have $\calO_{(Z'_{new}-  \pi^*(A))_m}( \pi^*(Z'' + K - A - D) ) = \calO_{(Z'_{new}-  \pi^*(A))_m}(  Z''_{new} + K_m - \pi^*(A) -D)$ and notice that the line 
bundle $\calO_{(Z'_{new}-  \pi^*(A))_m}( K_m  -D)$ doesn't change if we change the contact points the exceptional divisors $E_s | s \in S$.

On the other hand the cycle $ Z''_{new} - \pi^*(A)$ has got nonzero coeficcients along the exceptional divisors $E_s | s \in S$ and our claim follows from $e_{(Z'_{new}-  \pi^*(A))_m}(s_n) = \dim(\im(c^{-M E_{s_n}^*}(   (Z'_{new}-  \pi^*(A))_m ))$ if $s_n \in S_n$ and $M$ is enough large.

Notice that $ H^0(\calO_{Z'_{new}-  \pi^*(A)}(\pi^*(Z'' + K - A - D) ))_{reg} \neq \emptyset$ holds for our original analytic structure.

This property is equivalent to $h^0(\calO_{Z'_{new}-  \pi^*(A)}(\pi^*(Z'' + K - A - D) )) > h^0(\calO_{Z'_{new}-  \pi^*(A) - E_I}(\pi^*(Z'' + K - A - D) - E_I ))$ for every subset
$I \subset |Z'_{new}-  \pi^*(A)|$.

\emph{On the other hand the numbers $h^0(\calO_{Z'_{new}-  \pi^*(A) - E_I}(\pi^*(Z'' + K - A - D) - E_I ))$ change semicontinously under a deformation.}

\medskip

However since the original analytic structure was already generic we can assume (by restricting to a suitable open neighborhood of the Laufer deformations space which is a complement of analytic subvarieties) that the cohomology numbers
$h^0(\calO_{Z'_{new}-  \pi^*(A) - E_I}(\pi^*(Z'' + K - A - D) - E_I ))$ stay the same if we move the contact points of the exceptional divisors $E_s | s \in S$.

It means finally that the generic line bundle in $r_u^{-1}(\calO_{(Z'_{new}-  \pi^*(A))_u}( \pi^*(Z'' + K - A - D) ))$ has got no fixed components so the relative Abel map is dominant.

By Theorem\ref{th:dominantrel} we get:

$$h^1(\calO_{Z'_{new}-  \pi^*(A)}( \pi^*(Z'' + K - A - D))) = h^1(\calO_{(Z'_{new}-  \pi^*(A))_u}(\pi^*(Z'' + K - A - D))).$$

On the other hand we know that $H^1(\calO_{(Z'_{new}-  \pi^*(A))_u})^* \subset H^1(\calO_{Z'- A - E_{w'}})^*$.

Indeed, if a differential form $\omega \in H^1(\calO_{Z'_{new}-  \pi^*(A)})^*$ has got a pole on the exceptional divisor $E_{w'}$ of order $(Z'-A)_{w'}$, then since the number $N$ is very large and we blowed up the vertex $E_{w'}$in generic points, at one of them the differential form $\omega'$ hasn't got an arrow (a cut in its vanishing set, which is not contained in the union of exceptional divisors).

Indeed it can be easily seen that the maximal number of arrows is bounded by $\max_{0 \leq l \leq Z'- A} (-Z_K + l, E_{w'})$.
It means that if $N >  \max_{0 \leq l \leq Z'- A} (-Z_K + l, E_{w'})$, then $\omega'$ hasn't got an arrow at one of the points where we blow up the exceptional divisor $E_{w'}$.

This means that there is a vertex $s_n \in S_n$, such that $\omega$ hasn't got a pole along the exceptional divisor $E_{s_n}$, so $\omega \notin  H^1(\calO_{(Z'_{new}-  \pi^*(A))_u})^*$ and this proves indeed that $H^1(\calO_{(Z'_{new}-  \pi^*(A))_u})^* \subset H^1(\calO_{Z'- A - E_{w'}})^*$.

It means that $h^1(\calO_{(Z'_{new}-  \pi^*(A))_u}(\pi^*(Z'' + K - A - D)))  \leq   h^1(\calO_{Z'- A - E_{w'}}(Z'' + K - A - D))$, which means finally that 
$$h^1(\calO_{Z'- A}(Z'' + K - A - D)) = h^1(\calO_{Z'- A - E_{w'}}(Z'' + K - A - D)).$$

\end{proof}

By Seere duality we get that $$h^1(\calO_{Z'- A- E_{w'}}(Z'' + K - A - D)) = h^0(\calO_{Z'- A- E_{w'}}(Z' -Z''- E_{w'} + D)).$$

Notice that $Z' -Z''- E_{w'} + D$ is an effective divisor (although not $0$-dimensional), which doesn't intersect the exceptional divisor $E_{u'}$, which means that the map
$ H^0(\calO_{Z'- A- E_{w'}}(Z' -Z''- E_{w'} + D)) \to H^0(\calO_{E_{u'}})$ is surjective.

Let's have the following exact sequence:

\begin{equation*}
0 \to    H^0(\calO_{Z'- A- E_{w'} - E_{u'}}( Z' -Z''- E_{w'} - E_{u'} + D))          \to H^0(\calO_{Z'- A- E_{w'}}(Z' -Z''- E_{w'} + D)) \to H^0(\calO_{E_{u'}}).
\end{equation*}

Since the map $ H^0(\calO_{Z'- A- E_{w'}}(Z' -Z''- E_{w'} + D)) \to H^0(\calO_{E_{u'}})$ is surjective we get $$ h^0(\calO_{Z'- A- E_{w'}}(Z' -Z''- E_{w'} + D)) =h^0(\calO_{Z'- A- E_{w'} - E_{u'}}( Z' -Z''- E_{w'} - E_{u'} + D)) + 1.$$

On the other hand again by Seere duality we have $h^0(\calO_{Z'- A- E_{w'} - E_{u'}}(Z' -Z''- E_{w'} - E_{u'} + D))  = h^1(\calO_{Z'- A- E_{w'} - E_{u'}}( Z'' + K - A - D))$, so we get
$$h^1(\calO_{\min( Z'- A, Z' - E_{w'})}( Z'' + K - A - D)) = h^1(\calO_{Z'- A- E_{w'} - E_{u'} -  E_{u''}}(Z'' + K - A - D)) + 1.$$

 This means that we should prove:

\begin{equation*}
\chi(E_{u'}^* + E_{u''}^*)  - h^1(\calO_{Z' - E_{w'}}( Z'' + K - D))  <  \chi(E_{u'}^* + E_{u''}^* + A ) - h^1(\calO_{Z'- A- E_{w'} - E_{u'} -  E_{u''}}(Z'' + K - A - D)) - 1 .
\end{equation*}

However the proof of this is then formally step by step the same as in the previous case, so we finished the proof of this case.

\medskip

\underline{Assume finally that $A \ngeq E_{u'}$ and $A \ngeq E_{u'}$ and $A_{w'} \geq 3$:}

\medskip

In this case we have $h^1(\calO_{\min( Z'- A, Z' - E_{w'})}(Z'' + K - A - D))  = h^1(\calO_{Z'- A}(Z'' + K - A - D))$, however $(Z'-A)_{w'} \leq 1$, which means:
$$h^1(\calO_{\min( Z'- A, Z' - E_{w'})}(Z'' + K - A - D))  = h^1(\calO_{Z'- A -  E_{u'} -  E_{u''}}( Z'' + K - A - D)) = h^1(\calO_{Z'- A -  E_{u'} -  E_{u''}}(- A)).$$

It means that we have to show:

\begin{equation*}
\chi(E_{u'}^* + E_{u''}^*) - h^1(\calO_{Z' - E_{w'}}( Z'' + K - D)) <  \chi(E_{u'}^* + E_{u''}^* + A) - h^1(\calO_{Z'- A -  E_{u'} -  E_{u''}}( - A)).
\end{equation*}

Let us have a generic line bundle $\calL_{gen} \in  \im(c^{-E_{u'}^*  - E_{u''}^*}(Z' - E_{w'}))$ on the image of the Abel map,in particular we have $H^0(Z' - E_{w'}, \calL_{gen})_{reg} \neq \emptyset$.

Consider the line bundles in the inverse image $r'^{-1}(\calL_{gen})$, where $r' : \pic^{-E_{u'}^*  - E_{u''}^*}(Z') \to \pic^{-E_{u'}^*  - E_{u''}^*}(Z' - E_{w'})$ is the
restriction map.  

We know that the space $r'^{-1}(\calL_{gen})$ is $1$ dimensional and the image of the Abel map $c^{-E_{u'}^*  - E_{u''}^*}(Z')$ intersects it in one point.

It means that the image of the relative Abel map $c^{-E_{u'}^*  - E_{u''}^*, \calL_{gen}}(Z')$ is $1$-codimensional in it, so we get exactly the same way as in the first case:

\begin{equation*}
\chi(E_{u'}^* + E_{u''}^*) - h^1( Z' - E_{w'}, \calL_{gen}) = \min_{0 < B \leq Z'} \left( \chi(E_{u'}^* + E_{u''}^* + B ) -  h^1(\calO_{\min( Z'- B, Z' - E_{w'})}(-B) \otimes \calL_{gen}) \right).
\end{equation*}

Substituting $B = A$ we get:

\begin{equation*}
\chi(E_{u'}^* + E_{u''}^*) - h^1( Z' - E_{w'}, \calL_{gen}) \leq \chi(E_{u'}^* + E_{u''}^* + A ) -  h^1(\calO_{\min( Z'- A, Z' - E_{w'})}(-A) \otimes \calL_{gen}) .
\end{equation*}

\begin{equation*}
\chi(E_{u'}^* + E_{u''}^*) - h^1( Z' - E_{w'}, \calL_{gen}) \leq \chi(E_{u'}^* + E_{u''}^* + A ) -  h^1(\calO_{Z'- A -  E_{u'} -  E_{u''}}( - A)) .
\end{equation*}

Now we know that $\chi(E_{u'}^* + E_{u''}^*) - h^1( Z' - E_{w'}, \calL_{gen}) = \chi(E_{u'}^* + E_{u''}^*) - h^1(\calO_{Z' - E_{w'}}(Z'' + K - D)) + 1$, which proves our statement also in this case.

\end{proof}

\end{proof}

\section{Hyperelliptic invoultions on generic singularities, case of  $1$ vertex}

In the following we will prove the analouge theorem in the case when we consider one vertex except of two, the proof will be quite similar, although a bit more technical in a few steps:

\begin{theorem}\label{onevertex}
Let us have a rational homology sphere resolution graph $\mathcal{T}$ and a generic resolution $\tX$ corresponding to it, and let us have an effective integer cycle $Z \geq E$, such that 
$H^0(\calO_Z(K+Z))_{reg} \neq \emptyset$ and a vertex $u \in \calv$, such that $Z_{u} = 1$.

Assume furthemore that $e_Z(u) \geq 3$, with these conditions for every line bundle $\calL \in \im(c^{-2E_{u}^*}(Z))$ one has $h^0(Z, \calL) = 1$.
\end{theorem}

\begin{proof}

Notice that since $H^0(\calO_Z(K+Z))_{reg} \neq \emptyset$, we have $\chi(Z') > \chi(Z)$ for every cycle $0 \leq Z' < Z$, $h^0(\calO_Z(K+Z)) = h^1(\calO_Z) = 1 - \chi(Z)$ and $h^0(\calO_Z) = 1$.

Assume in the following to the contrary that there exists a line bundle $\calL \in \im(c^{-2E_{u}^*}(Z))$ with $h^0(Z, \calL) = 2$.

This means that there is a function on an open subset $U \subset E_{u}$, $f: U \to E_{u}$, such that $\calL = \calO_Z(x + f(x))$.

Notice that the map $f$ is dominant and injective by Lemma\ref{injectiveabel} and we have $f^2 = Id$.

It shows that $f$ is biholomorphic on the open set $U$ by choosing the open subset $U$ properly.

Since the map $c^{-E_u^*}(Z)$ is injective by Lemma\ref{injectiveabel} we also know that the map $f$ is not the identity map.

Note that we can formulate our condition in the equivalent way that the algebraic curve $ \overline{\im(c^{-E_{u}^*}(Z) )} \subset \pic^{-E_{u}^*}(Z)$ has got a symmetry point, which is $\calL / 2$.

Indeed, one direction is trrivial and the other follows from the fact that $\overline{\im(c^{-E_{u}^*}(Z) )}$ is irreducible and there is a local part of it which is symmetric to the point $\calL / 2$,
then by unicity theorem the whole algebraic curve is symmetric to it.

 We know that the curve $\overline{\im(c^{-E_{u}^*}(Z) )}$ is not a line, since $e_Z(u) \geq 3$.
We get similarly as in the proof of the Theorem\ref{twovertices} it can have only one symmetry point, which means that $\calL \in \im(c^{-2E_{u}^*}(Z))$ is the unique line bundle in the image of the Abel map, such that $h^0(Z, \calL) = 2$.

Let us have a generic point $x \in U$, we know that the line bundle $\calO_Z(K+Z)$ hasn't got a base point on the exceptional divisor $E_{u'}$, so there is a section $s \in H^0(\calO_Z(K+Z))_{reg}$, such that $x \in |s|$, we claim that we should have $f(x) \in |s|$ too.

Indeed it follows from the facts that $$h^0(\calO_Z(K+Z - x)) = h^0(\calO_Z(K+Z )) - 1 = h^0(\calO_Z(K+Z - x - f(x))).$$

Consider in the following the complete linear series $|\calO_Z(Z + K)|$ on the cycle $Z$.

We know that the line bundle $\calO_Z(Z + K)$ hasn't got a base point on the exceptional divisor $E_{u}$, so this complete linear series gives a map $\eta: E_{u} \to \bP(H^0(\calO_Z(Z + K))^*)$.

Notice that if $x \in U$ generic, then we have $\eta(x) = \eta(f(x))$.

In the following we claim that for a generic point $x \in U$ the only regular points $y \in U$, such that $\eta (y) = \eta(x)$ are $y = x$ and $y = f(x)$.

Since there are only finitely many points in $E_u \setminus U$ it means that for a generic point $x \in U$ the only points $y \in E_u$, such that $\eta (y) = \eta(x)$ are $y = x$ and $y = f(x)$.

This shows that the map $\eta: E_u \to \eta(E_u)$ is $2$-fold.

Indeed assume that $y \in U$ and $y \neq x, f(x)$ and $\eta(y) = \eta(x)$, this means that $h^0(\calO_Z(K+Z - x)) = h^0(\calO_Z(K+Z - x - y))$, however this means by Seere duality and Riemann-Roch that $ h^0(\calO_Z( x +y)) = 2$.

It means by our remarks above that $\calL = \calO_Z( x +y)$, so $ \calO_Z(y) = \calO_Z(f(x))$, however this is impossible, since the map $c^{-E_u^*}(Z)$ is injective by Lemma\ref{injectiveabel}.

Let us denote the neighbours of the vertex $u$ by $u_{n_1}, u_{n_2}, \cdots, u_{n_k}$, we got that the map $\eta: E_u \to \eta(E_u)$ is a 2-fold branch cover.

Notice that $e_Z(u) = h^1(\calO_Z) - h^1(\calO_{Z- E_u}) \leq (1 - \chi(Z)) - (1 - \chi(Z- E_u)) = (Z-E_u, E_u) -1$.

This means that $(Z-E_u, E_u) = \sum_{1 \leq j \leq k} Z_{u_j} \geq 4$, which means by $(Z_K - E_u, E_u ) = -2$ that $(Z - Z_K, E_u) \geq 2$.

\begin{lemma}\label{diffformscuts}

1) Let us have a generic differential form $\omega \in H^0(\calO_Z(Z+K))_{reg}$, and let its divisor be $\sum_{1 \leq i \leq t_{u}} p_i$ on the exceptional divisor $E_{u}$, we can assume that 
$p_i \in U, 1 \leq i \leq t_{u}$ if $\omega$ is generic enough.

Then we have $2 | t_{u}$ and with some reindexing we have $\calO_Z(p_i + p_{i + \frac{t_u}{2}}) = \calL$ for all $1 \leq i \leq t_{u}$, in particular $p_{i + \frac{t_u}{2}} = f(p_i)$.

2)  Let's have a cycle $E_u \leq Z' \leq Z$, such that $H^0(\calO_{Z'}(K + Z'))_{reg} \neq \emptyset$, then we have $2 | (Z' - Z_K, E_u)$.

3) If $r_1, \cdots, r_{\frac{t_u}{2}}$ are generic points on the exceptional divisor $E_{u}$, then we have $H^0(\calO_Z(Z+K - \sum_{1 \leq i \leq \frac{t_u}{2}} r_i))_{reg} \neq \emptyset$.

Furthermore let us have the restriction map
\medskip

 $$\alpha: H^0 \left(\calO_Z \left(Z+K - \sum_{1 \leq i \leq \frac{t_u}{2}   } r_i \right) \right) \to H^0 \left(\calO_{E_{u}} \left(Z+K - \sum_{1 \leq i \leq \frac{t_u}{2} } r_i \right) \right).$$

\medskip

Then we have $\dim(\im(\alpha)) = 1$.

4) The map $c^{- \frac{t_u}{2} \cdot E_{u}^*}(Z) : \eca^{- \frac{t_u}{2} \cdot E_{u}^*}(Z) \to \pic^{- \frac{t_u}{2} \cdot E_{u}^*}(Z)$ is birational to its image, this means in particular that $e_Z(u)  \geq \frac{t_u}{2}$. 

\end{lemma}

\begin{proof}

Fort part 1) notice first that since the line bundle $\calO_Z(Z+K)$ has got no base points on the exceptional divisor $E_{u}$.
This means that for a generic section $\omega \in H^0(\calO_Z(Z+K))_{reg}$, the divisor of $\omega$ consists of $t_{u}$ disjoint points $p_1, \cdots, p_{t_{u}}$ on the exceptional divisor $E_{u}$.

Notice that we have the open subset $ U \subset E_{u}$, such that the map $f: U \to E_{u}$ is a biholomorphism and $f \circ f = Id$ and if $\omega $ is enough generic we can assume that
$p_1, \cdots, p_{t_{u}} \in U$ .

Notice that if $x \in U$ is an arbitrary point and $s \in H^0(\calO_Z(Z+K))_{reg}$ is a section such that $x \in |s|$, then since $H^0(\calO_Z(Z+K - x)) = H^0(\calO_Z(Z+K -x - f(x)))$
we have $f(x) \in |s|$.

This means that the set $f(p_1), \cdots, f(p_{t_u})$ equals the set $p_1, \cdots, p_{t_{u}}$.

On the other hand we know by Lemma\ref{injectiveabel} that there is at most one point $x \in E_u$ such that $f(x) = x$ which means that we can assume that $f(p_i) \neq p_i$ if
$1 \leq i \leq t_{u}$.

This indeed means that $2 | t_{u}$ and with some reindexing we have $p_{i + \frac{t_u}{2}} = f(p_i)$ and also we have $\calO_Z(p_i + p_{i + \frac{t_u}{2}}) =\calO_Z(p_i + f(p_i))  = \calL$.

\medskip

For part 2) notice that if $E_u \leq Z' \leq Z$, such that $H^0(\calO_{Z'}(K + Z'))_{reg} \neq \emptyset$, then we can repeat the same proof as in part 1)  which gives indeed $2 | (Z' - Z_K, E_u)$.

Indeed we have the same situation as for the cycle $Z$ since $H^0(\calO_{Z'}(K + Z'))_{reg} \neq \emptyset$ and the algebraic curve $ \overline{\im(c^{-E_{u}^*}(Z') )} \subset \pic^{-E_{u}^*}(Z')$ has got a symmetry point, which is the projection of the line bundle $\calL / 2$.

Notice also that $\calO_Z(K+Z) = \calO_Z(\sum_{1 \leq i \leq t_{u}}(p_i + p_{i + \frac{t_u}{2}})) = \frac{t_u}{2} \cdot \calL$.

\medskip 

For part 3) assume that $r_1, \cdots, r_{\frac{t_u}{2}} \in U$ are generic points on the exceptional divisor $E_{u}$ and notice that $\calO_Z(K+Z) = \frac{t_u}{2} \cdot \calL = \calO_Z(\sum_{1 \leq i \leq \frac{t_u}{2}}(r_i + f(r_i)))$, this indeed yields $$H^0(\calO_Z(Z+K - \sum_{1 \leq i \leq \frac{t_u}{2}} r_i))_{reg} \neq \emptyset.$$

On the other hand by part 1) the points $f(r_i)$ are base points of the line bundle $\calO_Z(Z+K - \sum_{1 \leq i \leq \frac{t_u}{2}} r_i)$ so if $s \in H^0(\calO_Z(Z+K - \sum_{1 \leq i \leq \frac{t_u}{2}} r_i))_{reg}$, then the divisor of $s$ on the exceptional divsior $E_{u}$ should be $\sum_{1 \leq i \leq \frac{t_u}{2}}f(r_i)$.

This indeed means that $\dim(\im(\alpha)) = 1$.

\medskip

For part 4) assume that the Abel map $c^{-\frac{t_u}{2} \cdot E_{u}^*}(Z) : \eca^{-\frac{t_u}{2} \cdot E_{u}^*}(Z) \to \pic^{-\frac{t_u}{2} \cdot E_{u}^*}(Z)$ is not birational to its image, then it means that 
there are generic points $p_1, p_2, \cdots, p_{\frac{t_u}{2}} \in U$ and a different tupple of generic points $p'_1, p_2, \cdots, p'_{\frac{t_u}{2}} \in U$, such that $\calO_Z(\sum_{1 \leq i \leq \frac{t_u}{2}} p_i) = \calO_Z(\sum_{1 \leq i \leq \frac{t_u}{2}} p'_i)$.

Notice that $\calO_Z(Z + K) = \calO_Z(\sum_{1 \leq i \leq \frac{t_u}{2}} p'_i + \sum_{1 \leq i \leq \frac{t_u}{2}} f(p_i ))$, which contradicts part 1) and 2).

\end{proof}

Notice that $e_Z(u) \leq (1 - \chi(Z)) - (1 - \chi(Z- E_u)) = (Z- E_u, E_u) -1$, so we have $(Z- E_u, E_u) \geq 4$.

On the other hand by the previous lemma we have $e_Z(u)  \geq \frac{t_{u}}{2} = \frac{(Z- E_u, E_u) - 2}{2}$, we claim that in fact the strict inequality $e_Z(u)  > \frac{t_{u}}{2}$ also holds.

Assume to the contrary that $e_Z(u)  = \frac{t_{u}}{2}$, this means that the Abel map $c^{- \frac{t_{u}}{2} \cdot E_{u}^*}(Z) : \eca^{- \frac{t_{u}}{2} \cdot E_{u}^*}(Z) \to \pic^{- \frac{t_{u}}{2} \cdot E_{u}^*}(Z)$ is dominant.

From part 2) of Lemma\ref{diffformscuts} we know that if $p_1, \cdots, p_{\frac{t_{u}}{2}}$ are generic points on the exceptional divisor $E_{u}$, then we have $H^0(\calO_Z( Z+K - \sum_{1 \leq i \leq \frac{t_{u}}{2}} p_i))_{reg} \neq \emptyset$.

On the other hand let's have the trivial line bundle $\calL_{u} =\calO_{Z- E_{u}}$ and the restriction map $r_{u} : \pic^{-  \frac{t_{u}}{2}\cdot E_{u}^*}(Z) \to \pic^{0}(Z-E_{u})$.

We know that the line bundle $ \calO_Z(\sum_{1 \leq i \leq \frac{t_{u}}{2}} p_i)$ is a relatively generic line bundle in $r_{u}^{-1}(\calL_{u})$, since the Abel map $c^{- \frac{t_{u}}{2} \cdot E_{u}^*}(Z) : \eca^{- \frac{t_{u}}{2} \cdot E_{u}^*}(Z) \to \pic^{- \frac{t_{u}}{2}\cdot E_{u}^*}(Z)$ is dominant.

Since $H^0(\calO_Z( Z+K - \sum_{1 \leq i \leq \frac{t_{u}}{2} } p_i ) )_{reg} \neq \emptyset$ we know that $h^0(\calO_Z( Z+K - \sum_{1 \leq i \leq \frac{t_{u}}{2}} p_i ) ) > h^0(\calO_{Z- E_{u}}(Z- E_{u} + K - \sum_{1 \leq i \leq \frac{t_{u}}{2}} p_i))$, which means that $h^1(\calO_Z( \sum_{1 \leq i \leq \frac{t_{u}}{2} }p_i)) > h^1(\calO_{Z- E_{u}})$.

On the other hand from Theorem\ref{th:dominantrel} we know that $h^1(\calO_Z( \sum_{1 \leq i \leq \frac{t_{u}}{2}} p_i)) = h^1(\calO_{Z- E_{u}}(\sum_{1 \leq i \leq \frac{t_{u}}{2}} p_i)) =
h^1(\calO_{Z- E_{u}})$, which is a contradiction.

It means that we can assume in the following that $e_Z(u) \geq \frac{t_{u}}{2} + 1 = \frac{(Z- E_u, E_u)}{2}$.

We prove the next lemma in the following:

\begin{lemma}\label{cohom}
Suppose, that $\tX$ is a generic singularity with resolution graph $\mathcal{T}$, a vertex $u \in \calv$ with neighbour vertices $u_{n_1}, \cdots, u_{n_k}$ and $Z'$ is an effective integer cycle on it, such that $Z'_u = 1$, $|Z'|$ is connected, $H^0(\calO_{Z'}(Z' + K))_{reg} \neq \emptyset$, and $\sum_{1 \leq j \leq k} Z'_{u_{n_j}} \geq 4$.

Assume furthermore that there is a line bundle $\calL \in \im(c^{-2E_u^*}(Z'))$, such that $h^0(Z', \calL) = 2$.

Let's have a vertex $u_{n_j}$, where $1 \leq j \leq k$, then we have:

1) If the cohomological cycle of the cycle $Z' - E_{u_{n_j}}$ is $Z''$, then $|Z''|$ and $|Z' - Z''|$ is connected and $u \notin |Z' - Z''|$, in particular $u_{n_i} \notin |Z' - Z''|$ if $i \neq j$.

2) We have the equality $e_{Z'}(u) = \frac{(Z'- E_u, E_u)}{2}$.

3) We have the divisibility $2 | Z'_{u_{n_j}}$.

4) For the cohomology cycle of the cycle $Z' - E_{u_{n_j}}$, $Z''$ we have $Z''_{u_{n_j}} = Z'_{u_{n_j}} - 2$ and furthermore we have the equalities $h^1(\calO_{Z''}) = h^1(\calO_{Z'}) - 1$ and $e_{Z''}(u) = e_{Z'}(u) - 1$.
\end{lemma}

\begin{proof}

For part 1) notice first that a cohomological cycle is always connected so $|Z''|$ is trivially connected.

Next let us denote the intersection point of the exceptional divisors $E_u, E_{u_{n_j}}$ by $I$ and let us blow up the resolution at the intersection point $I$.
 Let us denote the new resolution by $\tX_{new}$, which is a generic corresponding to the blown up resolution graph $\mathcal{T}_{new}$.

Let us denote the cycle $Z'_{new} = \pi^*(Z') - E_{new}$ and let us have the line bundle $\calO_{Z'_{new} - E_{new}}(Z'_{new} - E_{new} + K_{new})$ on the cycle $Z'_{new} - E_{new}$.

Notice that we have 
$$h^0(\calO_{Z'_{new} - E_{new}}(Z'_{new} - E_{new} + K_{new})) = h^0(\calO_{Z'_{new}}(Z'_{new} + K_{new}))- 1 = h^0(\calO_{Z'}( Z' + K))-1.$$

Let us denote the fixed component cycle of the line bundle $\calO_{Z'_{new} - E_{new}}(Z'_{new} - E_{new} + K_{new})$ by $0 \leq A \leq Z'_{new}- E_{new}$.

This means that $H^0(\calO_{Z'_{new} - E_{new}}(Z'_{new} - E_{new} + K_{new})) = H^0(\calO_{Z'_{new} - E_{new}- A}(Z'_{new} - E_{new}- A + K_{new} ))$ and $H^0(\calO_{Z'_{new} - E_{new}- A}(Z'_{new} - E_{new}- A + K_{new} ))_{reg} \neq \emptyset$.

We can write $A = \pi^*(A') + t \cdot E_{new}$, notice that $A \ngeq  E_{u}$, because otherwise we would have $\eta(I) = \eta(p)$ for every $p \in E_{u}$, which is impossible since $\eta(E_{u})$ is a $1$-dimensional curve.

We claim that $A \geq E_{u_{n_j}}$, indeed assume to the contrary that $A \ngeq E_{u_{n_j}}$, then $I$ is a base point of the line bundle $\calO_{Z'- A'}(Z' - A' + K)$ and $H^0(\calO_{Z'- A'}(Z' - A' + K))_{reg}$, which is impossible by Lemma\ref{baseI}.

Notice that $A'$ is the minimal cycle, such that $H^0(\calO_{Z'- A'}(K + Z' - A' )) = H^0(Z' -  E_{u_{n_j}}, Z' -  E_{u_{n_j}} + K )$ and $H^0(\calO_{Z'- A'}( Z' - A'  + K))_{reg} \neq \emptyset$, which means that $Z'' = Z'- A'$ is the cohomological cycle of the cycle $Z' - E_{u_{n_j}}$.

\underline{In the following we prove that $|A'|$ is connected, or which is euqivalent $|A|$ is connected:}
\medskip

 Let us denote the component of $A$ in the component of $\calv_{new} \setminus u$ containing $u_{n_j}$ by $A_1$ and $A- A_1 = A_2$, we have:

$$(\chi(Z'_{new}- A_1 - E_{new}) - \chi(Z')) + (\chi(Z'_{new}- A_2) - \chi(Z)) = (\chi(Z'_{new}- A - E_{new}) - \chi(Z')).$$

 It means that we get the following:

\begin{multline*}
(\chi(Z'_{new} - A_1 - E_{new}) - \chi(Z')) + (\chi(Z'- A_2) - \chi(Z'))  = \\  h^0(\calO_{Z'}(Z'+K)) -  h^0(\calO_{Z'_{new}- t E_{new}-A}(Z'_{new}  -  E_{new} - A + K_{new})).
\end{multline*}

\begin{equation*}
(\chi(Z'_{new} - A_1 -  E_{new}) - \chi(Z')) + (\chi(Z'- A_2) - \chi(Z'))  = 1.
\end{equation*}

Notice that we have the inequalities:
$$h^0(\calO_{Z'}( Z' + K)) -  h^0(\calO_{Z'_{new}-   E_{new}-A_1}( Z'_{new}  - E_{new} - A_1 + K_{new})) \leq \chi(Z'_{new} - A_1 - E_{new}) - \chi(Z').$$
$$h^0(\calO_{Z'}(Z' + K)) -  h^0(\calO_{Z'- A_2}( Z' - A_2 + K)) \leq \chi(Z'- A_2) - \chi(Z').$$

\medskip

It means that we get:

\begin{multline*}
( h^0(\calO_{Z'}( Z' + K)) -  h^0(\calO_{Z'_{new}-   E_{new}-A_1}( Z'_{new}  -  E_{new} - A_1 + K_{new}))) + \\ ( h^0(\calO_{Z'}(Z' + K)) -  h^0(\calO_{Z'- A_2}( Z' - A_2 + K))) \leq 1.
\end{multline*}

On the other hand we know that $ h^0(\calO_{Z'}( Z' + K)) -  h^0(\calO_{Z'_{new}-   E_{new}-A_1}( Z'_{new}  - E_{new} - A_1 + K_{new})) \geq 1$, so we get that $h^0(\calO_{Z'}(Z' + K)) -  h^0(\calO_{Z'- A_2}( Z' - A_2 + K)) = 0$.

However this is impossible, since $H^0(\calO_{Z'}( Z' + K))_{reg} \neq \emptyset$, this contradiction proves part 1).

\bigskip

We prove part 2) by induction on the parameter $(Z'- E_u, E_u)$.

\underline{Assume first that $(Z'- E_u, E_u) = 4$:}

\medskip

In this case we know that $e_Z(u) \geq 2$.

Assume that $u_{n_j}$ is a neighbour vertex, such that $Z'_{u_{n_j}} \geq 1$ and assume again that $A'$ is the minimal cycle, such that $H^0(\calO_{Z'- A'}(K + Z' - A' )) = H^0(Z' -  E_{u_{n_j}}, Z' -  E_{u_{n_j}} + K )$ and $H^0(\calO_{Z'- A'}( Z' - A'  + K))_{reg} \neq \emptyset$.

This means that $Z'' = Z'- A'$ is the cohomological cycle of the cycle $Z' - E_{u_{n_j}}$ and by part 1) we know that $Z''_{u_{n_i}} = Z'_{u_{n_i}}$ for every $i \neq j$.

Notice that $H^0(\calO_{Z''}( Z' - A'  + K))_{reg} \neq \emptyset$ and $Z''_u \geq 1$, which means that $2 | (Z'' - E_u, E_u)$, so $(Z'' - E_u, E_u) \leq 2$.

We know that $e_{Z''}(u) \leq (Z'' - E_u, E_u) -1 \leq 1$ and from $ h^0(  \calO_{Z''}( Z'' + K)  ) = h^0(  \calO_{Z'}( Z' + K)    ) - 1$ we know that $e_{Z'}(u) \leq e_{Z''}(u) + 1 \leq 2$, which indeed proves part 2) in this case.

\underline{Assume on the other hand that $(Z'- E_u, E_u) > 4$:}

\medskip

This means by $2 | (Z'- E_u, E_u)$ that $(Z'- E_u, E_u) \geq 6$.

Assume that $u_{n_j}$ is a neighbour vertex, such that $Z'_{u_{n_j}} \geq 1$ and assume again that $A'$ is the minimal cycle, such that $H^0(\calO_{Z'- A'}(K + Z' - A' )) = H^0(Z' -  E_{u_{n_j}}, Z' -  E_{u_{n_j}} + K )$ and $H^0(\calO_{Z'- A'}( Z' - A'  + K))_{reg} \neq \emptyset$.
This means that $Z'' = Z'- A'$ is the cohomological cycle of the cycle $Z' - E_{u_{n_j}}$, by part 1) we know that $Z''_{u_{n_i}} = Z'_{u_{n_i}}$ for every $i \neq j$.

Notice that $H^0(\calO_{Z''}( Z''  + K))_{reg} \neq \emptyset$ and $(Z'')_u \geq 1$, which means that $2 | (Z'' - E_u, E_u)$, and we get $(Z'' - E_u, E_u) \leq (Z' - E_u, E_u) - 2$.

On the other hand from $ h^0( \calO_{Z''}( Z'' + K)  ) = h^0( \calO_{Z'}( Z' + K)    ) - 1$ we know that $e_{Z'}(u) \leq e_{Z''}(u) + 1$.

By the induction hypothesis we know that $e_{Z''}(u) \leq \frac{(Z''- E_u, E_u)}{2}$, so we get $e_{Z'}(u) \leq \frac{(Z'- E_u, E_u)}{2}$ and we already know that $e_{Z'}(u) \geq \frac{(Z'- E_u, E_u)}{2}$, which proves part 2) completely.

\bigskip

We prove part 3) again by induction on the parameter $(Z'- E_u, E_u)$.

Assume first that $(Z'- E_u, E_u) = 4$ and assume to the contrary that there exists a vertex $u_{n_j}$, such that $Z_{u_{n_j}}$ is odd:

It follows that there exists a vertex $u_{n_j}$, such that $Z_{u_{n_j} } = 1$.

Le $Z''$ be the cohomological cycle of the cycle $Z' - E_{u_{n_j}}$, from part 1) we know that $Z''_{u_{n_i}} = Z'_{u_{n_i}}$ for every $i \neq j$.
Notice that $H^0(\calO_{Z''}( Z''  + K))_{reg} \neq \emptyset$ and $(Z'')_u \geq 1$, but $(Z''- E_u, E_u) = (Z'- E_u, E_u) - 1$ is odd, which is impossible.

Assume in the following that $(Z'- E_u, E_u) > 4$ and assume to the contrary that there exists a vertex $u_{n_j}$, such that $Z_{u_{n_j}}$ is odd:

Let $Z'' $ be the cohomological cycle of the cycle $Z' - E_{u_{n_j}}$, we know again that $Z''_{u_{n_i}} = Z'_{u_{n_i}}$ for every $i \neq j$.

Notice that $2 | (Z'' - E_u, E_u)$, so $(Z'' - E_u, E_u) \leq (Z' - E_u, E_u) - 2$ and we have $e_{Z'}(u) \leq e_{Z''}(u) + 1$.
By part 2) we know that $e_{Z''}(u) = \frac{(Z'' - E_u, E_u)}{2}$ and $e_{Z'}(u) = \frac{(Z' - E_u, E_u)}{2}$, which gives $(Z'' - E_u, E_u) = (Z' - E_u, E_u) - 2$  and $(Z'')_{u_{n_j}} = (Z')_{u_{n_j}} - 2$.

Notice that we get $(Z''- E_u, E_u) \geq 4$, so by the induction hypothesis we get that $(Z'')_{u_{n_j}}$ is even, which indeed shows that $(Z')_{u_{n_j}}$ is even too, and it finishes the proof of part 3) completely.

\bigskip

The first statement of part 4) is already proved in part 3).
We also know that $h^1(\calO_{Z''}) = h^0(\calO_{Z''}(Z'' + K)) = h^0(\calO_{Z'}(Z' + K))-1 =  h^1(\calO_{Z'}) - 1$ from part 1).

The claim $e_{Z''}(u) = e_{Z'}(u) - 1$ follows trivially in the case $(Z'- E_u, E_u) > 4$ from part 2) and from the first statement of part 4). 

On the other hand assume that $(Z'- E_u, E_u) =4$, in this case we have $e_{Z'}(u) = 2$ by part 2) and we have to show that $e_{Z''}(u) = 1$.

Indeed we have $e_{Z''}(u) \geq e_{Z'}(u) - 1 = 1$ and $e_{Z''}(u) \leq (Z'' - E_u, E_u) - 1 = 1$ which proves indeed $e_{Z''}(u) = 1$ and this finishes the proof of part 4) completely.

\end{proof}

Let us return to the proof of our main theorem in the following.

Let us recall that we have an effective integer cycle $Z \geq E$ on a generic resolution $\tX$, such that  $H^0(\calO_Z(K+Z))_{reg} \neq \emptyset$ and we have furthermore $Z_{u} = 1$, $e_Z(u) \geq 3$ and there is a line bundle $\calL \in \im(c^{-2 E_u^*}(Z))$, such that $h^0(Z, \calL) = 2$.

By the Lemma\ref{cohom} we have $t_u = (Z- E_u, E_u) \geq 6$.

\emph{Let us define $\frac{t_u}{2} - 1$ number of connected cycles $Z_1 = Z, \cdots, Z_{\frac{t_u}{2} - 1}$ in the following recursively:}

Assume that $Z_{i-1}$ is already defined and let's have a vertex $u_{n_j} \in |Z_{i-1}|$ and let $Z_i$ be the chomological cycle of the cycle $Z_{i-1} - E_{u_{n_j}}$.

By  Lemma\ref{cohom} we know that $(Z_i)_{u_{n_j}} = (Z_{i-1})_{u_{n_j}} - 2$ and $(Z_i)_{u_{n_l}} = (Z_{i-1})_{u_{n_l}}$ for $l \neq j$, so in particular $(Z_i - E_u, E_u) = (Z_{i-1} - E_u, E_u) - 2$ and we have also $e_{Z_i}(u) = e_{Z_{i-1}}(u) - 1$.

Consider the connected cycle $Z' = Z^{\frac{t_u}{2} - 2}$, we have by Lemma\ref{cohom} $(Z' - E_u, E_u) = 6$, $Z'_{u} = 1$ , $e_{Z'}(u) = 3$, $2 | (Z')_{u_{n_j}}$ for all $0 \leq j \leq k$.

We also know that there is a line bundle $\calL' \in \im(c^{-2 E_u^*}(Z'))$, such that $h^0(Z', \calL') = 2$, indeed $\calL'$ is the restriction of the line bundle $\calL$ under the map
$\pic^{-2E_u^*}(Z) \to \im(c^{-2 E_u^*}(Z'))$.

There are three cases in the following:

\medskip

In the first case $u$ has got one neighbour in $|Z'|$, and by symmetry we can assume that this is $u_{n_1}$ and $(Z')_{u_{n_1}} = 6$.

It means that  in this case $Z' = Z^{\frac{t_u}{2} - 2}$ and let us have the vertex $u_{n_1} \in |Z'|$ and let $Z'' = Z_{\frac{t_u}{2} - 1}$ be the cohomological cycle of the cycle $Z' - E_{u_{n_1}}$.
By Lemma\ref{cohom} we have $Z''_{u_{n_1}} = 4$.

\medskip

In the second case $u$ has got $2$ neighbours in $|Z'|$, and by symmetry we can assume that these are $u_{n_1}, u_{n_2}$ and $(Z')_{u_{n_1}} = 4$ and $(Z')_{u_{n_2}} = 2$.
It means that in this case $Z' = Z^{\frac{t_u}{2} - 2}$ and let us have the vertex $u_{n_2} \in |Z'|$ and let $Z'' = Z_{\frac{t_u}{2} - 1}$ be the cohomological cycle of the cycle $Z' - E_{u_{n_2}}$.
By Lemma\ref{cohom} we have $Z''_{u_{n_1}} = 4$ and $Z''_{u_{n_2}} = 0$.

\medskip

In the third case $u$ has got $3$ neighbours in $|Z'|$, and by symmetry we can assume that these are $u_{n_1}, u_{n_2}, u_{n_3}$ and $(Z')_{u_{n_3}} = (Z')_{u_{n_2}} = (Z')_{u_{n_1}} = 2$.
It means that in this case $Z' = Z^{\frac{t_u}{2} - 2}$ and let us have the vertex $u_{n_3} \in |Z'|$ and let $Z'' = Z_{\frac{t_u}{2} - 1}$ be the cohomological cycle of the cycle $Z' - E_{u_{n_3}}$.
By Lemma\ref{cohom} we have $Z''_{u_{n_1}} = 2$, $Z''_{u_{n_2}} = 2$, $Z''_{u_{n_3}} = 0$.

\medskip

In both cases we have $Z''_{u} = 1$, $e_{Z''}(u)= 2$ and $h^0(\calO_{Z''}(K + Z'')) = h^0(\calO_{Z'}(K + Z')) - 1$ and furthermore $H^0(\calO_{Z''}(K + Z''))_{reg} \neq \emptyset$.

\emph{We claim first that in any case if the cohomological cycle of the cycle $Z'' - E_{u}$ is $C$, then $C_{u_{n_1}} \leq 1$:}

Let us denote the intersection point of the exceptional divisors $E_u, E_{u_{n_1}}$ by $I$, and consider the subspace $V \subset H^0(\calO_{Z''}(Z'' + K))$ consisting of the sections, which vanish at the intersection point $I$.
We know that $H^0(\calO_{Z'' - E_{u}}(Z'' - E_{u}+ K)) \subset V$ and $V = H^0(\calO_{Z'''}(Z'''+ K))$, where $Z'''$ is the cohomology cycle of $Z'' - E_{u_{n_1}}$.

We know that $(Z''')_{u} = 1$, $(Z''')_{u_{n_1}} = 2$ and $e_{Z'''}(u) = 1$ in the first and second case and $(Z''')_{u} = 1$, $(Z''')_{u_{n_1}} = 0$ and $e_{Z'''}(u) = 1$  in the third case by Lemma\ref{cohom}.

We get that $H^0(\calO_{Z'' - E_{u}}(Z'' - E_{u}+ K)) = H^0(\calO_{Z''' - E_u}(Z'''- E_u+ K))$, so in the third case we immediately get that $C_{u_{n_1}} \leq 1$.

In the first and second case let us have the subspace $V' \subset H^0(\calO_{Z'''}(Z''' + K))$ consisting of the sections, which vanish at the intersection point $I$.

We know that $H^0(\calO_{Z''' - E_{u}}(Z''' - E_{u}+ K)) \subset V'$ and $V' = H^0(\calO_{Z''''}(Z''''+ K))$, where $Z''''$ is the cohomology cycle of $Z''' - E_{u_{n_1}}$.

Since $H^0(\calO_{C}(C+ K)) = H^0(\calO_{Z''''}(Z''''+ K))$, it follows indeed that $C_{u_{n_1}} \leq 1$ in these cases too.

\medskip

Since the curve $ \overline{\im(c^{-E_{u}^*}(Z'') )}$ is the projection of the curve $ \overline{\im(c^{-E_{u}^*}(Z') )}$ we know that there is a line bundle $\calL_s \in \pic^{-2E_{u}^* }(Z'')$, such that if $\calL_1$ is a generic line bundle in $ \im(c^{-E_{u}^*}(Z'') )$, then $\calL_2 = \calL_s  \otimes \calL_1^{-1}$ is a generic line bundle in $ \im(c^{-E_{u}^*}(Z'') )$.

Indeed $\calL_s$ is the restriction of the line bundle $\calL'$ under the map $ \pic^{-2E_{u}^* }(Z') \to \pic^{-2E_{u}^* }(Z'')$ and in other words $H^0(Z'', \calL_s)_{reg} \neq \emptyset$ and $h^0(Z'', \calL_s) = 2$.

In the following similarly as in the proof of Theorem\ref{twovertices} we want to identify the line bundle $\calL_s$.

\medskip

\underline{We claim first that $H^0(\calO_{Z''}( Z'' + K) \otimes \calL_s^{-1})_{reg} \neq \emptyset$:}

\medskip

For this we have to show that if there is a cycle $0 \leq A < Z''$ and $H^0(\calO_A(K + A))_{reg} \neq \emptyset$, then $h^1(A, \calL_s) < h^1(Z'', \calL_s)$.

Indeed we have $h^1(Z'', \calL_s) = h^1(\calO_{Z''}) - 1$, and $h^1(A, \calL_s) \leq h^1(\calO_A) - 1$ if $A \geq E_{u}$ and $h^1(A, \calL_s) = h^1(\calO_A)$ if $A \leq Z'' - E_{u}$.

In the first case, we have clearly $h^1(\calO_A) - 1 < h^1(\calO_{Z''}) - 1$ and in the second case we have $h^1(\calO_A) \leq h^1(\calO_{Z'' - E_{u} }) = h^1(\calO_{Z''}) - 2$, which proves our claim completely.

\medskip

Let us have the cycle $Z''_r = Z'' - E_{u}$ and consider the line bundle $\calO_{Z''}( Z'' + K) \otimes \calL_s^{-1}$ and the restriction of this line bundle $\calL_r = \calO_{Z''_r}( Z'' + K)$.
Let us denote the Chern class $l'' = Z'' - Z_K + 2E_{u}^* $.

Consider the subspace $\eca^{ l'', \calL_r}(Z'')  \subset \eca^{l''}(Z'')$ consisting of  the divisors $D \in \eca^{l''}(Z'')$, such that $\calL_r = \calO_{Z''_r}(D)$.

We know that $\eca^{ l'', \calL_r}(Z'')$ is an irreducible smooth algebraic subvariety of $\eca^{l''}(Z'')$ and we have the relative Abel map $c^{ l'', \calL_r}(Z'') : \eca^{ l'', \calL_r}(Z'') \to r^{-1}(\calL_r)$.

Let us prove next the following lemma:

\begin{lemma}\label{relabelconst2}
The relative Abel map $c^{ l'', \calL_r}(Z'') : \eca^{ l'', \calL_r}(Z'') \to r^{-1}(\calL_r)$ is constant and the image is  the line bundle $\calO_{Z''}( Z'' + K) \otimes \calL_s^{-1}$.
\end{lemma}
\begin{proof}

First of all we know that $H^0(\calO_{Z''}( Z'' + K) \otimes \calL_s^{-1})_{reg} \neq \emptyset$.

This means that the line bundle $\calO_{Z''}( Z'' + K) \otimes \calL_s^{-1}$ is in the image of the map $c^{ l'', \calL_r}(Z'')$ and furthermore we have obviously 
$$(c^{ l'', \calL_r}(Z''))^{-1}(\calO_{Z''}( Z'' + K) \otimes \calL_s^{-1}) = (c^{l''}(Z''))^{-1}(\calO_{Z''}( Z'' + K) \otimes \calL_s^{-1}).$$

On the other hand we know that $h^1(\calO_{Z''}( Z'' + K) \otimes \calL_s^{-1}) = h^0(Z'', \calL_s) = 2$.

We claim in the following that $ h^1(\calO_{Z''_r}(Z'' + K)) = 0$, notice that we have $h^1(\calO_{Z''_r}(Z'' + K)) = h^0(\calO_{Z''_r}(- E_{u}))$ by Seere duality.

Let us have the following exact sequence:

\begin{equation*}
0 \to H^0(\calO_{Z''_r}(- E_{u})) \to H^0(\calO_{Z''}) \to H^0(\calO_{E_u}).
\end{equation*}

We know that tha map $H^0(\calO_{Z''}) \to H^0(\calO_{E_u})$ is surjective and $h^0(\calO_{E_u}) = 1$.

It means that we indeed get $h^0(Z'' - E_{u}, - E_{u}) = h^0(\calO_{Z''}) - 1  = 0$  since $H^0(\calO_{Z''}(Z'' + K))_{reg} \neq \emptyset$ and so $h^0(\calO_{Z''}) = 1$ (since the resolution
is generic).

\emph{It means that we have $h^1(\calO_{Z''}( Z'' + K) \otimes \calL_s^{-1}) =2$ and $ h^1(\calO_{Z''_r}(Z'' + K)) = 0$.}

Assume that $c^{ l'', \calL_r}(Z'') $ is nonconstant, this means that $ (c^{l''}(Z''))^{-1}( \calO_{Z''}( Z'' + K) \otimes \calL_s^{-1})$ is a proper smooth algebraic subvariety of $ \eca^{ l'', \calL_r}(Z'') $.

On the other hand notice that by Theorem\ref{relativspace} we have:
$$\dim(\eca^{ l'', \calL_r}(Z'') ) = (l'', Z'')- h^1(\calO_{Z''_r}) + h^1(Z''_r,\calL_r) =  (l'', Z'')- h^1(\calO_{Z''_r}).$$

We also have:
$$\dim((c^{l''}(Z''))^{-1}( \calO_{Z''}( Z'' + K) \otimes \calL_s^{-1})   ) = h^0(\calO_{Z''}( Z'' + K) \otimes \calL_s^{-1})  - h^0(\calO_{Z''}) = h^0(\calO_{Z''}( Z'' + K) \otimes \calL_s^{-1})  - 1.$$

It means that we have:
$$\dim((c^{l''}(Z''))^{-1}( \calO_{Z''}( Z'' + K) \otimes \calL_s^{-1})   ) = h^1( Z'', \calL_s) - 1 = h^1(\calO_{Z''}) - 2.$$

Notice that $ (l'', Z'')- h^1(\calO_{Z''_r})  = (l'', Z'') - h^1(\calO_{Z''}) + 2 = (l'', Z'') + \chi(Z'') + 1$ and $h^1(\calO_{Z''}) - 2 = - \chi(Z'') - 1$ and an easy calculation shows that they
are the same.

This proves that $\dim((c^{l''}(Z''))^{-1}( \calO_{Z''}( Z'' + K) \otimes \calL_s^{-1})   )  = \dim(\eca^{ l'', \calL_r}(Z'') )$, so indeed the map $c^{ l'', \calL_r}(Z'')$ is constant.

\end{proof}

In the follwing let us have a generic section $s \in H^0(\calO_{Z'' - E_{u}}(Z'' + K))_{reg}$ and let us denote $|s| = D$.

We know from the the previous statement that $\calO_{Z''}( Z'' + K) \otimes \calL_s^{-1} = \calO_{Z''}( D)$, which means that $\calL_s = \calO_{Z''}(Z'' + K - D)$.

Let us recal that the curve $\overline{\im(c^{-E_{u}^*}(Z') )}$ has a selfsymmetry.
Indeed we have the line bundle $\calL' \in \pic^{-2E_{u}^* }(Z')$, such that if $\calL_1$ is a generic line bundle in $ \im(c^{-E_{u}^*}(Z') )$, then $\calL_2 = \calL' \otimes \calL_1^{-1}$ is a generic line bundle in $ \im(c^{-E_{u}^*}(Z') )$. 

This is equivalent to $H^0(Z', \calL')_{reg} \neq \emptyset$ and $h^0(Z', \calL') = 2$.

We also know that the restriction of the line bundle $\calL'$ under the restriction map $r' : \pic^{ - 2E_{u}^*}(Z') \to \pic^{ -2 E_{u}^*}(Z'')$ must be $\calL_s$.

It means that $\calL' \in r'^{-1}(\calL_s)$ and notice that $\dim(r'^{-1}(\calL_s)) = 1$ since $h^1(\calO_{Z'}) - h^1(\calO_{Z''}) = 1$.

We know that $h^0(Z'', \calL_s) = 2$ and there is a function $g: U' \to E_{u}$, where $U' \subset E_{u}$, such that $\calL_s = \calO_{Z''}(p + g(p))$ for $p \in U'$.

It follows that $\calO_{Z'}(p + g(p)) = \calL'$ for generic $p \in E_{u'}$, which means that the line bundle and Chern class $(- 2E_{u}^*, \calL_s)$ are not relatively generic on the cycle $Z'$.

The next proposition will give the desired contradiction:

\begin{proposition}\label{relgencont2}
The line bundle and Chern class $( -2E_{u}^*, \calL_s)$ are relatively generic on the cycle $Z'$.
\end{proposition}
\begin{proof}

Assume first that we are in the first case, which means that $u$ has got $1$ neighbour in $|Z'|$ and $(Z')_{u_{n_1}} = 6$.

We also know that $Z''_{u_{n_1}} = 4$ and $H^1(\calO_{ Z' - E_{u_{n_1}}}) = H^1(\calO_{Z''})$, and so there is a unique line bundle in $ \pic^{-2 E_{u}^*}(Z' - E_{u_{n_1}})$, which projects to $\calL_s$, let us denote it by $\calL_{h}$.

For the proof of our theorem we are enough to prove in the following that the line bundle and Chern class $( -2E_{u}^*, \calL_h)$ are relatively generic on the cycle $Z'$.

By Theorem\ref{th:dominantrel} we have to show that:

\begin{equation*}
\chi(2 E_{u}^* ) - h^1(\calO_{Z' - E_{u_{n_1}}}(Z'' + K - D)) < \chi(2 E_{u}^* + A ) -  h^1(\calO_{\min( Z'- A, Z' - E_{u_{n_1}})}( Z'' + K - A - D)) ,
\end{equation*}
in case we have $0 < A \leq Z'$ and $H^0(\calO_{\min( Z'- A, Z' - E_{u_{n_1}})}( Z'' + K - A - D))_{reg} \neq \emptyset$.

We know that $ h^1(\calO_{Z' - E_{u_{n_1}}}(Z'' + K - D))  = h^1(\calO_{Z''}(Z'' + K - D)) = h^1(\calO_{Z''}) - 1 = -\chi(Z'')$, so we need to deal with the cohomology numbers $h^1(\calO_{\min( Z'- A, Z' - E_{u_{n_1}})}( Z'' + K - A - D))$ in the following.

\medskip

\underline{Asumme first that $A \geq E_{u}$:}

\medskip

In this case we know that $h^1(\calO_{\min( Z'- A, Z' - E_{u_{n_1}})}( Z'' + K - A - D)) = h^1(\calO_{\min( Z'- A, Z''_r)}( Z'' + K - A - D))$,
since $H^0(\calO_{Z'- A}(Z'- A + K)) \subset H^0(\calO_{Z''_r}(Z''_r+ K))$.

Since $\calO_{Z''_r}(Z'' + K - A - D) = \calO_{Z''_r}(- A)$  we get that 
$$h^1(\calO_{\min( Z'- A, Z' - E_{u_{n_1}})}(Z'' + K - A - D)) = h^1(\calO_{\min( Z'- A, Z' - E_{u_{n_1}}- E_{u})}(  - A) ).$$
 
It means that we have to prove:

\begin{equation*}
\chi(2E_{u}^*) - h^1( \calO_{Z' - E_{w'}}(Z'' + K - D)) <  \chi(2E_{u}^* + A ) - h^1(\calO_{\min( Z'- A, Z' - E_{u_{n_1}}- E_{u})}(  - A) ).
\end{equation*}

Let us have a generic line bundle $\calL_{gen} \in \im(c^{-2E_{u}^* }(Z' - E_{u_{n_1}}))$ on the image of the Abel map, it means in particular that $H^0(Z' - E_{u_{n_1}}, \calL_{gen})_{reg} \neq \emptyset$ and also we have $h^0(Z' - E_{u_{n_1}}, \calL_{gen})= 1$.

Consider the line bundles in the inverse image $r_{Z' - E_{u_{n_1}}}^{-1}(\calL_{gen})$, where $r_{Z' - E_{u_{n_1}}}: \pic^{-2E_{u}^*}(Z') \to \pic^{-2E_{u}^*}(Z' - E_{u_{n_1}})$ is the restriction map. 

We know that the space $r_{Z' - E_{u_{n_1}}}^{-1}(\calL_{gen})$ is $1$ dimensional and the image of the Abel map $c^{-2E_{u}^*}(Z')$ intersects it in one point.

It means that the image of the relative Abel map $c^{-2E_{u}^*, \calL_{gen}}(Z')$ is $1$-codimensional in the space $r_{Z' - E_{u_{n_1}}}^{-1}(\calL_{gen})$.

This means in particular that the relative Abel map is not dominant, which means by Theorem\ref{th:dominantrel} that we have:

\begin{equation*}
\chi(2E_{u}^* ) - h^1( Z' - E_{u_{n_1}}, \calL_{gen}) \geq \min_{0 < B \leq Z'} \left( \chi(2 E_u^* + B ) -  h^1(\calO_{\min( Z'- B, Z' - E_{u_{n_1}})}(-B) \otimes  \calL_{gen}) \right).
\end{equation*}

On the other hand by part 2) of Theorem\ref{th:hegy2rel} we get that indeed we have equality:

\begin{equation*}
\chi(2E_{u}^* ) - h^1( Z' - E_{u_{n_1}}, \calL_{gen}) = \min_{0 < B \leq Z'} \left( \chi(2 E_u^* + B ) -  h^1(\calO_{\min( Z'- B, Z' - E_{u_{n_1}})}(-B) \otimes  \calL_{gen}) \right).
\end{equation*}

If we substitute $B = A$, then we get:

\begin{equation*}
\chi(2 E_u^*) - h^1( Z' - E_{u_{n_1}}, \calL_{gen}) \leq \chi(2 E_u^* + A ) -   h^1(\calO_{\min( Z'- A,Z' - E_{u_{n_1}})}(-A) \otimes  \calL_{gen}) .
\end{equation*}

\begin{equation*}
\chi(2 E_u^*) - h^1( Z' - E_{u_{n_1}}, \calL_{gen}) \leq \chi(2 E_u^* + A ) -  h^1(\calO_{\min( Z'- A,Z' - E_{u_{n_1}}- E_{u})}(- A)) .
\end{equation*}

We know that $\chi(2 E_u^*) - h^1( Z' - E_{u_{n_1}}, \calL_{gen}) = \chi(2 E_u^*)- h^1(\calO_{Z' - E_{u_{n_1}}}( Z'' + K - D)) + 1$, which proves our statement in this case.

\medskip

\underline{Assume in the following that $A \ngeq E_{u}$ and $A \ngeq E_{u_{n_1}}$:}

\medskip

In this case we have $h^1(\calO_{\min( Z'- A, Z' - E_{u_{n_1}})}( Z'' + K - A - D))  = h^1(\calO_{Z'- A- E_{u_{n_1}}}(Z'' + K - A - D))$.

Notice that the cohomology cycle of $Z'- E_{u_{n_1}}$ is $Z''$, which means that the cohomology cycle of $Z'- A- E_{u_{n_1}}$ is at most $Z'- A- 2E_{u_{n_1}}$.

We have furthermore $H^0(\calO_{Z'- A- E_{u_{n_1}}}(Z'' + K - A - D))_{reg} \neq \emptyset$ which yields $$ h^1(\calO_{Z'- A- E_{u_{n_1}}}(Z'' + K - A - D)) =  h^1(\calO_{Z'- A- 2E_{u_{n_1}}}(Z'' + K - A - D)).$$

By Seere duality we get that $h^1(\calO_{Z'- A- 2E_{u_{n_1}}}(Z'' + K - A - D))= h^0(\calO_{Z'- A- 2E_{u_{n_1}}}( Z' -Z''- 2E_{u_{n_1}} + D))$.

Notice that $Z' -Z''- 2E_{u_{n_1}} + D$ is an effective divisor (although not $0$-dimensional), whose support doesn't intersect the exceptional divisor $E_{u}$.

This means that the map $ H^0(\calO_{Z'- A- 2E_{u_{n_1}}}( Z' -Z''- 2E_{u_{n_1}} + D)) \to H^0(\calO_{E_{u}})$ is surjective.

Let us have the following exact sequence:

\begin{equation*}
0 \to  H^0(\calO_{Z'- A- 2E_{u_{n_1}} - E_u}( Z' -Z''- 2E_{u_{n_1}}- E_u + D)) \to         H^0(\calO_{Z'- A- 2E_{u_{n_1}}}( Z' -Z''- 2E_{u_{n_1}} + D)) \to H^0(\calO_{E_{u}}).
\end{equation*}

Since the map $ H^0(\calO_{Z'- A- 2E_{u_{n_1}}}( Z' -Z''- 2E_{u_{n_1}} + D)) \to H^0(\calO_{E_{u}})$ is surjective and $h^0(\calO_{E_{u}}) = 1$ we get that:

$$h^0(\calO_{Z'- A- 2E_{u_{n_1}}}( Z' -Z''- 2E_{u_{n_1}} + D)) = h^0(\calO_{Z'- A- 2E_{u_{n_1}} - E_u}( Z' -Z''- 2E_{u_{n_1}}- E_u + D))+ 1.$$

On the other hand again by Seere duality we have 
$$h^0(\calO_{Z'- A- 2E_{u_{n_1}} - E_u}( Z' -Z''- 2E_{u_{n_1}}- E_u + D)) = h^1(\calO_{Z'- A- 2E_{u_{n_1}} - E_u}(Z'' + K - A - D)).$$

It means that we get $h^1(\calO_{\min( Z'- A, Z' - E_{u_{n_1}})}( Z'' + K - A - D))  = h^1(\calO_{Z'- A- 2E_{u_{n_1}} - E_u}(Z'' + K - A - D)) + 1$ and we should prove:

\begin{equation*}
\chi(2E_{u}^*)  - h^1(\calO_{Z' - E_{u_{n_1}}}(Z'' + K - D))  <  \chi(2E_u^* + A ) -  h^1(\calO_{Z'- A- 2E_{u_{n_1}} - E_u}(Z'' + K - A - D))  - 1 .
\end{equation*}

\begin{equation*}
 - h^1(\calO_{Z' - E_{w'}}(Z'' + K - D))  <  \chi(A) - h^1(\calO_{Z'- A- 2E_{u_{n_1}} - E_u}(- A )) - 1 .
\end{equation*}

Notice that $ h^1(\calO_{Z' - E_{u_{n_1}}}(Z'' + K - D)) = h^1(Z' - E_{u_{n_1}}, \calL_h) =  h^1(Z'', \calL_s) = h^1( \calO_{Z'- 2E_{u_{n_1}} - E_{u}})  +1$, so we have to prove that:

\begin{equation*}
  -h^1( \calO_{Z'- 2E_{u_{n_1}} - E_{u}}) - 1  <  \chi(A) -  h^1(\calO_{Z'- A- 2E_{u_{n_1}} - E_u}(- A )) - 1 .
\end{equation*}

\begin{equation*}
h^1( \calO_{Z'- 2E_{u_{n_1}} - E_{u}})  > h^1(\calO_{Z'- A- 2E_{u_{n_1}} - E_u}(- A )) - \chi(A)
\end{equation*}

On the other hand this is trivial since $H^0( \calO_{Z'- 2E_{u_{n_1}} - E_{u}})_{reg} \neq \emptyset$.

\medskip

\underline{Assume in the following that $A \ngeq E_{u}$, and assume that $1 \leq A_{u_{n_1}} \leq 3$:}

\medskip

This will be the most subtle case of the proof.

In this case we have $h^1(\calO_{\min( Z'- A, Z' - E_{u_{n_1}})}(Z'' + K - A - D))  = h^1(\calO_{Z'- A}(Z'' + K - A - D))$.

This means by Seere duality that $h^1(\calO_{Z'- A}(Z'' + K - A - D)) = h^0(\calO_{Z'- A}(Z' -Z''+ D))$.

Let us prove first the following Lemma, where we will heavily use the genericity of the analytic structure of $Z$ and $\tX$:

\begin{lemma}\label{propred2}
We have the equality $h^1(\calO_{Z'- A}( Z'' + K - A - D)) = h^1(\calO_{Z'- A - E_{u_{n_1}}}( Z'' + K - A - D))$.
\end{lemma}
\begin{proof}

Let us blow up the exceptional divisor $E_{u_{n_1}}$ sequentially $t = (Z'- A)_{u_{n_1}}-1$ times in $N$ different generic points, where $N$ is a large number.

Let us denote the set of last vertices by $S_n$ and let us blow up the last vertices $M$ times, where $M$ is a large number, let us denote the vertices we get by $S$.

Let the new singularity be $\tX_{new}$, and let us denote the new excpetional divisors by $E_{i, j}$, where $ 1 \leq i \leq N$ and $1 \leq j \leq t$ and $E_s, s \in S$.

Let us denote the cycles
$$Z'_{new} = \pi^*(Z') - \sum_{1 \leq i \leq N, 1 \leq j \leq t} j \cdot E_{i, j} - \sum_{s \in S}(t+1) E_s,$$
$$Z''_{new} = \pi^*(Z'') - \sum_{1 \leq i \leq N, 1 \leq j \leq t} j \cdot E_{i, j} - \sum_{s \in S}(t+1) E_s.$$

\medskip

We know that $h^1(\calO_{Z'- A}( Z'' + K - A - D)) = h^1(\calO_{Z'_{new}-  \pi^*(A)}(\pi^*(Z'' + K - A - D)))$.

Let us denote the vertex set $\calv_{new} \setminus S = \calv_m$ and the subsingularity supported on the vertex set $\calv_m$ by $\tX_m$ and the restriction of the cycle $Z'_{new}-  \pi^*(A)$ by $(Z'_{new}-  \pi^*(A))_m$.

Similarly let us denote the vertex set $\calv_{new} \setminus (S \cup S_n) = \calv_u$ and the corresponding resolution by $\tX_u$ and the restriction of the cycle $Z'_{new}-  \pi^*(A)$ by $(Z'_{new}-  \pi^*(A))_u$.

Let us fix the analytic type of $\tX_m$ and let us change the analytic type of the resolution $\tX$  by moving the contact of the tubular neighborhoods of the exceptional divisors $E_s | s \in S$ with their neighbours.

Notice that the differential forms in $H^1(\calO_{Z'-A})^*$ haven't got a pole on the exceptional divisors $E_s | s \in S$, and have got poles on their neighbours of order at most $1$, since
we have blown up the exceptional divisor $E_{u_{n_1}}$ $t$ times and at each blow up the order of pole of a differential form decreases by at least $1$ (See \cite{NNI}).

Since $H^0(\calO_{Z'_{new}-  \pi^*(A)}( \pi^*(Z'' + K - A - D)))_{reg} \neq \emptyset$ this means in particular that:

$$ h^1(\calO_{Z'_{new}-  \pi^*(A)}( \pi^*(Z'' + K - A - D))) = h^1(\calO_{(Z'_{new}-  \pi^*(A))_m}(\pi^*(Z'' + K - A - D))). $$

Notice that since $t \geq 1$  and $C_{u_{n_1}} \leq 1$, where $C$ is the cohomological cycle of the cycle $Z'' - E_{u}$ we have:

$$H^0(\calO_{Z'' - E_{u}}(Z'' - E_{u} + K)) \subset H^0(\calO_{(Z'_{new})_u}((Z'_{new})_u + K)).$$

\emph{It means that a line bundle on the cycle $\pi^*( Z'' - E_{u})$ is determined by its restriction to the cycle $\min ((Z'_{new})_u, \pi^*( Z'' - E_{u}))$.}

On the other hand notice that while changing the glueing of the exceptional divisors $E_s | s \in S$ with their neighbours the line bundle $\calO_{(Z'_{new})_u}(\pi^*(Z'' + K))$ doesn't change, so we can use the same divisor $D$ during the process simultaneously.

Notice that while moving generically the contact points of the exceptional divisors $E_s | s \in S$, the line bundle $\calO_{(Z'_{new}-  \pi^*(A))_m}( \pi^*(Z'' + K - A - D) )$ becomes a generic line bundle in $r_u^{-1}(\calO_{(Z'_{new}-  \pi^*(A))_u}( \pi^*(Z'' + K - A - D) ))$, where $r_u$ is the restriction map $\pic((Z'_{new}-  \pi^*(A))_m) \to \pic((Z'_{new}-  \pi^*(A))_u)$.

Indeed we have $\calO_{(Z'_{new}-  \pi^*(A))_m}( \pi^*(Z'' + K - A - D) ) = \calO_{(Z'_{new}-  \pi^*(A))_m}(  Z''_{new} + K_m - \pi^*(A) -D)$ and notice that the line 
bundle $\calO_{(Z'_{new}-  \pi^*(A))_m}( K_m  -D)$ does not change if we change the contact points the exceptional divisors $E_s | s \in S$.

On the other hand the cycle $ Z''_{new} - \pi^*(A)$ has got nonzero coeficcients along the exceptional divisors $E_s | s \in S$ and our claim follows from $e_{(Z'_{new}-  \pi^*(A))_m}(s_n)
= \dim(\im(c^{-M E_{s_n}^*}(   (Z'_{new}-  \pi^*(A))_m ))$ if $s_n \in S_n$ and $M$ is enough large.

Notice that $ H^0(\calO_{Z'_{new}-  \pi^*(A)}(\pi^*(Z'' + K - A - D) ))_{reg} \neq \emptyset$ holds for our original analytic structure.

This property is equivalent to $h^0(\calO_{Z'_{new}-  \pi^*(A)}(\pi^*(Z'' + K - A - D) )) > h^0(\calO_{Z'_{new}-  \pi^*(A) - E_I}(\pi^*(Z'' + K - A - D) - E_I ))$ for every subset
$I \subset |Z'_{new}-  \pi^*(A)|$.

\emph{On the other hand the numbers $h^0(\calO_{Z'_{new}-  \pi^*(A) - E_I}(\pi^*(Z'' + K - A - D) - E_I ))$ change semicontinously under a deformation.}

\medskip

However since the original analytic structure was already generic we can assume (by restricting to a suitable open neighborhood of the Laufer deformations space which is a complement of analytic subvarieties) that the cohomology numbers
$h^0(\calO_{Z'_{new}-  \pi^*(A) - E_I}(\pi^*(Z'' + K - A - D) - E_I ))$ stay the same if we move the contact points of the exceptional divisors $E_s | s \in S$.

It means finally that the generic line bundle in $r_u^{-1}(\calO_{(Z'_{new}-  \pi^*(A))_u}( \pi^*(Z'' + K - A - D) ))$ has got no fixed components so the relative Abel map is dominant.

By Theorem\ref{th:dominantrel} we get:

$$h^1(\calO_{Z'_{new}-  \pi^*(A)}( \pi^*(Z'' + K - A - D))) = h^1(\calO_{(Z'_{new}-  \pi^*(A))_u}(\pi^*(Z'' + K - A - D))).$$

On the other hand we know that $H^1(\calO_{(Z'_{new}-  \pi^*(A))_u})^* \subset H^1(\calO_{Z'- A - E_{u_{n_1}}})^*$.

Indeed, if a differential form $\omega \in H^1(\calO_{Z'_{new}-  \pi^*(A)})^*$ has got a pole on the exceptional divisor $E_{u_{n_1}}$ of order $(Z'-A)_{u_{n_1}}$, then since the number $N$ is very large and we blowed up the vertex $E_{u_{n_1}}$in generic points, at one of them the differential form $\omega'$ hasn't got an arrow (a cut in its vanishing set, which is not contained in the union of exceptional divisors).

Indeed it can be easily seen that the maximal number of arrows is bounded by $\max_{0 \leq l \leq Z'- A} (-Z_K + l, E_{u_{n_1}})$.
It means that if $N >  \max_{0 \leq l \leq Z'- A} (-Z_K + l, E_{u_{n_1}})$, then $\omega'$ hasn't got an arrow at one of the points where we blow up the exceptional divisor $E_{u_{n_1}}$.

This means that there is a vertex $s_n \in S_n$, such that $\omega$ hasn't got a pole along the exceptional divisor $E_{s_n}$, so $\omega \notin  H^1(\calO_{(Z'_{new}-  \pi^*(A))_u})^*$ and this proves indeed that $H^1(\calO_{(Z'_{new}-  \pi^*(A))_u})^* \subset H^1(\calO_{Z'- A - E_{u_{n_1}}})^*$.

It means that $h^1(\calO_{(Z'_{new}-  \pi^*(A))_u}(\pi^*(Z'' + K - A - D)))  \leq   h^1(\calO_{Z'- A - E_{u_{n_1}}}(Z'' + K - A - D))$, which means finally that 
$$h^1(\calO_{Z'- A}(Z'' + K - A - D)) = h^1(\calO_{Z'- A - E_{u_{n_1}}}(Z'' + K - A - D)).$$

\end{proof}

\begin{remark}
Notice that the proof of this lemma works also in case $A_{u_{n_1}} \leq 4$, we will use that in the proof of the next case.
\end{remark}

Notice that if $A_{u_{n_1}} \leq 3$, then repeating the same proof with $t = (Z'- A)_{u_{n_1}}-2$ gives that 

$$h^1(\calO_{Z'- A}(Z'' + K - A - D)) = h^1(\calO_{Z'- A - 2E_{u_{n_1}}}(Z'' + K - A - D)).$$

It means by Seere duality that we get:

$$h^1(\calO_{Z'- A - 2E_{u_{n_1}}}(Z'' + K - A - D)) = h^0(\calO_{Z'- A- 2E_{u_{n_1}}}(Z' -Z''- 2E_{u_{n_1}} + D)).$$

Notice that $Z' -Z''- 2E_{u_{n_1}} + D$ is an effective divisor (although not $0$-dimensional), which doesn't intersect the exceptional divisor $E_{u}$, which means that the map
$ H^0(\calO_{Z'- A- 2E_{u_{n_1}}}(Z' -Z''- 2E_{u_{n_1}} + D)) \to H^0(\calO_{E_{u}})$ is surjective.

Let us have the following exact sequence:

\begin{equation*}
 0 \to      H^0(\calO_{Z'- A-  2E_{u_{n_1}} - E_{u}}( Z' -Z''- 2 E_{u_{n_1}} - E_{u} + D))                       \to H^0(\calO_{Z'- A- 2E_{u_{n_1}}}(Z' -Z''- 2E_{u_{n_1}} + D)) \to H^0(\calO_{E_{u}}).
\end{equation*}

Since the map $ H^0(\calO_{Z'- A- 2E_{u_{n_1}}}(Z' -Z''- 2E_{u_{n_1}} + D)) \to H^0(\calO_{E_{u}})$ is surjective and $h^0(\calO_{E_{u}}) = 1$ we get that:

\begin{equation*}
h^0(\calO_{Z'- A-  2E_{u_{n_1}}}(Z' -Z''-  2E_{u_{n_1}} + D)) =h^0(\calO_{Z'- A-  2E_{u_{n_1}} - E_{u}}( Z' -Z''- 2 E_{u_{n_1}} - E_{u} + D)) + 1.
\end{equation*}

On the other hand again by Seere duality we have 
$$h^0(\calO_{Z'- A-  2E_{u_{n_1}} - E_{u}}( Z' -Z''-  2E_{u_{n_1}} - E_{u} + D)) = h^1(\calO_{Z'- A-  2E_{u_{n_1}} - E_{u}}( Z'' + K - A - D)),$$

$$h^1(\calO_{\min( Z'- A, Z' - E_{u_{n_1}})}( Z'' + K - A - D)) = h^1(\calO_{Z'- A-  E_{u_{n_1}} - E_{u}}(Z'' + K - A - D)) + 1.$$

It means that we should prove:

\begin{equation*}
\chi(2E_{u}^* )  - h^1(\calO_{Z' - E_{u_{n_1}}}( Z'' + K - D))  <  \chi(2E_{u}^*  + A ) - h^1(\calO_{Z'- A- E_{u_{n_1}} - E_{u}}(Z'' + K - A - D))  - 1 .
\end{equation*}

However the proof of this is then formally step by step the same as in the first case, so we finished the proof in this case.

\medskip

\underline{Assume finally that $A \ngeq E_{u}$ and $A_{u_{n_1}} \geq 4$:}

\medskip

 In this case have $h^1(\calO_{\min( Z'- A, Z' - E_{u_{n_1}})}(Z'' + K - A - D))  = h^1(\calO_{Z'- A}(Z'' + K - A - D))$.

 If $A_{w'} = 4$, then Lemma\ref{propred2} shows that

$$h^1(\calO_{Z'- A}(Z'' + K - A - D)) = h^1(\calO_{Z'- A - E_{u_{n_1}}}(Z'' + K - A - D)).$$

On the other hand we have $(Z'- A - E_{u_{n_1}})_{u_{n_1}} = 1$, which means that $H^1(\calO_{\min( Z'- A, Z' - E_{u_{n_1}})})^* \subset H^1(\calO_{Z'- A -  E_{u}})^*$.

It follows that $h^1(\calO_{\min( Z'- A, Z' - E_{u_{n_1}})}(Z'' + K - A - D))  = h^1(\calO_{Z'- A -  E_{u}}( Z'' + K - A - D)) = h^1(\calO_{Z'- A -  E_{u}}(- A))$, since the line bundle
$\calO_{Z'- A -  E_{u}}( Z'' + K - D)$ is trivial.

On the other hand if $A_{w'} \geq 5$, then we have $(Z'- A)_{u_{n_1}} \leq 1$, which means again that 
$$h^1(\calO_{\min( Z'- A, Z' - E_{u_{n_1}})}(Z'' + K - A - D))  = h^1(\calO_{Z'- A -  E_{u}}( Z'' + K - A - D)) = h^1(\calO_{Z'- A -  E_{u}}(- A)).$$

In both cases we have to prove that:

\begin{equation*}
\chi(2 E_{u}^* ) - h^1(\calO_{Z' - E_{u_{n_1}}}( Z'' + K - D)) <  \chi(2E_{u}^* + A ) - h^1(\calO_{Z'- A -  E_{u}}( - A) ).
\end{equation*}

Let us have a generic line bundle $\calL_{gen} \in \im( c^{-2E_{u}^* }(Z' - E_{u_{n_1}}))$ on the image of the Abel map, it means in particular that $H^0(Z' - E_{u_{n_1}}, \calL_{gen})_{reg} \neq \emptyset$ and also we have $h^0(Z' - E_{u_{n_1}}, \calL_{gen})= 1$.

Consider the line bundles in the inverse image $r'^{-1}(\calL_{gen})$, where $r' : \pic^{-2E_{u}^* }(Z') \to \pic^{-2E_{u}^* }(Z' - E_{u_{n_1}})$ is the restriction map. 

We know that the space $r'^{-1}(\calL_{gen})$ is $1$ dimensional and the image of the Abel map $c^{-2E_{u}^*}(Z')$ intersects it in one point.

It means that the image of the relative Abel map $c^{-2E_{u}^*, \calL_{gen}}(Z')$ is $1$-codimensional in the space $r'^{-1}(\calL_{gen})$.

Similarly as in the first case we get:

\begin{equation*}
\chi(2E_{u}^* ) - h^1( Z' - E_{u_{n_1}}, \calL_{gen}) = \min_{0 < B \leq Z'} \left( \chi(2 E_u^* + B ) -  h^1(\calO_{\min( Z'- B, Z' - E_{u_{n_1}})}(-B) \otimes  \calL_{gen}) \right).
\end{equation*}

If we substitute $B = A$, then we get:

\begin{equation*}
\chi(2E_{u}^* ) - h^1( Z' - E_{u_{n_1}}, \calL_{gen}) \leq \chi(2E_{u}^*  + A ) -  h^1(\calO_{\min( Z'- A, Z' - E_{u_{n_1}})}(-A) \otimes \calL_{gen}) .
\end{equation*}

Since $h^1(\calO_{Z'- A -  E_{u}}( - A)) \leq h^1(\calO_{\min( Z'- A, Z' - E_{u_{n_1}})}(-A) \otimes \calL_{gen})$ it means:

\begin{equation*}
\chi(2E_{u}^*) - h^1( Z' - E_{u_{n_1}}, \calL_{gen}) \leq \chi(2E_{u}^* + A ) -  h^1(\calO_{Z'- A -  E_{u}}( - A)) .
\end{equation*}

We know that $\chi(2E_{u}^*) - h^1( Z' - E_{u_{n_1}}, \calL_{gen}) = \chi(2E_{u}^* ) - h^1(\calO_{Z' - E_{u_{n_1}}}(Z'' + K - D)) + 1$, which proves our statement also in these cases.

\bigskip

The proof of the other two major cases about the shape of the support $|Z'|$ will be much simpler.

\bigskip

\underline{Assume in the following that $u$ has got $2$ neighbours in $|Z'|$ and $(Z')_{u_{n_1}} = 4$ and $(Z')_{u_{n_2}} = 2$:}

\medskip

Let us denote the intersection point of the exceptional divisors $E_u, E_{u_{n_1}}$ by $I$ and the intersection point of the exceptional divisors $E_u, E_{u_{n_2}}$ by $I'$.

\noindent

Let us have the subspace $V_2 \subset H^0(\calO_{Z'}(Z'+ K))$ consisting of the sections which vanish at the intersection point $I'$.

\noindent

We know from Lemma\ref{cohom} that $V_2 = H^0(\calO_{Z''_2}(Z''_2+ K))$, where $Z''_2$ is the cohomological cycle of the cycle $Z' - E_{u_{n_2}} $.

\noindent

We also have $(Z''_2)_{u} = 1, (Z''_2)_{u_{n_1}} = 4, (Z''_2)_{u_{n_2}} = 0$ and $e_{Z''_2}(u) = 2$.

We know that $h^1(\calO_{Z' - E_{u_{n_2}} }) = h^1(\calO_{Z'}) - 1$ and $H^0(\calO_{Z''_2} ( Z''_2 + K  ))_{reg} \neq \emptyset$, which means that 

$$1- \chi(Z''_2) =  h^1(\calO_{Z''_2}) = h^1(\calO_{Z'}) - 1 =  - \chi(Z').$$

Similarly let us have the subspace  $V_1 \subset H^0(\calO_{Z'}(Z'+ K))$ consisting of the sections which vanish at the intersection point $I$.

We know from Lemma\ref{cohom} that $V_1 = H^0(\calO_{Z''_1}(Z''_1+ K))$, where $Z''_1$ is the cohomological cycle of the cycle $Z' - E_{u_{n_1}} $.

We also have $(Z''_1)_{u} = 1, (Z''_1)_{u_{n_1}} = 2, (Z''_1)_{u_{n_2}} = 2$ and $e_{Z''_1}(u) = 2$.

We know that $h^1(\calO_{Z' - E_{u_{n_1}} }) = h^1(\calO_{Z'}) - 1$ and $H^0(\calO_{Z''_1} ( Z''_1 + K  ))_{reg} \neq \emptyset$, which means that 

$$1 - \chi(Z''_{1}) =  h^1(\calO_{Z''_1}) = h^1(\calO_{Z'}) - 1 =   - \chi(Z').$$

Let us have furthermore the subspace $V_{1, 2} \subset H^0(\calO_{Z''_1}(Z''_1 + K))$ consisting of the sections, which vanish at the intersection point $I'$.

We know from Lemma\ref{cohom} that $V_{1, 2} = H^0(\calO_{Z''_{1, 2}}(Z''_{1, 2}+ K))$, where $Z''_{1, 2}$ is the cohomological cycle of the cycle $Z''_1 - E_{u_{n_2}}$ and also the cohomological cycle of the cylce $Z''_2 - E_{u_{n_1}}$.

We know that the cycle $Z''_{1, 2}$ is connected since it is a cohomological cycle.

We also know that $(Z''_{1, 2})_{u} = 1, (Z''_{1, 2})_{u_{n_1}} = 2, (Z''_{1, 2})_{u_{n_2}} = 0$ and $e_{Z''_{1, 2}}(u) = 1$.

We know that $h^1(\calO_{Z''_{1, 2}}) = h^1(\calO_{Z''_1}) - 1$ and $H^0(\calO_{Z''_{1, 2}}( Z''_{1, 2} + K  ))_{reg} \neq \emptyset$, which means that 

$$ 1 - \chi( Z''_{1, 2}) = h^1(\calO_{Z''_{1, 2}}) = h^1(\calO_{Z''_1}) - 1 =  h^1(\calO_{Z'}) - 2 = -1 - \chi(Z').$$

\medskip

Let us denote by $\calL_{s, 1}$ and $\calL_{s, 2}$ the restrictions of $\calL'$ to $\pic^{-2 E_u^*}(Z''_1)$ and $\pic^{-2 E_u^*}(Z''_2)$, in particular $\calL_{s, i}$ denotes the unique line bundle in $\im(c^{-2 E_u^*}(Z''_i))$ such that $h^0(Z''_i, \calL_{s, i}) = 2$.

Consider a generic section $s_2 \in H^0(\calO_{Z''_2 - E_{u}}(Z''_2 + K))_{reg}$ and let us denote $|s_2| = D_2$.

We know exactly the same way as in Lemma\ref{relabelconst2} that $\calO_{Z''_2}( Z''_2+ K) \otimes \calL_{s, 2}^{-1} = \calO_{Z''_2}( D_2)$, which means that
$$\calL_{s, 2} = \calO_{Z''_2}(Z''_2 + K - D_2).$$

Similarly consider a generic section $s_1 \in H^0(\calO_{Z''_1 - E_{u}}(Z''_1 + K))_{reg}$ and let us denote $|s_1| = D_1$.

We know exactly the same way as in Lemma\ref{relabelconst2} that $\calO_{Z''_1}( Z''_1+ K) \otimes \calL_{s, 1}^{-1} = \calO_{Z''_2}( D_1)$ which means that
$$\calL_{s, 1}= \calO_{Z''_1}(Z''_1 + K - D_1).$$

Notice that the sections $s_1, s_2, D_1, D_2$ can be fixed if we fix a generic analytic type on the cycle $Z' - E_u$, which have got two components, say $E_{u_{n_1}} \leq A, E_{n_{u_2}} \leq B$ and we fix the two cuts what $E_u$ gives on the cycles $A, B$.

Notice that we can still move the analytic type of $\tX$ if we move the intersection points $I, I'$ on the exceptional divisor $E_u$ while keeping the datas fixed above.

If $\calL \in \im(c^{-2E_{u}^*}(Z'))$ is the unique line bundle, such that $h^0(Z', \calL) = 2$, then we have $\calL_{s, 1} = \calL | Z''_1$ and $\calL_{s, 2} = \calL | Z''_2$, it follows obviously that 
$$\calL_{s, 1}  | Z''_{1, 2} = \calL_{s, 1}  | Z''_{1, 2}.$$

It means that $\calO_{Z''_{1, 2}}(Z''_1 + K - D_1) = \calO_{Z''_{1, 2}}(Z''_{2} + K - D_2)$, so $\calO_{Z''_{1, 2}}(Z''_1 - D_1) = \calO_{Z''_{1, 2}}(Z''_{2}- D_2)$.

Let's move the intersection point $I'$ on the exceptional divisor $E_u$ while keeping the datas fixed above, then the restriction of divisors $ D_1, D_2$ to $Z''_{1, 2}$ are remaining the same.

This means that we get that the line bundle $\calO_{Z''_{1, 2}}(Z''_1 - Z''_2)$ doesn't change also.

Let us tdenote the restriction of the cycle $Z''_1 - Z''_2$ to he vertex set$|A| \cup u$ by $Z'''$. Notice that we have $(Z''_1 - Z''_2)_{u_{n_2}} = 2$.

Notice that $\calO_{Z''_{1, 2}}(Z''_1 - Z''_2) = \calO_{Z''_{1, 2}}(Z''') \otimes \calO_{Z''_{1, 2}}(2 E_{u_{n_2}})$ and the line bundle $\calO_{Z''_{1, 2}}(Z''')$ remains constant while we move the intersection point $I'$.

On the other hand we know that $e_{Z''_{1, 2}}(u) > 0$, which means that the line bundle $\calO_{Z''_{1, 2}}(2 E_{u_{n_2}})$ changes if we move the intersection point $I'$ generically, this contradiction proves our proposition also in this case.

\bigskip

\underline{Assume in the following that $u$ has got $3$ neighbours in $|Z'|$ and $(Z')_{u_{n_1}} = (Z')_{u_{n_2}} = (Z')_{u_{n_3}} = 2$:}

\medskip

Let us denote the intersection point of the exceptional divisors $E_u, E_{u_{n_1}}$ by $I$ and the intersection point of the exceptional divisors $E_u, E_{u_{n_2}}$ by $I'$.

Let us have the subspace $V_2 \subset H^0(\calO_{Z'}(Z'+ K))$ consisting of the sections which vanish at the intersection point $I'$.

\noindent

We know from Lemma\ref{cohom} that $V_2 = H^0(\calO_{Z''_2}(Z''_2+ K))$, where $Z''_2$ is the cohomological cycle of the cycle $Z' - E_{u_{n_2}} $.

\noindent

We also have $(Z''_2)_{u} = 1, (Z''_2)_{u_{n_1}} = 2, (Z''_2)_{u_{n_2}} = 0$, $(Z''_2)_{u_{n_3}} = 2$ and $e_{Z''_2}(u) = 2$.

We know that $h^1(\calO_{Z' - E_{u_{n_2}} }) = h^1(\calO_{Z'}) - 1$ and $H^0(\calO_{Z''_2} ( Z''_2 + K  ))_{reg} \neq \emptyset$, which means that 

$$1- \chi(Z''_2) =  h^1(\calO_{Z''_2}) = h^1(\calO_{Z'}) - 1 =  - \chi(Z').$$

Similarly let us have the subspace  $V_1 \subset H^0(\calO_{Z'}(Z'+ K))$ consisting of the sections which vanish at the intersection point $I$.

We know from Lemma\ref{cohom} that $V_1 = H^0(\calO_{Z''_1}(Z''_1+ K))$, where $Z''_1$ is the cohomological cycle of the cycle $Z' - E_{u_{n_1}} $.

We also have $(Z''_1)_{u} = 1, (Z''_1)_{u_{n_1}} = 0, (Z''_1)_{u_{n_2}} = 2, (Z''_1)_{u_{n_3}} = 2$ and $e_{Z''_1}(u) = 2$.

We know that $h^1(\calO_{Z' - E_{u_{n_1}} }) = h^1(\calO_{Z'}) - 1$ and $H^0(\calO_{Z''_1} ( Z''_1 + K  ))_{reg} \neq \emptyset$, which means that 

$$1 - \chi(Z''_{1}) =  h^1(\calO_{Z''_1}) = h^1(\calO_{Z'}) - 1 =   - \chi(Z').$$

Let us have furthermore the subspace $V_{1, 2} \subset H^0(\calO_{Z''_1}(Z''_1 + K))$ consisting of the sections, which vanish at the intersection point $I'$.

We know from Lemma\ref{cohom} that $V_{1, 2} = H^0(\calO_{Z''_{1, 2}}(Z''_{1, 2}+ K))$, where $Z''_{1, 2}$ is the cohomological cycle of the cycle $Z''_1 - E_{u_{n_2}}$ and also the cohomological cycle of the cylce $Z''_2 - E_{u_{n_1}}$.

We know that the cycle $Z''_{1, 2}$ is connected since it is a cohomological cycle.

We also know that $(Z''_{1, 2})_{u} = 1, (Z''_{1, 2})_{u_{n_1}} = 0, (Z''_{1, 2})_{u_{n_2}} = 0, (Z''_{1, 2})_{u_{n_2}} = 2$ and $e_{Z''_{1, 2}}(u) = 1$.

We know that $h^1(\calO_{Z''_{1, 2}}) = h^1(\calO_{Z''_1}) - 1$ and $H^0(\calO_{Z''_{1, 2}}( Z''_{1, 2} + K  ))_{reg} \neq \emptyset$, which means that 

$$ 1 - \chi( Z''_{1, 2}) = h^1(\calO_{Z''_{1, 2}}) = h^1(\calO_{Z''_1}) - 1 =  h^1(\calO_{Z'}) - 2 = -1 - \chi(Z').$$

\medskip

Let us denote by $\calL_{s, 1}$ and $\calL_{s, 2}$ the restrictions of $\calL'$ to $\pic^{-2 E_u^*}(Z''_1)$ and $\pic^{-2 E_u^*}(Z''_2)$, in particular $\calL_{s, i}$ denotes the unique line bundle in $\im(c^{-2 E_u^*}(Z''_i))$ such that $h^0(Z''_i, \calL_{s, i}) = 2$.

Consider a generic section $s_2 \in H^0(\calO_{Z''_2 - E_{u}}(Z''_2 + K))_{reg}$ and let us denote $|s_2| = D_2$.

We know exactly the same way as in Lemma\ref{relabelconst2} that $\calO_{Z''_2}( Z''_2+ K) \otimes \calL_{s, 2}^{-1} = \calO_{Z''_2}( D_2)$, which means that
$$\calL_{s, 2} = \calO_{Z''_2}(Z''_2 + K - D_2).$$

Similarly consider a generic section $s_1 \in H^0(\calO_{Z''_1 - E_{u}}(Z''_1 + K))_{reg}$ and let us denote $|s_1| = D_1$.

We know exactly the same way as in Lemma\ref{relabelconst2} that $\calO_{Z''_1}( Z''_1+ K) \otimes \calL_{s, 1}^{-1} = \calO_{Z''_2}( D_1)$ which means that
$$\calL_{s, 1}= \calO_{Z''_1}(Z''_1 + K - D_1).$$

Notice that the sections $s_1, s_2, D_1, D_2$ can be fixed if we fix a generic analytic type on the cycle $Z' - E_u$, which have got three components, say $E_{u_{n_1}} \leq A, E_{n_{u_2}} \leq B, E_{n_{u_3}} \leq C$ and we fix the three cuts what $E_u$ gives on the cycles $A, B, C$.

Notice that we can still move the analytic type of $\tX$ if we move the intersection points $I, I'$ on the exceptional divisor $E_u$ while keeping the datas fixed above.

Notice that we have $\calL_{s, 1} = \calL' | Z''_1$ and $\calL_{s, 2} = \calL' | Z''_2$, it follows obviously that 
$$\calL_{s, 1}  | Z''_{1, 2} = \calL_{s, 1}  | Z''_{1, 2}.$$

It means that $\calO_{Z''_{1, 2}}(Z''_1 + K - D_1) = \calO_{Z''_{1, 2}}(Z''_{2} + K - D_2)$, so $\calO_{Z''_{1, 2}}(Z''_1 - D_1) = \calO_{Z''_{1, 2}}(Z''_{2}- D_2)$.

Let us move the intersection point $I'$ on the exceptional divisor $E_u$ while keeping the datas fixed above and also the intersection pont $I$ and $E_{u} \cap E_{u_{n_3}}$, then the divisors $ D_1, D_2$ are remaining the same,

This means that we get that the line bundle $\calO_{Z''_{1, 2}}(Z''_1 - Z''_2)$ doesn't change also.

Let us denote the restriction of the cycle $Z''_1 - Z''_2$ to $|C| \cup u$ by $Z'''$. Notice that we have $(Z''_1 - Z''_2)_{u_{n_2}} = 2$ and  $(Z''_1 - Z''_2)_{u_{n_1}} = -2$ .

Notice that $\calO_{Z''_{1, 2}}(Z''_1 - Z''_2) = \calO_{Z''_{1, 2}}(Z''') \otimes \calO_{Z''_{1, 2}}(2 E_{u_{n_2}}) \otimes \calO_{Z''_{1, 2}}(-2 E_{u_{n_1}})$ and the line bundle $\calO_{Z''_{1, 2}}(Z''') \otimes \calO_F(-2 E_{u_{n_1}})$ remains constant while we move the intersection point $I'$.

On the other hand we know that $e_{Z''_{1, 2}}(u) > 0$, which means that the line bundle $\calO_F(2 E_{u_{n_2}})$ changes if we move the intersection point $I'$ generically, this contradiction proves our proposition also in this last case.
\end{proof}

\end{proof}

\section{Proof of Theorem\textbf{C}}

In our last Theorem we show that the condition $e_Z(u') \geq 3$ is crucial and our main theorems are not nessecarily true without this condition.

In fact we prove the following statement:

\begin{theorem}\label{counterex}
 Let us have a rational homology sphere resolution graph $\mathcal{T}$ and a generic resolution $\tX$ corresponding to it.
Consider an effective integer cycle $Z \geq E$ and two vertices $u', u'' \in |Z|$ such that 
$H^0(\calO_Z(K+Z))_{reg} \neq \emptyset$, $Z_{u'} = Z_{u''} = 1$ and $0 < e_Z(u') = e_Z(u'') = e_Z(u', u'') \leq 2$.

With these conditions there exists a line bundle $\calL \in \im(c^{-E_{u'}^* - E_{u''}^*}(Z))$, such that $h^0(Z, \calL) = 2$.
\end{theorem}

\begin{proof}

Let us have the complete linear series $|H^0(\calO_Z(Z + K))|$ on the cycle $Z$.
Snce the line bundle $\calO_Z(Z + K)$ hasn't got a base point on the exceptional divisors $E_{u'}, E_{u''}$ this complete linear series gives a map $\eta: E_{u'} \cup E_{u''} \to \bP(H^0(\calO_Z(Z + K))^*)$.

Notice that $H^0(\calO_Z(Z + K))^* \cong H^1(\calO_Z)$ and $V_Z(u'), V_Z(u')$ are linear subspaces of $H^1(\calO_Z)$, which means that
$\bP(V_Z(u')), \bP(V_Z(u'))$ are projective subspaces of $\bP(H^0(\calO_Z(Z + K))^*)$.
We know that $\bP(V_Z(u')) = \bP(V_Z(u'))$ and $\dim(\bP(V_Z(u'))) =  \dim( \bP(V_Z(u'))) \leq 1$ since $e_Z(u') = e_Z(u'') = e_Z(u', u'') \leq 2$.

One can show easily that $\eta(E_{u'}) \subset \bP(V_Z(u'))$ and $\eta(E_{u''}) \subset \bP(V_Z(u''))$.

It follows that there are two generic smooth points $p \in E_{u'}, q \in E_{u''}$ such that $\eta(p) = \eta(q)$.
It means that $$H^0(\calO_Z(Z + K - p)) = H^0(\calO_Z(Z + K - q)) = H^0(\calO_Z(Z + K - p -q)).$$

It yields $h^1(\calO_Z(p + q)) = h^1(\calO_Z(p)) = 1$, so $h^1(\calO_Z(p + q)) =2$.
This proves our statement completely.
\end{proof}

\end{document}